\documentclass{asl}
\usepackage{amsfonts}
\usepackage{amssymb}

\usepackage{enumerate} 

\title{A Strong Intuitionistic Theory of Functionals}
\date{}


\author{Farida Kachapova}
\address{School of Computer and Mathematical Sciences\\
Auckland University of Technology\\
Auckland, New Zealand}

\email{farida.kachapova@aut.ac.nz}

\thanks{ }

\newpage

\newtheorem{theorem}{Theorem}[section]

\newtheorem{lemma}{Lemma}[section]

\theoremstyle{definition}
\newtheorem{definition}{Definition}[section]

\begin{document}

\begin{abstract}
In this paper we construct a Beth model for intuitionistic functionals of high types and use it to create a relatively strong theory $SLP$ containg intuitionistic principles for functionals, in particular, the theory of the "creating subject", axioms for lawless functionals and some versions of choice axioms.

We prove that the intuitionistic theory $SLP$ is equiconsistent with a classical typed set theory $TI$, where the comprehension axiom for sets of type $n$ is restricted to formulas with no parameters of types $>n$. We show that each fragment of $SLP$ with types $\leqslant s$ is equiconsistent with the corresponding fragment of $TI$ and that it is stronger than the previous fragment of $SLP$. Thus, both $SLP$ and $TI$ are much stronger than the second order arithmetic. By constructing the intuitionistic theory $SLP$ and interpreting in it the classical set theory $TI$, we contribute to the program of justifying classical mathematics from the intuitionistic point of view.

\end{abstract}

\maketitle

\section{Introduction}

The program of justifying classical mathematics from the intuitionistic point of view was developed by several authors (see, for example, \cite{myhi68, frie77, bern76, lyub97}). One approach is to construct an axiomatic theory with intuitionistic principles where a classical set theory is interpretable. Myhill \cite{myhi68} showed the consistency of a classical analysis with respect to an intuitionistic analysis with Kripke schema. Bernini \cite{bern76} constructed such an intuitionistic axiomatic theory with functionals of high types and interpreted a classical typed set theory in it. However, his theory turned out to be inconsistent \cite{wend78}. The consistency of its modifications \cite{bern78, wend78} was not proven.

In this paper we construct a consistent modification of Bernini's theory \cite{bern76} in the following steps using reverse mathematics approach. First we consider a basic axiomatic theory $L$ with intuitionistic logic and construct a Beth model $\mathcal{B}_s$ for its fragment $L_s$ with types $\leqslant s$. The model $\mathcal{B}_s$ is the generalisation to high-order functionals of the Beth model for choice sequences from \cite{vand78}.

Then we check what intuitionistic principles for high-order functionals hold in the Beth models $\mathcal{B}_s (s\geqslant 1)$. Combining those together we create a relatively strong intuitionistic theory $SLP$; it includes axioms for lawless functionals, some versions of choice axioms and axioms for the "creating subject". This is reflected in the name: $S$ is for strong, $L$ is for lawless and $P$ is for proofs, i.e. the theory of the "creating subject". The Beth model $\mathcal{B}_s$  is used to prove the consistency of $SLP$. 

Next we consider a sub-theory $TI$ of the classical typed set theory $TT$ and interpret $TI$ in $SLP$. $TI$ has a restricted comprehension axiom, but it is still quite strong; we believe that many parts of classical mathematics can be developed in $TI$ and therefore justified from the intuitionistic point of view, due to the interpretation of $TI$ in $SLP$.   

In section 2 we remind the reader the general definition and basic properties of a Beth model. In section 3 we define the basic theory $L$. In section 4 we construct our Beth model for the fragment $L_s (s\geqslant 1)$. In section 5 we extend $L$ to a theory $LP$ by axioms for the "creating subject" and in section 6 we extend our Beth model to the fragment of $LP$ with types $\leqslant s$. Section 7 is technical and contains lemmas about the Beth model used in further sections. In sections 8-11 we consider different intuitionistic principles in our Beth model. The principles that hold in the model are combined to create the theory $SLP$ in section 12. In the rest of section 12 we interpret $TI$ in $SLP$. 

The original theory of Bernini \cite{bern76} was intended to be a very strong intuitionistic theory. So in the last two sections we investigate the proof-theoretical strengths of our theories and show that $SLP$ is relatively strong, though weaker than the full typed set theory $TT$. In section 13 we construct a truth predicate in the language $TI$ and use it to prove that each fragment of $TI$ with types $\leqslant s$ is stronger than the previous fragment. Similarly, in section 14 we formalise the forcing relation of our Beth model in the language $TI$ and show that $SLP$ is equiconsistent with $TI$ and that each fragment of $SLP$ with types $\leqslant s$ is stronger than the previous fragment. In particular, $SLP$ is stronger than the second-order arithmetic.

In the rest of the Introduction we explain some notations and terminology. 

In this paper we will use the same terminology as in the books by Kolmogorov and Dragalin \cite{kolm82,drag87}; in particular, this relates to the following concepts of a logical-mathematical language and an axiomatic theory.

A \textit{logical-mathematical language of first order} (or \textit{logical-mathematical language,} or just \textit{language} in short) is defined as $ \Omega=\langle Srt, Cnst, Fn, Pr \rangle $, where

$ Srt $ is a non-empty set of sorts of objects, and for each sort $ \pi \in Srt $ there is a countable collection of variables of this sort;

$ Cnst $ is the set of all constants of the language;

$ Fn $ is the set of all functional symbols of the language;

$ Pr $ is the set of all predicate symbols of the language.

In the language $ \Omega $ we can construct terms, atomic formulas and formulas. 

Terms of different sorts are constructed recursively from constants and variables using functional symbols. The complexity of a term $ t $ is the number of occurrences of functional symbols in $ t $.

For each predicate symbol $P\left(x_1,\ldots,x_k \right) $ and terms $ t_1,\ldots,t_k $ of corresponding sorts we construct an atomic formula $P\left(t_1,\ldots,t_k \right) $.

Formulas are constructed from atomic formulas and logical constant $ \bot $ using logical connectives $ \wedge $ (conjunction), $ \vee $ (disjunction) and $ \supset $ (implication), and quantifiers $ \forall $ and $ \exists $. The negation $\neg\varphi$ is an abbreviation for $\varphi\supset\perp$.
The logical connective $ \equiv $ (equivalence) is defined as $\left( \varphi\supset\psi\right) \wedge\left( \psi\supset\varphi\right)  $. The complexity of a formula $ \varphi $ is the number of occurrences of logical symbols (the main three connectives and quantifiers) in $ \varphi $.

Next we define an \textit{axiomatic theory} (or just \textit{theory} in short)  as $ Th = \langle \Omega, l, A \rangle $, where each of the three objects is described as follows.

$ \Omega $ is a logical-mathematical language.

$ l $ is the logic of the theory. We will consider only two well-known logics: the classical logic CPC (classical predicate calculus) and the intuitionistic logic HPC (Heyting's predicate calculus); their detailed description can be found, for example, in \cite{drag87}.

$ A $ is some set of closed formulas of the language $ \Omega $ called the set of non-logical axioms of $ Th $. When axioms are stated as non-closed formulas, it means that they must be closed by universal quantifiers over all parameters. 

The notation $ Th \vdash \varphi $ (formula $ \varphi $ is derivable in the theory $Th$) means that $ \varphi $ is derivable in the logic $ l $ from a finite subset of the axiom set $ A $. The theory $ Th $ is consistent if it is not true that 
$ Th \vdash \bot $.

In this paper we consider axiomatic theories where variables have indices for types. We will often omit an index of a variable when its value is clear; usually it is the same index as in a previous occurrence of the variable.

\section{Definitions}

\subsection{Definition of Beth model}
A \textit{path} in a partially ordered set $b$ is defined as a maximal linearly ordered subset of $b$. A \textit{path through} an element $x$ is defined as a path containing $x$. 

The following algebraic definition of a Beth model is a slight modification of the definition by van Dalen \cite{vand78}.

\begin{definition}
A \textbf{Beth model} $ \mathcal{B} $ for a logical-mathematical language $ \Omega $ is a quadruple of objects $\langle M, \leqslant, D, Val \rangle$ described as follows.

$\langle M, \leqslant \rangle$ is a partially ordered set, called the \textit{domain} of the model; elements of $M$ are denoted $\alpha, \beta, \gamma,\ldots$.
\smallskip

$D$ defines domains for variables of the language: for each sort $ \pi $, $ D_\pi $ is the domain for objects of the sort $ \pi $.

We define an \textit{evaluated formula} as a formula of the language $ \Omega $, in which all parameters are replaced by objects from suitable domains. Another way to describe an evaluated formula is given in \cite{vand78}: we assume that the language $ \Omega $ is extended by constants for all elements of $ D_\pi $ for each sort $ \pi $ ; then an evaluated formula is defined as a closed formula of the extended language.

$Val$ is a \textit{valuation mapping} that assigns $ T $ or $\bot$ to any couple $ (\alpha,\varphi) $, where \\$\alpha\in M$ and $ \varphi $ is an atomic evaluated formula, in such a way that this monotonicity condition holds:
\[(\beta\leqslant\alpha \& Val(\alpha,\varphi)=T)\Rightarrow Val(\beta,\varphi)=T.\]

\end{definition}

Elements of $M$ are interpreted as moments or worlds, the relation $ < $ on $M$ is interpreted as "later in time". 

A topological definition of a Beth model can be found in \cite{drag87}.

\begin{definition}
In the Beth model a \textbf{forcing} relation $\alpha\Vdash\varphi $ is defined for any $\alpha \!\in \!M$  and evaluated formula $\varphi$ by induction on the complexity of $\varphi$.

For an atomic formula $ \varphi$, $\alpha\Vdash\varphi $
iff for any path $S$ in $M$ through $\alpha$, there is
$\beta \in S$ such that $Val(\beta,\varphi)=T$;

\smallskip

$\alpha\not\Vdash \perp$ (falsity is never forced);
\smallskip

$\alpha\Vdash \psi \wedge \eta $ iff $(\alpha\!\Vdash\psi\:\&\:\alpha\!\Vdash \eta)$;
\smallskip      

$\alpha\Vdash\psi\vee\eta$ iff for any path $S$ through $\alpha$,
there is $\beta\!\in\!S$ such that $\beta\Vdash\psi$  or $\beta\Vdash\eta$;
\smallskip

$\alpha\Vdash\psi\supset\eta\ $ iff for any $ \beta\leqslant\alpha,(\beta\Vdash\psi\Rightarrow\beta\Vdash\eta)$;
\smallskip

$\alpha\Vdash\forall x\psi(x)$ iff for any element $c$ of the appropriate domain,
$\alpha\Vdash\psi(c)$;
\smallskip

$\alpha\Vdash\exists x\psi(x)$ iff for any path $ S $ through $\alpha$,
 there is $\beta\!\in$ S and an element
$c$ of the appropriate domain, such that $\beta\Vdash\psi(c).$

\label{def:forcing}
\end{definition}

We will assume that $\langle M, \leqslant \rangle$ is a tree, which is the usual case. 

For the Beth model $ \mathcal{B} $, $\mathcal{B}\Vdash \varphi$ (an evaluated formula $\varphi$ holds in  $\mathcal{B}$, or forced in $\mathcal{B}$) iff  $ \varepsilon \Vdash  \varphi$, where $\varepsilon $ is the root of the tree $M$.

\subsection{Facts about Beth models}
The following lemma and other facts about Beth models are explained in more detail in \cite{vand78}.
\begin{lemma}
A Beth model $ \mathcal{B} $ for a logical-mathematical language $ \Omega $ has the following properties.
\begin{enumerate}
\item Monotonicity: 
$ (\alpha\Vdash\varphi \wedge \beta\leqslant\alpha)\Rightarrow\beta\Vdash\varphi$. 
\medskip

\item $ \alpha\Vdash\varphi\Leftrightarrow $ for any path $ S $ in $ M $ through $ \alpha $ there is $ \beta\in S $ with $ \beta\Vdash\varphi. $
\medskip

\item Soundness Theorem. For any closed formula $ \varphi $ of the language $\Omega$,
\[ (HPC\vdash\varphi)\Rightarrow (\mathcal{B}\Vdash\varphi). \]

\end{enumerate}  \label{lemma:Beth}
\end{lemma}

We use classical logic in metamathematics and in this lemma, in particular.

Van Dalen \cite{vand78} constructed a Beth model for the language of intuitionistic analysis $FIM$, which has two types of variables: over natural numbers and over choice sequences. We generalize this model to a Beth model for a language with many types of functionals.

\section{Axiomatic theories $L$ and $L_s$}

The language of theory $ L $ has the following variables:
\begin{list}{}{}
\item $x,y,z,\ldots$ over natural numbers (variables of type 0),
\item and variables of type $ n$ $(n\geqslant 1):$
\begin{list}{}{}
\item $F^n,G^n,H^n,\ldots$ over $n$-functionals (functionals of type $n$);
\item $A^n,B^n,C^n,\ldots$ over lawlike $n$-functionals;
\item $\mathcal{F}^n,\mathcal{G}^n,\mathcal{H}^n,\ldots$ over lawless $n$-functionals.
\end{list}
\end{list}
\smallskip

Constants: 0 of type 0 and for each $n\geqslant 1$ a constant $K^n$ (an analog of 0 for type $ n $).
\smallskip

Functional symbols: $N^n,Ap^n (n\geqslant 1) $, $S$ for successor, $\cdot$ and $+$.

\smallskip 
Predicate  symbols: $=_n$ for each $n\geqslant 0$.

\medskip
Terms and $n$-functionals are defined recursively as follows.
\begin{enumerate}
\item Every numerical variable is a term.
\item Constant 0 is a term.
\item Every variable of type $ n $ is an $ n $-functional.
\item Constant $ K^n $ is an $ n $-functional.
\item If $t$ is a term, then $St$ is a term.
\item If $t_1$ and $t_2$ are terms, then $t_1+t_2$ and $t_1\cdot t_2$ are terms.
\item If $Z$ is an $n$-functional, then $ N^n(Z) $ is an $n$-functional (a successor of $Z$).
\item If $Z$ is a $1$-functional, $t$ is a term, then $Ap^1(Z,t) $ is a term.
\item If $Z$ is an $(n+1)$-functional, $t$ is a term, then $ Ap^{n+1}(Z,t) $
 is an $n$-functional.
\end{enumerate}

$ Ap^n(Z,t) $ is interpreted as the result of application of functional $ Z $ to term $ t $. We also denote $ Ap^n(Z,t) $ as $Z(t)$.

Here 1-functional is interpreted as a function from natural numbers to natural numbers, and $ (n+1) $-functional is interpreted as a function from natural numbers to $n$-functionals.

As usual, lawlike (or constructive) functionals are interpreted as determined by laws or algorithms.
Lawless functionals are opposite extreme: they could not be determined by any law. At each moment only an initial segment of a lawless functional is known, and there are no restrictions on its future values.
\smallskip

Atomic formulas:
\begin{list}{}{}
\item $ t=_0\tau$, where $t$ and $\tau$ are terms,
\item $ Z=_nV $, where $Z$ and $V$ are  $n$-functionals ($ n\geqslant1 $).
\end{list}

Formulas are constructed from atomic formulas using logical connectives and quantifiers.

For a formula $ \varphi $, its sort $sort(\varphi)$ is the maximal type of parameters in $ \varphi $; it is 0 if $ \varphi $ has no parameters. 

The theory $ L $ has intuitionistic predicate logic $HPC$ with equality axioms and the following non-logical axioms.

\begin{enumerate}
\item $ Sx \neq 0$, \qquad $Sx=Sy\supset x=y. $

\item $x+0=x$, \qquad $x+Sy=S(x+y)$.

\item $x\cdot 0=0$, \qquad $x\cdot Sy=x\cdot y+x$.

The first four axioms define arithmetic at the bottom level. As usual, we introduce terms for all primitive recursive functions in the theory $L$.

\item Induction for natural numbers: $ \varphi(0) \wedge \forall x \left(\varphi(x)\supset \varphi(Sx)\right) \supset \forall x \varphi(x)$, 
\\
where $\varphi$ is any formula of $L$.
\smallskip
\item $ K^{n+1}(x)=K^n$,\qquad $\neg \left( N^n\left( F^n\right)=K^n \right)$.
\smallskip
\item $ N^{n+1}\left(F^{n+1} \right)(x)=N^n\left( F(x)\right)$,\qquad $ N^{n}\left(F^{n} \right) = N^{n}\left(G^{n} \right) \supset F=G. $

The axioms 5 and 6 describe $ K^n $ and $ N^n $ as analogs of zero and successor function, respectively, on level $ n $.

\item Principle of primitive recursive completeness of lawlike functions:
\[ \exists A^1 \forall x \left(A(x)=t \right), \]
where $t$ is any term containing only variables of type 0 and variables over lawlike 1-functionals.
\end{enumerate}

$L$ is a basic theory with $n$-functionals, which is a modification of Bernini's theory $HL$ \cite{bern76}. Later we extend it to a richer theory using reverse mathematics approach.

Denote $L_s$ the fragment of the theory $L$ with types not greater than $ s $. The language of $L_1$ has one type of functionals and is essentially the language of the intuitionistic analysis $ FIM $.

\section{Beth model for language $L_s$}
Any formal proof in the theory $L$ is finite and therefore it is a proof in some fragment $L_s$. So instead of constructing a model for the entire language $L$ with infinitely many types of functionals, we construct a Beth model $\mathcal{B}_s$ for the fragment $L_s$ ($s\geqslant 1$) by generalizing the model of van Dalen \cite{vand78}. 

\subsection{Notations}
A sequence of elements $ x_{0}, x_{1}, \ldots,x_{n} $ is denoted 

\noindent $ x=\langle x_{0}, x_{1}, \ldots,x_{n}\rangle $, and we denote $lth(x)=n+1$. Symbol $ \ast $ denotes the concatenation function: 
\smallskip

$ \langle x_{0}, x_{1}, \ldots,x_{n}\rangle * \langle y_{0}, y_{1}, \ldots,y_{k}\rangle = \langle x_{0}, x_{1}, \ldots,x_{n}, y_{0}, y_{1}, \ldots,y_{k}\rangle.$
\smallskip

For a sequence $ x $, $ \langle x \rangle_{n}$ denotes its $n^{th}$ element and $ \overline{x}(n) $ denotes the initial segment of $ x $ of length $ n $ for any $n<lth(x)$. So if $x=\langle x_{0}, x_{1}, \ldots,x_{m}\rangle$ and $n\leqslant m$, then $\bar{x}(n)=\langle x_{0}, x_{1}, \ldots,x_{n-1}\rangle$.

For a function $ f $ on natural numbers, $ \overline{f}(n) $ denotes the initial segment of $ f $ of length $ n $, that is the sequence  $ \langle f(0), f(1),\ldots,f(n-1) \rangle $.

The notation $ f:a \dashrightarrow b $ means that $ f $ is a partial function from set $ a $ to set $b$. 

For a function $ f $ of two variables denote $ f^{[x]}=\lambda y.f(x,y) $. 

The notation $ Z\downarrow $ means that the object $ Z $ is defined.
\medskip

Next we introduce a few notations for a fixed set $ b $. 
\begin{enumerate}[(a)]
\item $ b^{(n)} $ is the set of all sequences of elements of $ b $ that have length $ n $.

\item $ b^{*} $ is the set of all finite sequences of elements of $ b $.

\item On the set $ b^{*} $ a partial order $ \leqslant $ is defined by the following:
\[ y \leqslant x \textup{  iff  } x \textup{ is an initial segment of }y. \]

\end{enumerate}

With this order $ b^{*} $ is a tree growing down; its root is the empty sequence $< >$.
\smallskip

Suppose $\langle d,\leqslant \rangle$ is a partially ordered set and $ f: (d\times \omega) \dashrightarrow c$. 
 \begin{enumerate}[(a)]
\item $f$ is called \textit{monotonic} on $d$ iff for any $ x, y\in d, (y \leqslant x \Rightarrow f^{[x]} \subseteq f^{[y]})$. 
\medskip
\item $f$ is called \textit{complete} on $d$ iff
for any path $S$ in $d$,   $\bigcup \left\lbrace f^{[x]} \mid x\in S\right\rbrace $ is a total function on $ \omega $.
\end{enumerate}

\subsection{Definition of Beth model $\mathcal{B}_s$}
The following Beth model $\mathcal{B}_s$ is a new, improved version of the Beth-Kripke model that we introduced in \cite{kash89}. 
\begin{definition} We fix $ s\geqslant1 $ and define a Beth model $\mathcal{B}_s$ for the language $L_s$:
\[ \mathcal{B}_s=~\langle M, \preccurlyeq, D, Val \rangle, \] 
where each component is described as follows.

1) First we introduce a triple of objects $\langle a_k,d_k, \preccurlyeq_k \rangle$ by induction on $k$; here $\preccurlyeq_k$ is a partial order on $d_k$.

$a_0=\omega$, the set of all natural numbers.

$d_0=\lbrace \langle x\rangle \mid x\in a_0^*\rbrace$; thus, each element of $d_0$ is a sequence of length 1.

$\preccurlyeq_0$ is generated by the order $ \leqslant $ on $a_0^*$: $\langle x\rangle\preccurlyeq_0 \langle y\rangle$ iff $x\leqslant y$, that is the sequence $y$ is an initial segment of the sequence $x$.

For $\langle x\rangle\in d_0$ we denote $lh(\langle x\rangle)=lth(x)$ and call it the \textit{length} of $\langle x\rangle$.

For $k\geqslant 1$,
\[ a_{k}= \{ f|f:(d_{k-1} \times\omega) \dashrightarrow a_{k-1}, \textit{f is monotonic and complete on } d_{k-1}\}; \]

\[ d_k=\displaystyle\bigcup_{m=0}^\infty \left( {a_0^{(m)}} \times{a_1^{(m)}}\times\ldots \times {a_{k}^{(m)}}\right); \]
\medskip

for $x,y\in d_k$, $x\preccurlyeq_k y$ iff for each $i\leqslant k$, $\langle x\rangle_i\; \leqslant \langle y\rangle_i$, that is the sequence $\langle y\rangle_i$ is an initial segment of the sequence $\langle x\rangle_i$.
\medskip

If $x \in {a_0}^{(m)} \times\ldots \times {a_{k}}^{(m)} $, we denote $ lh(x)=m$ and call it the \textit{length} of $x$. 
\medskip

3) Domain of the Beth model is $M=d_{s-1}$ with partial order $\preccurlyeq_{s-1}$, which we will denote just $\preccurlyeq$. 

With this order $M$ is a subtree of the direct product of trees ${a_0}^* ,\ldots, {a_{s-1}}^* $. $M$ is a non-countable tree growing down. Its root is 
$ \varepsilon = <\underbrace{< >,\ldots,< >}_s>$. 
Elements of $M$ are denoted $\alpha,\beta, \gamma,\ldots $. 

4) Domain for natural numbers is $a_0 =\omega$.

5) Domain for $k$-functionals is $ a_{k}$ $(k=1,2,\ldots,s)$.

6) Domain for lawlike $k$-functionals: 
\[ b_{k}=\{f \in a_{k} | f^{[<>]} \textit{  is a total function} \} . \]

7) Before defining a domain for lawless functionals we introduce an auxiliary set $c_{k} (k\geqslant 1)$: 
\[c_{k}= \{\ \xi |\ \xi :(\omega \times a_{k-1}) \rightarrow a_{k-1} 
\textit{  and  }\forall n (\xi^{[n]} \textit{ is a bijection on } a_{k-1}) \}.\]

Elements of $ c_{k} $ are called $k\textit{-permutations}$. For any $k\textit{-permutation }  \xi$ we define $\nu_k(\xi)$ as the function  $f \in a_{k}$ such that for any $ n\in\omega, $ $x \in d_{k-1} $:  
\begin{displaymath}
f(x,n) =
\begin{cases}
\xi^{[n]}(\langle \langle x\rangle_{k-1}\rangle_n)  & \text{ if } n<lh(x),\\
\text{undefined}& \text{otherwise },
\end{cases}
\end{displaymath} 

Domain for lawless for $k$-functionals is $l_{k}=\{\nu_k(\xi)| \xi\in c_{k} \}, k=1,2,\ldots,s$.
\medskip 

8) It remains to define a valuation mapping $Val$. The definition consists of steps (a)-(d).

(a) For each constant $ Q $ of the language $ L $ we define its interpretation $ \widehat{Q} $.

i) $ \widehat{0}=0 $.

ii) The interpretation of constant $ K^k$  $(k=1,\ldots,s) $ is given by the following:
\smallskip
\\$\widehat{K}^{1} (x,n)=0$ for any $ n\in \omega$ and $ x\in d_0 $;
\medskip\\
$\widehat{K}^{k+1}(x,n)=\widehat{K}^{k}$ for any $ n\in \omega$ and $ x\in d_k $.
\medskip

(b) For any $ \alpha\in M $ and functional symbol $ h $ of the language $ L_{s} $ we define its interpretation $ \widehat{h}^{[\alpha]}$ as follows.
\medskip

i) If $\theta$ is $\cdot$ or +, then $\widehat{\theta}^{[\alpha]}=\theta$.
\medskip

ii) Next we define a successor function $S^{k} $ of type $ k$  $(k=0, 1,\ldots, s) $:\\
$S^{0}(x)=x+1$;   $S^{n+1}(f)=S^{n} \circ f.$
\medskip

Functional symbol $S$ is interpreted by: $ \widehat{S}^{[\alpha]} = S^{0}$.
\medskip

Functional symbol $ N^{k} $ is interpreted by: $ \widehat{N}^{k[\alpha]} = S^{k}$.
\medskip

iii) Functional symbol $ Ap^{k} $ is interpreted by: 
$ \widehat{Ap}^{k[\alpha]}(f,n)=f(\bar{\alpha}(k),n) $ for any $ n\in \omega, f\in a_{k}, k=1, \ldots, s$.

Thus, a $ k$-functional $ f $ depends not on the entire node $ \alpha $ but only on its first $k$ components.
\medskip

We define an \textit{evaluated term} as a term of the language $ L_s $, in which all parameters are replaced by objects from suitable domains. We define an
\textit{evaluated k-functional}  as a $ k $-functional of the language $ L_s $, in which all parameters are replaced by objects from suitable domains. These definitions are similar to the definition of an evaluated formula.

(c) Suppose $ Z $ is an evaluated term or $ n$-functional. We define its interpretation $ Z^{[\alpha]} $ at node $ \alpha $ by induction on the complexity of $ Z. $
\smallskip

i) If $ Z $ is an evaluated variable, then $ Z^{[\alpha]} = Z. $
\smallskip

ii) If $ Z $ is a constant, then $  Z^{[\alpha]} = \widehat{Z}. $
\smallskip

iii) If $ Z=h(V_{1}, \ldots, V_{m}), $ where $ h $ is a functional symbol, then 
\medskip
\\$  Z^{[\alpha]} = \widehat{h}^{[\alpha]}(V_{1}^{[\alpha]},\ldots, V_{m}^{[\alpha]}). $
\medskip

(d) The valuation mapping $ Val $ is defined by:

\[Val(\alpha,Z=_kV)=T \textup{    iff    } 
Z^{[\alpha]}\downarrow \& V^{[\alpha]}\downarrow \& Z^{[\alpha]}=V^{[\alpha]},k=0,1,\ldots,s.\]

This completes the definition of the model 
$\mathcal{B}_s$. \label{def:my_beth_model}
\end{definition}

Clearly, $ b_{k}\subseteq a_{k}$ and $l_{k}\subseteq a_{k}, k=1,\ldots,s$. 
For $ \alpha\in M, lh(\alpha)=lth(\langle\alpha\rangle_k) $ for any $ k<s $.

The main difference with the Beth model in \cite{kach13} is in the definition of a $k$-functional and $\widehat{Ap}^k$: at each node $\alpha$ the $k$-functional in the previous model depended only on one component $\langle\alpha\rangle_{k-1}$, while in the current model it depends on components $\langle\alpha\rangle_{0}, \ldots,\langle\alpha\rangle_{k-1}$. This modification allows some versions of choice axioms to be forced in the new model, as will be shown further.
   
According to Definition \ref{def:forcing}, the forcing relation $\alpha\Vdash\varphi $ is defined in the Beth model $\mathcal{B}_s$.

Next lemma states natural properties of the application predicate.

\begin{lemma}
Suppose $k=1, \ldots, s$,  $f\!\in \! a_{k}$,  $n\in \omega $ and $ \widehat{Ap}^{k[\alpha]}(f,n) \downarrow$. Then 

\begin{enumerate}
\item $ \widehat{Ap}^{k[\alpha]}(f,n) \in a_{k-1};$ 
\bigskip

\item $ \beta\preccurlyeq\alpha  \Rightarrow \widehat{Ap}^{k[\beta]}(f,n) \downarrow \& \widehat{Ap}^{k[\beta]}(f,n) = \widehat{Ap}^{k[\alpha]}(f,n).$ 

\end{enumerate} \label{lemma:ap_predicate}
\end{lemma} 

\begin{proofplain}
follows from the definition of $\widehat{Ap}$ and monotonicity of the $k$-functional $f$.
\end{proofplain}

In the next lemma we show the validity of the interpretations of terms and functionals. 

\begin{lemma}
Suppose $ \alpha\in M $, and $ Z, V $ are evaluated $ k $-functionals ($ k=1, \ldots, s $) or terms of the language $ L_{s} $ (in this case $ k=0$). Then the following holds.

\begin{enumerate}
\item $ \left(Z^{[\alpha]}\downarrow\right) \Rightarrow \left(Z^{[\alpha]} \in a_{k}\right)$ .
\bigskip

\item $ \left(Z^{[\alpha]}\downarrow\right) \& \beta\preccurlyeq\alpha \Rightarrow \left(Z^{[\beta]}\downarrow\right) \&  \left(Z^{[\beta]}= Z^{[\alpha]}\right)$.
\medskip

\item $\left[ Val(\alpha,Z=_kV)=T\right]  \& \beta\preccurlyeq\alpha \Rightarrow \left[ Val(\beta,Z=_kV)=T\right]  $.
\medskip

\item For any path $ S $ in $ M $ there is $ \gamma\in S$ such that $ Z^{[\gamma]}\downarrow $ .

\end{enumerate} \label{lemma:term_int}
\end{lemma}

\begin{proof}
Parts 1, 2 and 3 have the same proof as in \cite{kach13}. Part 4 is proven by induction on the complexity of $ Z$. We consider only the non-trivial case $ Z=Ap^{k+1}(V,t) $. Consider a path $ S $ in $ M $. By the inductive assumption there are $ \alpha, \beta \in S $ such that $V^{[\alpha]}\downarrow $ and $ t^{[\beta]}\downarrow $. 
\medskip

Denote
$ f=V^{[\alpha]}, n= t^{[\beta]}$ and $ S'=\left\lbrace \bar{\delta}(k+1)\mid\delta\in S \right\rbrace $. Then $ S' $ is a path in $d_{k}$, $f \in a_{k+1}$ and $ n \in \omega$ by part 1.  
\medskip
\noindent Since the function $ f $ is complete, there is $ x \in S' $, for which $ f^{[x]}(n)\downarrow $. 

\medskip
For some $ \delta\in S$, $x= \bar{\delta}(k+1)$. Denote $ \gamma=min\left\lbrace \alpha,\beta,\delta \right\rbrace$. Then $\gamma\in S$. 

\medskip
By part 2, $ V^{[\gamma]}=V^{[\alpha]}=f , t^{[\gamma]}=t^{[\beta]}=n$ and 
$\bar{\gamma}(k+1)\preccurlyeq_{k+1} \bar{\delta}(k+1)=x$,

\medskip
\noindent since $\gamma\preccurlyeq\delta$. So $ Z^{[\gamma]} = V^{[\gamma]}(\bar{\gamma}(k+1), t^{[\gamma]}) =f(\bar{\gamma}(k+1), n)=f(x,n)=f^{[x]}(n) $, since $ f $ is monotonic. 
Hence $ Z^{[\gamma]}\downarrow $ .
\medskip
\end{proof}

Lemma \ref{lemma:term_int}.3 implies that  $\mathcal{B}_s$ is a Beth model. Clearly  $\mathcal{B}_s$ is embedded in  $\mathcal{B}_{s+1}$ and  $\mathcal{B}_1$ is essentially the van Dalen's model \cite{vand78}. Our interpretation of a $ k $-functional is similar to the interpretation of a choice sequence in \cite{vand78}; the difference is in the recursive use of the tree $d_{k-1}$ instead of the tree $ \omega ^{*} $.

\begin{lemma} 
Suppose $ \alpha\in M, f \in l_{k}, n\in\omega$. Then 

\[ (f(n)^{[\alpha]}\downarrow)   \Leftrightarrow (n < lh(\alpha)).\] 
\label{lemma:lawless1} 
\end{lemma}
\begin{proof}
It follows from the interpretation of a lawless functional.
\end{proof} 

The following are some examples of $ k $-functionals in the model $\mathcal{B}_s$, $k=1,\ldots,s$.

\begin{enumerate}
\item $\widehat{K}^{k}$ is a lawlike $ k $-functional; it is defined at the root of $ M $ and it is the same at every node; $ \widehat{K}^{k}\in b_{k}$.
\medskip

\item For $ x\in d_{k-1}, n\in \omega $, define:
\begin{displaymath}
f(x,n) =
\begin{cases}
\langle\langle x\rangle_{k-1}\rangle_n & \text{ if } n<lh(x),\\
\text{undefined}& \text{otherwise}.
\end{cases}
\end{displaymath} 

Clearly $f=\nu_k(\xi)  $, where $ \xi$ is an element of $ c_{k} $ given by: $\xi(n,y)=y$. So $ f\in l_{k} $; $ f $ is a lawless $ k $-functional whose values are determined by the path.
\medskip

\item For $ x\in d_{k-1}, n\in \omega $, define:
\begin{displaymath}
g(x,n) =
\begin{cases}
\text{undefined}& \text{ if } lh(x)=0,\\
\langle\langle x\rangle_{k-1}\rangle _0  & \text{ if } lh(x)>0.
\end{cases}
\end{displaymath} 

The $ k$-functional $ g $ is neither lawlike, nor lawless.
\end{enumerate}

\begin{theorem}[Soundness for $ L_s $]
\[ (L_s\vdash\varphi)\Rightarrow  (\mathcal{B}_s \Vdash\bar{\varphi}), \]
where $ \bar{\varphi} $ is the closure of the formula $ \varphi $, that is the formula $ \varphi $ with universal quantifiers over all its parameters.
\label{theorem:soundness_l} 
\end{theorem}

\begin{proofplain}
was given in \cite{kach13}.
\end{proofplain}

\section{Axiomatic theories $LP$ and $LP_s$}

Since $L$ contains arithmetic, in the language of $L$ we can define G\"{o}del numbering of symbols and expressions. For an expression $q$ we will denote $\llcorner q \lrcorner$ the G\"{o}del number of $q$ in this numbering. Detailed explanations of G\"{o}del numbering can be found, for example, in \cite{mend09}.

The language of theory $ LP $ is the language of $ L $ with an extra predicate symbol $ P_{\llcorner \bar{X}.\varphi \lrcorner} (z,\bar{X})$ for every formula $ \varphi $ of $ L $, which has all its parameters in the list $ \bar{X} $; here $ \bar{X} $ denotes a list of variables $ X_1,\ldots,X_n $. According to tradition, we denote this symbol as $ \vdash_z \varphi(\bar{X}) $. It is interpreted that the formula $ \varphi(\bar{X}) $ has been proven by the "creating subject" at time $z$.

Axioms of the theory $ LP $ are all axioms of $ L $, where the axiom schemata are taken for the formulas of the new language, and the following axioms (CS1)-(CS3) for the "creating subject"; in all three of them $ \varphi $ is an arbitrary formula of $L$.
\medskip

\begin{center}
\textbf{Axioms for the "creating subject".} 
\end{center}

(CS1) $(\vdash_z \varphi)\vee \neg (\vdash_z \varphi)$. 
\medskip

(CS2) $(\vdash_z \varphi)\supset (\vdash_{z+y} \varphi)$. 
\medskip

(CS3) $\exists z (\vdash_z \varphi)\equiv \varphi $.
\medskip

The language of $ LP_s $ is the language of $ L_s $ with extra predicate symbols $P_{\llcorner \bar{X}.\varphi \lrcorner} (z,\bar{X})$, where both $ \varphi $ and $ \bar{X}$ belong to the language of $ L_s $. Axioms of the theory $ LP_s $ are all axioms of $ LP $, which are formulas of the language $ LP_s $.

If $(A)$ is a formula of $LP$, we denote $(A_s)$ the version of $(A)$ with types not greater than $s$.

\section{Extending the model $\mathcal{B}_s$ to the language $LP_s$}
We extend the valuation mapping $Val$ to each atomic formula $ \vdash _t \varphi $, where $t$ is an evaluated term and $\varphi$ is an evaluated formula of $L_s$.

\[Val(\alpha,\vdash _t \varphi)=T \textup{\quad iff \quad} 
t^{[\alpha]}\downarrow \& \exists \gamma \left[(\alpha \preccurlyeq\gamma) \& (lh(\gamma)=t^{[\alpha]}) \& (\gamma\Vdash\varphi)\right].\] 

\begin{lemma} 
\[\left[ Val(\alpha,\vdash _t \varphi)=T\right]  \& (\beta\preccurlyeq\alpha) \Rightarrow \left[ Val(\beta,\vdash _t \varphi)=T\right].\]
\label{lemma:proof_pred}
\end{lemma} 
\begin{proof}
Suppose $ Val(\alpha,\vdash _t \varphi)=T $ and $ \beta\preccurlyeq\alpha$. Then there exists $ \gamma $ such that $ \alpha\preccurlyeq \gamma$, $lh(\gamma)=t^{[\alpha]}$ and $\gamma\Vdash\varphi$.
Since $\beta\preccurlyeq \alpha \preccurlyeq\gamma$, by Lemma \ref{lemma:term_int}.2, $lh(\gamma)=t^{[\alpha]}=t^{[\beta]}$ and $ Val(\beta,\vdash _t \varphi)=T $.
\end{proof}

The Lemma \ref{lemma:proof_pred} implies that $\mathcal{B}_s$ is a Beth model for the language $LP_s$.

\begin{lemma} 
Suppose $ \alpha\in M $, $ Z $ is an evaluated term or functional, and $ \varphi $ is an evaluated formula of the language $ LP_{s} $. Then 

\[ \left(Z^{[\alpha]}\downarrow\right)  \Rightarrow  \left[\alpha \Vdash \varphi(Z) \Leftrightarrow \alpha \Vdash \varphi(Z^{[\alpha]}) \right]. \]
 \label{lemma:sub_in_forcing} 
\end{lemma} 
\begin{proof} 
It follows from statement 4.5.1 on pg. 85 of \cite{drag87}.
\end{proof} 

\begin{theorem}[Soundness for $ LP_s $]
\[ LP_s\vdash\varphi\Rightarrow  \mathcal{B}_s \Vdash\bar{\varphi}. \]
\label{theorem:soundness}
\end{theorem}
\begin{proof}
Proof is by induction on the length of derivation of $ \varphi $. 

For axioms and derivation rules of the intuitionistic logic HPC it follows from the Soundness Theorem for Beth models (Lemma \ref{lemma:Beth}.3). The axioms for equality hold in the model due to the definition of $ Val$. 

For the induction axiom (axiom 4) it easily follows from the definition of forcing. For axioms 1-3, 5-7 it follows from Theorem \ref{theorem:soundness_l}. 

It only remains to check the three axioms for the "creating subject".

\begin{center}
Proof for $ (CS1_s) $
\end{center}

Consider an arbitrary path $ S $ in $ M $. There is an element $\alpha\in S$ with $ lh(\alpha)=z$. If $\alpha \Vdash \varphi $, then $\alpha \Vdash \left(\vdash_z \varphi \right)$.

Suppose $\alpha \nVdash \varphi $. Since 
$ lh(\alpha)=z $, we have for any $ \beta\preccurlyeq\alpha $,  $\beta \nVdash \left(\vdash_z \varphi \right) $. So $\alpha\Vdash\neg \left(\vdash_z \varphi \right) $. 

This proves $ \mathcal{B}_s \Vdash \left(\vdash_z \varphi \right)\vee \neg \left(\vdash_z \varphi \right)$.

\begin{center}
Proof for $ (CS2_s) $
\end{center}

Suppose $\alpha \Vdash \left(\vdash_z \varphi \right)$. Consider an arbitrary path $S$ in $M$ through $\alpha$ and the element $ \gamma\in S$ with $lh(\gamma)=z$. Then $\gamma\Vdash\varphi$. 

Choose $ \beta\in S $ with $ lh(\beta)=z+y$. Then $\beta\preccurlyeq\gamma $ and by monotonicity of forcing $\beta\Vdash\varphi $. This implies $\alpha \Vdash \left(\vdash_{z+y} \varphi \right)$.

\begin{center}
Proof for $ (CS3_s)$
\end{center}

$ \Rightarrow $ Suppose $\alpha \Vdash \exists z\left(\vdash_z \varphi \right)$. 
Consider an arbitrary path $S$ in $M$ through $\alpha$. There are $z\in\omega$ and $\beta\in S$ such that $\beta \Vdash \left(\vdash_z \varphi \right)$. Then for some $\delta\in S$ we have $Val(\delta,\vdash_z\varphi)=T$. So there is $\gamma$ with $\gamma\Vdash\varphi$ and $\delta\preceq\gamma$, hence $\gamma\in S$.

Thus, for any path $S$ through $\alpha$ there is $\gamma\in S$ with $\gamma\Vdash\varphi$. By Lemma \ref{lemma:Beth}.2, $\alpha\Vdash\varphi$.

$ \Leftarrow $ Suppose $\alpha\Vdash\varphi$. For any path $S$ through $\alpha$ we have: $\alpha\in S$ and $\alpha \Vdash \left(\vdash_z \varphi \right)$ for $z=lh(\alpha)$. Therefore $\alpha \Vdash \exists z\left(\vdash_z \varphi \right)$.
\end{proof}

Next we consider some intuitionistic principles in the language \textit{LP} and check whether they hold in the model $ \mathcal{B}_s$.

\section{Permutations in the model $\mathcal{B}_s$}
This technical section contains some facts about the model $\mathcal{B}_s$ necessary for the rest of the paper. 
We defined a $ k $-permutation as an element of the set $ c_{k} $, see Definition \ref{def:my_beth_model}.

In this section we fix permutations $ \xi_{0} \in c_{1},\ldots, \xi_{s-1} \in c_{s}$. They generate a few mappings introduced in the following definition.

\begin{definition}
Suppose $ k=0,\ldots,s-1 $.
\begin{enumerate}
\item Define $ \eta_{k} $ by the following:
\[ \eta_{k}(x,n)=\left(\xi^{[n]}_{k}\right)^{-1}(x). \]

\item Define $\tilde{\tilde{\xi}}_k$ by the following: 
\[\tilde{\tilde{\xi}}_k(< x_{0},\ldots, x_{m}>) = <\xi^{[0]}_{k}(x_{0}),\ldots, \xi^{[m]}_{k}(x_{m})> .\]

\item Define $\tilde{\xi}_k$ by the following:
\[\tilde{\xi}_k(< y_{0},\ldots,y_{k}>)=< \tilde{\tilde{\xi}}_{0}(y_{0}),\ldots, \tilde{\tilde{\xi}}_{k}(y_{k})>. \]

\item \[ \tilde{\xi}=\tilde{\xi}_{s-1}. \] 

\item Define $ \Lambda_{k} $ by induction on $k$ $ (k=0,1,\ldots,s)$.

$ \Lambda_{0}(n)=n. $

Assume $k\geqslant 1$ and $ \Lambda_{k-1}$ is defined. Define $\Lambda_{k}(f) = g$, where
\[ g(x,n)= \Lambda_{k-1}\left( f\left( \tilde{\xi}_{k-1}(x),n\right) \right).\]
\end{enumerate} \label{def:permutations}
\end{definition}

This completes the Definition \ref{def:permutations}. 

Clearly, $\tilde{\tilde{\xi}}_k$ is a bijection on $a_k^*$ and $\tilde{\tilde{\eta}}_k$ is its inverse. Also $\tilde{\xi}_k$ is a bijection on $d_k$ and $\tilde{\eta}_k$ is its inverse. For a lawlike $ k $-functional $f$, $ \Lambda_{k}(f)=f $ because $f(x,n)$ does not depend on $x$.

The following technical lemma gives reasons for introducing the aforementioned mappings.

\begin{lemma}
Suppose $ k=0,\ldots,s-1 $.
\begin{enumerate}
\item $ \eta_{k}\in c_{k+1}$.
\medskip

\item $\tilde{\xi_{k}}$ is an isomorphism of the partially ordered set $d_{k} $, i.e. a bijection preserving the order:
\[y \preccurlyeq_k x\Leftrightarrow \tilde{\xi}_{k}(y) \preccurlyeq_k \tilde{\xi}_{k}(x).\]

\item $ \tilde{\xi} $  is an isomorphism of the partially ordered set $M$, so
\[\beta\preccurlyeq   \alpha\Leftrightarrow\tilde{\xi}(\beta) \preccurlyeq \tilde{\xi}(\alpha).\]

\item Suppose $ S $ is a path in $d_{k} $ through $ x $ and $ S'= \left\lbrace \tilde{\xi}_{k}(y) \mid y \in S \right\rbrace $. Then $ S' $ is a path in $ d_{k} $ through $ \tilde{\xi}_{k}(x) $.
\medskip

\item Suppose $ S $ is a path in $ M $ through $ \alpha $ and $ S'= \left\lbrace \tilde{\xi}(\beta) \mid \beta \in S \right\rbrace $. Then $ S' $ is a path in $ M$ through $ \tilde{\xi}(\alpha) $.
\medskip

\item For $ k=0, 1,\ldots,s, $
\[f \in a_{k}  \Leftrightarrow  \Lambda_{k}(f) \in a_{k}.\]
\end{enumerate}
\label{lemma:permutations}
\end{lemma}

\begin{proof}
1. Follows from the definition.
\medskip

2. Clearly $\tilde{\tilde{\xi}}_k$ is an isomorphism of the partially ordered set $a_{k}^* $, i.e. a bijection preserving the order:
\[y\leqslant x\Leftrightarrow \tilde{\tilde{\xi}}_k(y) \leqslant \tilde{\tilde{\xi}}_k(x).\]
The rest follows from the definition of $\tilde{\xi}_k$.

3. Follows from part 2 for $k=s-1$.

4. Follows from part 2.

5. Follows from part 4 for $k=s-1$.

6. $ \Rightarrow $ The proof is by induction on $ k $. The basis step $k=0$ is obvious. We assume that the statement holds for $ k $ and prove it for $ k+1$. 

Suppose $ f\in  a_{k+1} $ and $g= \Lambda_{k+1}(f)  $. To show that $ g\in  a_{k+1} $ we check three conditions.

(i) We need to prove that $ g:(d_k \times\omega) \dashrightarrow a_k$. 

Suppose $ x\in d_k$, $ n\in\omega $ and $ y=g(x,n) $ is defined. Then $ \tilde{\xi}_{k} (x)\in d_k$ and $y=\Lambda_{k}\left( f\left( \tilde{\xi}_{k} (x),n  \right)  \right)  $ is defined. So $ f\left( \tilde{\xi}_{k} (x),n \right)  \in a_k $ and by the inductive assumption $ y=\Lambda_{k}\left( f\left( \tilde{\xi}_{k} (x),n  \right)  \right) \in a_k $.
\medskip

(ii) Monotonicity. 

Suppose $ y \preccurlyeq_k x $ and $ g(x,n)\downarrow\: $. Then $ f\left( \tilde{\xi}_{k} (x),n  \right)\downarrow$ and $ \tilde{\xi}_{k} (y) \preccurlyeq_k \tilde{\xi}_{k} (x) $ by part 2. By monotonicity of $ f $, $ f\left( \tilde{\xi}_{k} (y),n  \right)\downarrow$ and $ f\left( \tilde{\xi}_{k} (y),n  \right) = f\left( \tilde{\xi}_{k} (x),n  \right)$, so $ g(y,n)\downarrow $ and $g(y,n) = g(x,n)$.
\medskip

(iii) Completeness.

Consider $ n\in\omega $ and a path $S$ in $d_k$. Denote $ S'= \left\lbrace \tilde{\xi}_{k}(z) \mid z \in S \right\rbrace $. By part 4, $ S' $ is a path in $d_k$. Then there is $ y\in S' $ such that $ f(y,n) \! \downarrow $. For some $z\in S, y=\tilde{\xi}_{k} (z)$. So $ f\left( \tilde{\xi}_{k} (x),n  \right) $ is defined and belongs to $ a_k$. By the inductive assumption $ g(x,n)=\Lambda_{k} \left( f\left( \tilde{\xi}_{k} (x),n  \right)\right) \in a_{k} $ and $ g(x,n) $ is defined.
\medskip

$ \Leftarrow $ This follows from the previous if we use $\eta_{0},\ldots, \eta_{s-1} $ instead of $ \xi_{0},\ldots, \xi_{s-1}$.
\end{proof}

The following lemma describes a relation between permutations and forcing.

\begin{lemma}
Suppose:
\begin{list}{}{}
\item $ \bar{X} $ is a list of variables $ X_1,\ldots,X_r $ of types $ k_1,\ldots,k_r $, respectively;
\item $ \bar{f} $ is a list $ f_1,\ldots,f_r $  of objects from corresponding domains;
\item $ \bar{g} $ is the list $ g_1,\ldots,g_r $, where $ g_i=\Lambda_{k_i}\left( f_i\right)$;
\item $ \beta=\tilde{\xi} (\alpha) $.
\end{list}

\begin{enumerate}
\item Suppose $Z(\bar{X}) $ is an $ n $-functional ($ n=1,\ldots,s $) or a term (then $ n=0 $), and all its parameters are in the list $ \bar{X} $. Then 
\bigskip

i) $ \left(Z(\bar{f})^{[\beta]}\downarrow\right)\Rightarrow \left(Z(\bar{g})^{[\alpha]}\downarrow\right)\&  \left[Z(\bar{g})^{[\alpha]}=\Lambda_n \left( Z(\bar{f})^{[\beta]}\right)\right]. $
\bigskip

ii) $ \left(Z(\bar{f})^{[\beta]}\downarrow\right) \Leftrightarrow \left(Z(\bar{g})^{[\alpha]}\downarrow\right).$
\bigskip

\item Suppose $\varphi(\bar{X}) $ is a formula of $ L_s $ and all its parameters are in the list $ \bar{X} $. Then 
\[  \beta\Vdash \varphi(\bar{f})  \;\Leftrightarrow\;  \alpha\Vdash \varphi(\bar{g})  .\]

\item Suppose $\varphi(\bar{X}) $ is a formula of $ LP_s $ and all its parameters are in the list $ \bar{X} $. Then 
\[  \beta\Vdash \varphi(\bar{f})  \;\Leftrightarrow\;  \alpha\Vdash \varphi(\bar{g})  .\]
\end{enumerate} \label{lemma:equiv_nodes}
\end{lemma} 

\begin{proof}
1. i) It is proven by induction on the complexity of $ Z $. We consider the case $ Z=Ap^{n+1}(V,t) $. Other cases are straightforward. Suppose $ Z(\bar{f})^{[\beta]}\downarrow $. Then 
\[  V(\bar{f})^{[\beta]}\downarrow \& t(\bar{f})^{[\beta]}\downarrow \& V(\bar{f})^{[\beta]} \left(\bar{\beta}(n),t(\bar{f})^{[\beta]}\right)\downarrow. \]
By the inductive assumption
\[  V(\bar{g})^{[\alpha]}\downarrow \& t(\bar{g})^{[\alpha]}\downarrow \& V(\bar{g})^{[\alpha]}=\Lambda_{n+1} \left( V(\bar{f})^{[\beta]}\right) \& t(\bar{g})^{[\alpha]}=t(\bar{f})^{[\beta]} . \]
\begin{multline*} 
Z(\bar{g})^{[\alpha]}=V(\bar{g})^{[\alpha]} \left(\bar{\alpha}(n),t(\bar{g})^{[\alpha]} \right) = \Lambda_{n+1}\left( V\left( \bar{f}\right) ^{[\beta]}\right)
\left(\bar{\alpha}(n),t(\bar{f})^{[\beta] } \right)
\\ \\
=\Lambda_{n}\left[ V(\bar{f})^{[\beta]}\left(\tilde{\xi}_n \left(\bar{\alpha}(n) \right),t(\bar{f})^{[\beta]}  \right)  \right]   =\Lambda_n  \left[V(\bar{f})^{[\beta]}\left(\bar{\beta}(n), t(\bar{f})^{[\beta]}\right)  \right]
\\ \\
=\Lambda_n \left(Z(\bar{f})^{[\beta]} \right). 
\end{multline*}

 So  $ Z(\bar{g})^{[\alpha]}\downarrow\;. $

ii) The implication from left to right follows from part i). The implication from right to left is proven using $ \eta_{0},\ldots,\eta_{s-1} $ instead of $ \xi_{0},\ldots,\xi_{s-1} $.

2. The proof involves the same steps as in \cite{kach13}. 

3. The proof is by induction on the complexity of $ \varphi $ using Definition \ref{def:forcing} of forcing. Let us consider the case when $ \varphi $ is an atomic formula $ \vdash_{t(\bar{X})} \psi(\bar{X}) $. All other cases are proven the same way as in \cite{kach13}.

$ \Rightarrow $ Suppose $ \beta \Vdash \varphi (\bar{f}) $. Consider an arbitrary path $ S $ in $ M $ through $ \alpha $. Denote $ S'= \left\lbrace \tilde{\xi}(\gamma) \mid \gamma\in S \right\rbrace $. By Lemma \ref{lemma:permutations}.5, $ S' $ is a path in $ M $ through $ \beta $. So there is $ \delta'\in S' $ such that $ Val\left(\delta',\vdash_{t(\bar{f})} \psi(\bar{f}) \right)=T $. There exists $ \gamma' $ such that $ \delta'\preccurlyeq\gamma' $, $lh(\gamma')=t(\bar{f})^{[\delta']}$ and $ \gamma' \Vdash \psi(\bar{f})$. Clearly, $\gamma' \in S'$ and for some $\gamma\in S,\gamma'=\tilde{\xi}(\gamma) $. By part 2, $\gamma \Vdash \psi (\bar{g})$.

For some $\delta \in S$, $\delta'=\tilde{\xi}(\delta)$ and $\delta\preccurlyeq\gamma$. By the definition of $ \tilde{\xi} $ and part 1,
\smallskip
\\$lh(\gamma)=lh(\gamma')=t(\bar{f})^{[\delta']}=t(\bar{g})^{[\delta]}$, so $ Val\left(\delta,\vdash_{t(\bar{g})} \psi(\bar{g}) \right)=T $. Hence $\alpha \Vdash \varphi (\bar{g}) $.
\bigskip

$ \Leftarrow $ follows from the previous if we take $ \eta_{0},\ldots,\eta_{s-1} $ instead of $ \xi_{0},\ldots,\xi_{s-1} $.
\end{proof}

The following lemma states some properties of forcing in the model $ \mathcal{B}_s $.

\begin{lemma}
Suppose:
\begin{list}{}{}
\item $ \bar{X} $ is a list of variables $ X_1,X_2,\ldots,X_r $; 
\item $ \bar{f} $ is a list $f_1, f_2,\ldots,f_r $  of objects from corresponding domains.
\end{list}

\begin{enumerate}
\item Suppose $\varphi(\bar{X}) $ is a formula of $ LP_s $ with all its parameters in the list $ \bar{X} $ and $ sort(\varphi)\leqslant m $, where $ m\geqslant 1$. Suppose 
\[ \bigwedge_{i=0}^{m-1} \left( \langle \alpha \rangle_i  = \langle \beta\rangle _i  \right).\]
Then 
\[ \left(\alpha\Vdash \varphi(\bar{f})\right)\Leftrightarrow\left(\beta\Vdash \varphi(\bar{f})\right).\]

\item Suppose $\varphi(\mathcal{H}^p,\bar{X}) $ is a formula of $ LP_s $ with all its parameters in the list $\mathcal{H}^p, \bar{X} $; $ sort(\varphi)\leqslant p $ and $ \varphi $ does not have non-lawlike parameters of type $ p $ other than $ \mathcal{H}^p $. Suppose $g,h\in l_{p}$ and 
\[(\forall n < lh(\alpha)) \left( g(n)^{[\alpha]} = h(n)^{[\alpha]}\right).  \]
Then 
\[ \left(\alpha\Vdash \varphi(g,\bar{f})\right)\Leftrightarrow\left(\alpha\Vdash \varphi(h,\bar{f})\right).\]

\end{enumerate} \label{lemma:lawless2}

\end{lemma}

\begin{proof} 
Denote $k_i$ the type of variable $X_i$, $i=1\ldots,r$.

1. Denote $ l=lh(\alpha)$. Since $ \langle \alpha \rangle_0  = \langle \beta\rangle _0 $, we have $lh(\beta)=lh(\alpha)=l$. For any $ k=0,\ldots,s-1 $ define $ \xi_{k} $ by the following: for $ n\in\omega,x\in a_k,$

\begin{displaymath}
\xi_k(n,x)=
\begin{cases}
\langle \langle \beta\rangle _k\rangle _n & \text{if } \left( n<l \& x= \langle \langle \alpha\rangle _k\rangle _n\right) ,\\
\langle \langle \alpha\rangle _k\rangle _n & \text{if } \left( n<l \& x= \langle \langle \beta\rangle _k\rangle _n\right) ,\\
x & \text{otherwise }.
\end{cases} 
\end{displaymath}

Then each $ \xi_k \in c_{k+1} $ and $ \beta=\tilde{\xi}(\alpha) $.
Denote $ g_i=\Lambda_{k_i}\left(f_i \right)$ for $ i=1,\ldots,r $. By Lemma \ref{lemma:equiv_nodes}.3, 

\begin{equation}
\beta\Vdash \varphi(\bar{f}) \Leftrightarrow\alpha\Vdash \varphi(\bar{g}). \label{eq:lawless1}
\end{equation} 

If $ k<m$, then $\langle \alpha \rangle_k  = \langle \beta\rangle _k$ and 
\begin{eqnarray}
(\forall n\in \omega)(\forall x \in a_k)\left(\xi_k^{[n]}(x)=x \right),
\nonumber\\ 
(\forall y \in d_k)\left(\tilde{\xi}_k(y)=y \right),
\nonumber
\end{eqnarray}
so for any $ j\leqslant m $,
\[ (\forall z\in a_j)\left(\Lambda_j(z)=z \right) . \]
Since $ sort(\varphi)\leqslant m $, for each $ i=1,\ldots,r$, we have $k_i\leqslant m $ and 
$ g_i=\Lambda_{k_i}(f_i)=f_i$. So by (\ref{eq:lawless1}) 
\[ \beta\Vdash \varphi(\bar{f})\Leftrightarrow\alpha\Vdash \varphi(\bar{f}).\]

2. For some $ \chi_1,\chi_2 \in c_{p} $, $ g=\nu_p(\chi_1) $ and $ h=\nu_p(\chi_2) $. Clearly $ p\geqslant 1$. Define permutations $ \xi_0,\ldots,\xi_{s-1} $:
\[\xi_{p-1}^{[n]}=\left( \chi_2^{[n]}\right)^{-1}\circ\chi_1^{[n]}; \quad   \xi_k^{[n]}(x)=x\;  \text{ for }k\neq p-1. \]

By Lemma \ref{lemma:equiv_nodes}.3, it is sufficient to prove the following:
\begin{equation}
\tilde{\xi}(\alpha)=\alpha; \label{eq:lawless2}
\end{equation}
\begin{equation}
\Lambda_{k_i}\left(f_i \right)=f_i,\;i=1,\ldots,r; \label{eq:lawless3}
\end{equation}
\begin{equation}
\Lambda_p(h)=g. \label{eq:lawless4}
\end{equation}

\begin{center}
Proof of (\ref{eq:lawless2}). 
\end{center}

Denote $\alpha=\langle\alpha_0,\ldots,\alpha_{s-1}\rangle $. Then $ \tilde{\xi} (\alpha)= \langle \tilde{\tilde{\xi}}_0(\alpha_0),\ldots,\tilde{\tilde{\xi}}_{s-1}(\alpha_{s-1}) \rangle$.
Clearly $\tilde{\tilde{\xi}}_k(\alpha_k)=\alpha_k$ for $ k\neq p-1 $.

Now consider $ k=p-1 $. Fix $ n<lh(\alpha) $. Since $ g=\nu_p(\chi_1) $ and $ h=\nu_p(\chi_2) $, 
\[ \chi_1^{[n]}\left( \langle \alpha_{p-1}\rangle_n \right) = g(\bar{\alpha}(p),n) =g(n)^{[\alpha]}= h(n)^{[\alpha]}=h(\bar{\alpha}(p),n)=\chi_2^{[n]}\left( \langle \alpha_{p-1}\rangle_n \right). \]

Thus, $ \chi_1^{[n]}\left( \langle \alpha_{p-1}\rangle_n \right) =\chi_2^{[n]}\left( \langle \alpha_{p-1}\rangle_n \right) $ and 
\[ \xi_{p-1}^{[n]}\left(\langle \alpha_{p-1} \rangle_n \right) = \left( \chi_2^{[n]}\right) ^{-1}
\left(  \chi_1^{[n]}\left( \langle \alpha_{p-1}\rangle_n \right) \right) =
\langle \alpha_{p-1} \rangle_n. \]

So for any $ n<lh(\alpha) $, $ \xi_{p-1}^{[n]}\left(\langle \alpha_{p-1} \rangle_n \right) = \langle \alpha_{p-1} \rangle_n$ and $\tilde{\tilde{\xi}}_{p-1}(\alpha_{p-1})=\alpha_{p-1} $.

\begin{center}
Proof of (\ref{eq:lawless3}). 
\end{center}

For $k<p-1,\; \tilde{\xi}_k(x)=x $ for any $ x\in d_k$ and 
\begin{equation}
(\forall y \in a_{j})\left( \Lambda_j(y)=y\right), j<p. \label{eq:lawless5}
\end{equation}
If $ f_i $ is a domain object for a lawlike $ k_i $-functional, then $\Lambda_{k_i}(f_i)=f_i  $; otherwise $ k_i<p $ and $\Lambda_{k_i}(f_i)=f_i  $ by (\ref{eq:lawless5}).

\begin{center}
Proof of (\ref{eq:lawless4}). 
\end{center}

For any $ n\in\omega,x\in d_{p-1}$,
\[ \Lambda_p(h)(x,n)=\Lambda_{p-1}\left(h\left(\tilde{\xi}_{p-1}(x),n \right) \right)= h\left( \tilde{\xi}_{p-1}(x),n \right) \text{ by } (\ref{eq:lawless5}).\]

Since $ \tilde{\xi}_{p-1}(x) $ has the same length as $ x $, it is sufficient to prove:
\[h\left( \tilde{\xi}_{p-1}(x),n \right)=g(x,n) \text{ for any }n<lh(x) .\]

Since $g=\nu_p(\chi_1)$ and $ h=\nu_p(\chi_2) $, we have: 
\begin{multline*}
h\left( \tilde{\xi}_{p-1}(x),n \right)= \chi_2^{[n]}\left(\langle \langle \tilde{\xi}_{p-1}(x)\rangle_{p-1}\rangle_n\right)=\chi_2^{[n]}\left(\xi_{p-1}^{[n]}\left(\langle \langle x\rangle_{p-1}\rangle_n \right)\right)\\
=\chi_2^{[n]}\left[\left(\chi_2^{[n]} \right)^{-1}\circ\chi_1^{[n]} \left(\langle \langle x\rangle_{p-1}\rangle_n  \right)\right]
=\chi_1^{[n]} \left(\langle \langle x\rangle_{p-1}\rangle_n \right)=g(x,n).
\end{multline*}
\end{proof}

\section{Axioms for lawless functionals}

We can introduce finite sequences in $LP$ and for any $k$-functional $F$ use the notation $\bar{F}(n)=<F(0),\ldots,F(n-1)>$.

The following are modifications for the language $ LP $ of the axioms for lawless sequences in \cite{vand78} and the axioms for lawless functionals in \cite{bern76}. 
\medskip

$ (LL1) $ The axiom of existence of lawless functionals:
\[ \exists \mathcal{F}^n (\forall y\leqslant x)(\mathcal{F}(y)=F^n(y)).\]

\medskip

$ (LL2) $ $\mathcal{F}^n =\mathcal{G}^n \vee \mathcal{F}^n\ne\mathcal{G}^n. $
\medskip

$ (LL3) $ The principle of open data:
\[\varphi(\mathcal{H}^n)\supset\exists x\forall\mathcal{G}^n [ \mathcal{\bar{G}}(x)=\mathcal{\bar{H}}(x) \supset\varphi(\mathcal{G})],\]
where $ \varphi $ is a formula of $ LP $, $sort(\varphi) \leqslant n$ and $\varphi$ does not have non-lawlike parameters of type $n$ other than $\mathcal{H}^n$.
\smallskip

Denote $ (LL) $ a conjunction of the closures of $ (LL1) $, $ (LL2) $ and $ (LL3) $.

\begin{theorem}
$\mathcal{B}_s \Vdash (LL _{s}) $.
\label{theorem:lawless}
\end{theorem}
\begin{proof}
$(LL1_s)$. Fix $f\in a_{n}$ and a path $ S $ in $ M $. By induction on $ x $ we will prove:
\[(\exists \alpha \in S)(\forall y\leqslant x)f(y)^{[\alpha]} \downarrow.\] 

For $ x=0 $ it follows from Lemma \ref{lemma:term_int}.4. 

Assume it holds for $ x $: for some $ \beta\in S $, $ (\forall y\leqslant x)f(y)^{[\beta]} \downarrow $. By Lemma \ref{lemma:term_int}.4, there is $ \gamma\in S $ such that $ f(x+1)^{[\gamma]}\downarrow $. For $ \alpha=min\left\lbrace \beta,\gamma \right\rbrace$ we have $ (\forall y\leqslant x+1)f(y)^{[\alpha]} \downarrow $.

Now let us prove that $(LL1_s)$ holds in the model. Fix an arbitrary $ x \in \omega$ and $ \alpha \in S$ such that $ (\forall y\leqslant x)f(y)^{[\alpha]} \downarrow$ . 
\smallskip

The path $ S=\left\lbrace \langle \bar{g}_0(k),\ldots,\bar{g}_{s-1}(k)\rangle \mid k\in \omega \right\rbrace$ for some functions 
\smallskip
\\$ g_0:\omega\rightarrow a_0,\ldots, g_{s-1}:\omega\rightarrow a_{s-1}$.
Denote $ \beta = \langle\bar{g}_0(x+1),\ldots, \bar{g}_{s-1}(x+1)\rangle$ 
\smallskip
\\and $\gamma=min\left\lbrace \alpha,\beta \right\rbrace$. 
Then $\beta, \gamma \in S$ and $ (\forall y\leqslant x)f(y)^{[\gamma]} \downarrow $.

Define $ \xi $ by:
\begin{displaymath}
\xi(y,z)=
\begin{cases}
f(y)^{[\gamma]} & \text{if } y\leqslant x\& z=g_{n-1}(y),\\
g_{n-1}(y) & \text{if } y\leqslant x\& z=f(y)^{[\gamma]},\\
z & \text{otherwise }.
\end{cases}
\end{displaymath} 

For any $ y\leqslant x$, $ g_{n-1}(y)\in a_{n-1} $ and $ f(y)^{[\gamma]}\in a_{n-1} $. So $ \xi:(\omega\times a_{n-1})\rightarrow a_{n-1}$ and for any $ y\in \omega,\; \xi^{[y]}$ is a bijection on $ a_{n-1} $. Therefore $ \xi\in c_{n} $. Take $h=\nu_n(\xi)\in l_{n} $. It remains to prove:
\[ \gamma\Vdash (\forall y\leqslant x)\left[ h(y)=f(y)\right].\]

For any $ y\leqslant x,h(y)^{[\gamma]}=h\left(\bar{\gamma}(n),y \right)=\xi^{[y]}\left( \langle\langle\gamma\rangle_{n-1}\rangle_y\right)$. Since $ \gamma\preccurlyeq\beta $, we have
\medskip\\
$\langle\gamma\rangle_{n-1} \leqslant \langle\beta\rangle_{n-1}$ and $\langle\langle\gamma\rangle_{n-1}\rangle_y=g_{n-1}(y) $. So $h(y)^{[\gamma]}=\xi^{[y]}(g_{n-1}(y))=f(y)^{[\gamma]}$
\smallskip
\\by the definition of $ \xi $.
\medskip

For $(LL2_s)$ the proof follows from the interpretation of equality.
\medskip

$(LL3_s)$. Suppose $ \bar{X} $ is a list of variables $ X_1,\ldots,X_r $ and $ \bar{f} $ is a list $ f_1,\ldots,f_r $ of objects from corresponding domains. Suppose all parameters of $ \varphi $ are in the list $\mathcal{H}^n, \bar{X} $. Suppose $h\in l_{n}$ and 

\begin{equation}
\alpha\Vdash\varphi(h,\bar{f}). \label{eq:lawless6}
\end{equation}

Take $ x=lh(\alpha) $. It is sufficient to prove that for any $ g\in l_{n}$, 
\[ \alpha\Vdash (\forall y < x)\left(g(y)=h(y)\right) \supset \varphi(g,\bar{f}) . \]

Suppose $ \beta\preccurlyeq \alpha$ and $ (\forall y<x)\beta\Vdash (g(y)=h(y))$. Since $ x=lh(\alpha) $, we have
\medskip\\
$(\forall y<x) \left[ g(y)^{[\alpha]}\downarrow \& h(y)^{[\alpha]}\downarrow\right]$ by Lemma \ref{lemma:lawless1}. 
So $(\forall y<x)\left( g(y)^{[\alpha]}=h(y)^{[\alpha]}\right)$.
\medskip\\
By (\ref{eq:lawless6}) and Lemma \ref{lemma:lawless2}.2, $ \alpha\Vdash \varphi(g,\bar{f})$.
\end{proof}

\section{Choice axioms in the model $\mathcal{B}_s$}
For an $m$-functional $ F$ of language $LP$ and $n\leqslant m$ denote  $F\underbrace{(0)\ldots(0)}_n$ as $ F(0)^n$ for brevity.
\medskip

We consider the following two axioms of choice and show that they hold in the model $\mathcal{B}_s$.

$ (C1) $ The axiom of choice for numbers:
\[ \forall x\exists y \varphi(x,y)\supset\exists F^m\forall x\varphi(x,F(x)(0)^{m-1}), \] 
where $ \varphi $ is a formula of $ LP $ and $m \geqslant max(sort(\varphi),1) $.
\bigskip

$ (C2) $ The axiom of choice with uniqueness:  
\[ \forall x\exists!G^n \varphi(x,G)\supset\exists F^m\forall x\varphi(x,F(x)(0)^{m-n-1}), \] 
where $ \varphi $ is a formula of $ LP $ and $m \geqslant max(sort(\varphi),n+1) $. 

Denote $ (C) $ a conjunction of the closures of $ (C1) $ and $ (C2) $.

\begin{theorem}
$\mathcal{B}_s \Vdash (C_{s}) $.
\label{theorem:choice}
\end{theorem}

\begin{proof}
In case $ s=1 $ the proof is the same as in \cite{vand78}, since the model described there is essentially $\mathcal{B}_1$. We will assume $ s\geqslant 2 $.
\medskip

For $a\in a_0^{(k)} \times\ldots\times a_{m-1}^{(k)}$, denote:
\[extension(a)=a\ast<<\underbrace{
\widehat{K}^m,\ldots,\widehat{K}^m}_k>,\ldots,<\underbrace{\widehat{K}^{s-1},\ldots,\widehat{K}^{s-1}}_k >>.\]

The operation of extension extends a node $a\in d_{m-1}$ by constant functionals to a node in the domain $M$; it preserves the length of the node. 
\medskip

By Lemma \ref{lemma:lawless2}.1, for any formula $\chi $ with $ sort(\chi)\leqslant m $ and its evaluation $ \tilde{\chi} $:

\begin{equation}
\bar{\alpha}(m)=\bar{\beta}(m)\Rightarrow\left(\alpha\Vdash\tilde{\chi}\Leftrightarrow \beta\Vdash\tilde{\chi}\right).  
\label{eq:contraction}
\end{equation}
\medskip

\begin{center}
PROOF for $(C1_s)$
\end{center}

Denote $ \psi(x,y) $ the formula $ \varphi(x,y) $, in which all parameters, except $x$ and $y$, are replaced by objects from corresponding domains.

Denote 

\[ \textit{shortest }(\alpha,\beta,x)\leftrightharpoons \alpha\preccurlyeq \beta\& \exists y \left[\beta\Vdash \psi(x,y)\right]  \& (\forall \gamma \succ\beta)\neg \exists z \left[\gamma\Vdash \psi(x,z)\right]. \]

We have: 
\begin{equation}
\exists y \left[\alpha\Vdash \psi(x,y)\right] \Rightarrow \exists! \beta \textit{shortest }(\alpha,\beta,x).
\label{eq:shortest1}
\end{equation}

Indeed, of the finite number of nodes $\beta\succcurlyeq \alpha$ with $\exists y \left[\beta\Vdash \psi(x,y)\right]$ we take the unique $\beta$ with smallest length.

Fix $\alpha\in M$. Suppose 
\[ \alpha\Vdash\forall x\exists y\psi(x,y). \]
Denote $\alpha'=extension(\bar{\alpha}(m))$. So by (\ref{eq:contraction}),
\begin{equation}
 \alpha'\Vdash\forall x\exists y\psi(x,y) 
\label{eq:choice1}
\end{equation}

CASE $m=2$. First we construct the proof for this case because it clearly demonstrates the main construction. Other cases will be considered later. 

Thus, $ sort(\varphi)\leqslant 2 $. 

By (\ref{eq:contraction}), it is sufficient to prove: 
\begin{equation}
\alpha'\Vdash \exists F^2\forall x\psi(x,F(x)(0)).
\label{eq:choice2}
\end{equation}

First we construct a function $f\in a_2$ that can be used in (\ref{eq:choice2}). 

For fixed $\beta$ and $x\in \omega$ define function $q_{\beta, x}$: 
\begin{displaymath}
q_{\beta,x}(v,k)=
\begin{cases}
0 \quad \text{ if } lh(v)\geqslant lh(\beta) \& \neg(v\preccurlyeq_0\bar{\beta}(1)),\\
\text{the smallest number }y\text{, for which }\beta\Vdash \psi(x,y) \\
\qquad \text{if } v\preccurlyeq_0\bar{\beta}(1)\& \exists z\left[\beta\Vdash \psi(x,z)\right] ,\\
\text{undefined } \qquad \text{ otherwise.}
\end{cases}
\end{displaymath} 

Next we construct a function that can be used in (\ref{eq:choice2}).
\begin{displaymath}
f(u,x)=
\begin{cases}
\widehat{K}^1 \qquad \text{ if } lh(u)\geqslant lh(\alpha) \& \neg(u\preccurlyeq_1\bar{\alpha}(2)),\\
q_{\beta, x} \text{ for the unique } \beta \text{ with } shortest(extension(u),\beta,x) \\ 
\qquad \qquad \text{ if } u\preccurlyeq_1\bar{\alpha}(2) \& \exists y\left[extension(u)\Vdash \psi(x,y)\right] ,\\
\text{undefined } \quad \text{ otherwise.}
\end{cases}
\end{displaymath} 

\begin{center}
Proof that $f\in a_2$.
\end{center}
 
1) First we prove that $f:(d_1\times\omega)\dashrightarrow a_1$. 
Suppose $f(u,x)\downarrow$. We need to prove that $f(u,x)\in a_1$. 

When $f(u,x)=\widehat{K}^1$, it is obvious. Assume $f(u,x)=q_{\beta, x}$. Then 
$u\preccurlyeq_1\bar{\alpha}(2)$, $shortest(extension(u),\beta,x)$ and 
\begin{equation}
\exists z\left[\beta\Vdash \psi(x,z)\right]. 
\label{eq:choice3}
\end{equation}

The proof that $q_{\beta,x}\in a_1$ consists of three parts.

a) Clearly, $q_{\beta,x}:(d_0\times\omega)\dashrightarrow a_0$.

b) Monotonicity of $q_{\beta,x}$.

Suppose $q_{\beta,x}(v,k)\downarrow$ and $w\preccurlyeq_0v$. Then $lh(v)\geqslant lh(\beta)$. There are two cases.
\smallskip

i) $\neg(v\preccurlyeq_0 \bar{\beta}(1))$.

If $w\preccurlyeq_0 \bar{\beta}(1)$, then $\bar{\beta}(1)=\bar{w}(m)$ and $v=\bar{w}(n)$ for some $m$ and $n$, which 
\smallskip
\\implies $v\preccurlyeq_0 \bar{\beta}(1)$ due to $lh(v)\geqslant lh(\beta)$. Thus, $\neg(w\preccurlyeq_0 \bar{\beta}(1))$ and $q_{\beta,x}(w,k)=0=q_{\beta,x}(v,k)$.
\smallskip

ii) $v\preccurlyeq_0 \bar{\beta}(1)$.

Then $w\preccurlyeq_0 \bar{\beta}(1)$. By (\ref{eq:choice3}) both $q_{\beta,x}(w,k)$ and $q_{\beta,x}(v,k)$ equal the smallest number $y$, for which $\beta\Vdash \psi(x,y)$.

c) Completeness of $q_{\beta,x}$.

Consider a path $S_0$ in $d_0$ and $k\in \omega$. Take $v\in S_0$ with length $lh(\beta)$. There are two cases.

i) $\bar{\beta}(1)\notin S_0$.

Then $\neg(v\preccurlyeq_0 \bar{\beta}(1))$ and $q_{\beta,x}(v,k)=0$.

ii) $\bar{\beta}(1)\in S_0$.

Then $v=\bar{\beta}(1) $ and by (\ref{eq:choice3}), $q_{\beta,x}(v,k)$ equals the smallest number $y$, for which $\beta\Vdash \psi(x,y)$.

This completes the proof that $q_{\beta,x}\in a_1$ and therefore $f:(d_1\times\omega)\dashrightarrow a_1$.

2) Monotonicity of $f$.

Suppose $f(u,x)\downarrow$ and $v\preccurlyeq_1 u$. Then $lh(v)\geqslant lh(u)\geqslant lh(\alpha)$. There are two cases.

i) $\neg(u\preccurlyeq_1\bar{\alpha}(2))$.

In this case $\neg(v\preccurlyeq_1\bar{\alpha}(2))$, so $f(v,x)=\widehat{K}^1=f(u,x)$.

ii) $u\preccurlyeq_1\bar{\alpha}(2)$.

Then $v\preccurlyeq_1\bar{\alpha}(2)$, there is a unique $\beta$ with $shortest(extension(u),\beta,x)$ and $f(u,x)=q_{\beta,x}$.

For some $y$ we have:  $extension(u)\Vdash \psi(x,y)$. 

Since $extension(v)\preccurlyeq extension(u)$, we have: $shortest(extension(v),\beta,x)$ and by monotonicity of forcing, $extension(v)\Vdash \psi(x,y)$. So $f(v,x)=q_{\beta, x}$ for the same $\beta$.

3) Completeness of $f$.

Consider $x\in \omega$ and a path $S_1$ in $d_1$. There are two cases.

i) $\bar{\alpha}(2)\notin S_1$.

Take $u\in S_1$ with length $lh(\alpha)$. Then $\neg(u\preccurlyeq_1\bar{\alpha}(2))$ and $f(u,x)=\widehat{K}^1$.

ii) $\bar{\alpha}(2)\in S_1$.

Denote $S'_1=\left\lbrace extension(v)\mid v\in S_1\right\rbrace $. Then $S'_1$ is a path in $M$ through $\alpha'$ and by (\ref{eq:choice1}), for some $\delta\in S'_1$, $\exists y\left[\delta\Vdash \psi(x,y)\right] $. 
\smallskip

Take $\gamma=min\{\alpha',\delta\}$ and $u=\bar{\gamma}(2)$. Then $u\preccurlyeq_1\bar{\alpha}(2)$, $\gamma=extension(u)$, $u\in S_1$ and $\exists y\left[\gamma\Vdash \psi(x,y)\right] $. By (\ref{eq:shortest1}), there is a unique $\beta$ with $shortest(\gamma,\beta,x)$. Therefore $f(u,x)=q_{\beta, x}$. This completes the proof of $f\in a_2$.
\medskip
\begin{center}
Proof of (\ref{eq:choice2}). 
\end{center}

It is sufficient to prove:
\[ \forall x\left[\alpha'\Vdash \psi(x,f(x)(0))\right].\]

Fix $x\in \omega$. By Lemma \ref{lemma:Beth}.2, it is sufficient to prove:
\[(\forall \text{ path }S_2\text{ through }\alpha')(\exists \gamma\in S_2)\left[\gamma\Vdash \psi(x,f(x)(0) \right] . \]

Consider a path $S_2$ in $M$ through $\alpha'$. Denote $S'_2=\left\lbrace extension(\bar{\sigma}(2))\mid \sigma\in S_2\right\rbrace$. Then $S'_2$ is also a path in $M$ through $\alpha'$. 
\smallskip

By (\ref{eq:choice1}), for some $\delta'\in S'_2$, $\exists y\left[\delta' \Vdash \psi(x,y) \right] $. Denote $\delta=min\{\alpha', \delta'\}$ and 
\smallskip
\\$u=\bar{\delta}(2)$. 
Then $\delta \in S'_2$,  $u\preccurlyeq_1 \bar{\alpha}(2)$, $\delta=extension(u)$ and $\exists y 
\left[\delta\Vdash \psi(x,y) \right]$. 
\medskip

By (\ref{eq:shortest1}), for some $\beta$ we have: $shortest(\delta,\beta,x),\delta\preccurlyeq\beta$ and $f(x)^{[\delta]}=f(u,x)=q_{\beta, x}$. 

Therefore $f(x)(0)^{[\delta]}=f(x)^{[\delta]}(\bar{\delta}(1),0)=q_{\beta, x}(\bar{\delta}(1),0)=$ the smallest $y$, for 
\medskip
\\
which $\beta \Vdash \psi(x,y)$. Then $f(x)(0)^{[\delta]}=y$ and $\beta \Vdash \psi(x,y)$. 
\medskip

So $\delta \Vdash \psi(x,f(x)(0)^{[\delta]})$ and by Lemma \ref{lemma:sub_in_forcing}, 
$\delta \Vdash \psi(x,f(x)(0))$. Since $\delta \in S'_2$, there is $\gamma \in S_2$ with $\bar{\gamma}(2)=\bar{\delta}(2)$ and by (\ref{eq:contraction}), $\gamma\Vdash \psi(x,f(x)(0))$. This completes the proof for case $m=2$.
\smallskip

CASE $m=1$ is similar with $f\in a_1$ defined by:
\begin{displaymath}
f(u,x)=
\begin{cases}
0 \quad \text{ if } lh(u)\geqslant lh(\alpha) \& \neg(u\preccurlyeq_0\bar{\alpha}(1)),\\
\text{the smallest number }y\text{, for which }\beta\Vdash \psi(x,y)\\
\qquad \qquad\text{for the unique } \beta \text{ with }shortest(extension(u),\beta,x) \\
\qquad \text{if } u\preccurlyeq_0\bar{\alpha}(1) \& \exists z\left[extension(u)\Vdash \psi(x,z)\right],\\ 
\text{undefined } \: \text{ otherwise.}
\end{cases}
\end{displaymath} 

CASE $m\geqslant 3$. 

We use the following functions defined for fixed $\beta$ and $x\in \omega$.

$q_{0,\beta, x}=q_{\beta, x}$ as defined in case $m=2$; 

\begin{displaymath}
q_{1,\beta,x}(u,k)=
\begin{cases}
\widehat{K}^1 & \text{if } lh(u)\geqslant lh(\beta) \& \neg(u\preccurlyeq_1\bar{\beta}(2)),\\
q_{0,\beta, x}  & \text{if } u\preccurlyeq_1\bar{\beta}(2),\\
\text{undefined } & \text{otherwise;}
\end{cases}
\end{displaymath} 

\begin{displaymath}
q_{2,\beta,x}(u,k)=
\begin{cases}
\widehat{K}^2 & \text{if } lh(u)\geqslant lh(\beta) \& \neg(u\preccurlyeq_2\bar{\beta}(3)),\\
q_{1,\beta, x}  & \text{if } u\preccurlyeq_2\bar{\beta}(3),\\
\text{undefined } & \text{otherwise;}
\end{cases}
\end{displaymath} 

\begin{center}
\ldots
\end{center}

\begin{displaymath}
q_{m-2,\beta,x}(u,k)=
\begin{cases}
\widehat{K}^{m-2} & \text{if } lh(u)\geqslant lh(\beta) \& \neg(u\preccurlyeq_{m-2}\bar{\beta}(m-1)),\\
q_{m-3,\beta, x}  & \text{if } u\preccurlyeq_{m-2}\bar{\beta}(m-1),\\
\text{undefined } & \text{otherwise.}
\end{cases}
\end{displaymath} 

Finally, we define function $f$:

\begin{displaymath}
f(u,x)=
\begin{cases}
\widehat{K}^{m-1} \qquad \text{ if } lh(u)\geqslant lh(\alpha) \& \neg(u\preccurlyeq_{m-1}\bar{\alpha}(m)),\\
q_{m-2,\beta, x} \text{ for the unique } \beta \text{ with } shortest(extension(u),\beta,x) \\ 
\qquad \qquad \text{ if } u\preccurlyeq_{m-1}\bar{\alpha}(m) \& \exists y\left[extension(u)\Vdash \psi(x,y)\right] ,\\
\text{undefined } \quad \text{ otherwise.}
\end{cases}
\end{displaymath} 

Then $f\in a_m$ and $\forall x[\alpha' \Vdash \psi(x,f(x)\underbrace{(0)\ldots(0)}_{m-1})] $.
\bigskip

\begin{center}
PROOF for $(C2_s)$
\end{center}

It involves similar steps as the proof for $(C1_s)$ but with new notations. 

Denote $\psi(x,G^n)$ the formula $\varphi(x,G^n)$, in which all parameters, except $x$ and $G$, are replaced by objects from corresponding domains. 
\medskip

Denote $\textit{shortest }(\alpha,\beta,x)\leftrightharpoons \\
\alpha\preccurlyeq \beta\& (\exists g\in a_n)\left[\beta\Vdash \psi(x,g)\right]  \& (\forall \gamma \succ\beta)\neg (\exists h\in a_n) \left[\gamma\Vdash \psi(x,h)\right].$

Clearly, 
\begin{equation}
(\exists g\in a_n) \left[\alpha\Vdash \psi(x,g)\right] \Rightarrow \exists! \beta \textit{shortest }(\alpha,\beta,x).
\label{eq:shortest2}
\end{equation}

Fix $\alpha\in M$. Suppose 
\[ \alpha\Vdash\forall x\exists! G^n\psi(x,G). \]

Denote $\alpha'=extension(\bar{\alpha}(m))$. So by (\ref{eq:contraction}),
\begin{equation}
 \alpha'\Vdash\forall x\exists! G^n\psi(x,G).  
\label{eq:choice2.1}
\end{equation}

By (\ref{eq:choice2.1}), for any $\beta$ and $x\in \omega$:

\begin{equation}
(\beta\preccurlyeq\alpha' \vee \alpha'\preccurlyeq \beta)\&(g\in a_n)\&\left[ \beta\Vdash \psi(x,g)\right] \Rightarrow \text{ this } g \text{ is unique}.
\label{eq:unique}
\end{equation}

CASE $m=n+2$. First we construct the proof for this case because it clearly demonstrates the main construction. Other cases will be considered later. 

Thus, $ sort(\varphi)\leqslant n+2$. 

By (\ref{eq:contraction}), it is sufficient to prove: 
\begin{equation}
\alpha'\Vdash \exists F^m\forall x\psi(x,F(x)(0)).
\label{eq:choice2.2}
\end{equation}

For fixed $x\in \omega$ and fixed $\beta$ with $(\beta\preccurlyeq\alpha' \vee \alpha'\preccurlyeq \beta)$ by (\ref{eq:unique}) there is a unique function $Q_{\beta,x}$ such that:

\begin{displaymath}
Q_{\beta,x}(v,k)=
\begin{cases}
\widehat{K}^n \quad \text{ if } lh(v)\geqslant lh(\beta) \& \neg(v\preccurlyeq_n\bar{\beta}(n+1)),\\
\text{the unique }g\in a_n\text{ with }\beta\Vdash \psi(x,g) \\
\qquad \text{if } v\preccurlyeq_n\bar{\beta}(n+1)\& (\exists h\in a_n) \left[\beta\Vdash \psi(x,h)\right] ,\\
\text{undefined } \qquad \text{ otherwise.}
\end{cases}
\end{displaymath} 

Next we construct a function $f$ that can be used in (\ref{eq:choice2.2}). 
\begin{displaymath}
f(u,x)=
\begin{cases}
\widehat{K}^{n+1} \qquad \text{ if } lh(u)\geqslant lh(\alpha) \& \neg(u\preccurlyeq_{n+1}\bar{\alpha}(n+2)),\\
Q_{\beta, x} \text{ for the unique } \beta \text{ with } shortest(extension(u),\beta,x) \\ 
\qquad \qquad \text{ if } u\preccurlyeq_{n+1}\bar{\alpha}(n+2)\&(\exists g\in a_n) \left[extension(u)\Vdash \psi(x,g)\right] ,\\
\text{undefined } \quad \text{ otherwise.}
\end{cases}
\end{displaymath} 

\begin{center}
Proof that $f\in a_{n+2}$.
\end{center}
 
1) First we prove that $f:(d_{n+1}\times\omega)\dashrightarrow a_{n+1}$. 
Suppose $f(u,x)\downarrow$. We need to prove that $f(u,x)\in a_{n+1}$. 

When $f(u,x)=\widehat{K}^{n+1}$, it is obvious. 

Assume $f(u,x)=Q_{\beta,x}$. Then $u\preccurlyeq_{n+1}\bar{\alpha}_{n+2}$,  $shortest(extension(u),\beta,x)$ and
\begin{equation}
(\exists h\in a_n)\left[\beta\Vdash \psi(x,h)\right]. 
\label{eq:choice2.3}
\end{equation}

Since $extension(u)\preccurlyeq \alpha'$ and $extension(u)\preccurlyeq \beta$, we have: 
\[\beta\preccurlyeq \alpha' \vee \alpha'\preccurlyeq \beta.\]

The proof that $Q_{\beta,x}\in a_{n+1}$ consists of three parts.
\smallskip

a) Clearly, $Q_{\beta,x}:(d_n\times\omega)\dashrightarrow a_n$.
\smallskip

b) Monotonicity of $Q_{\beta,x}$.

Suppose $Q_{\beta,x}(v,k)\downarrow$ and $w\preccurlyeq_nv$. Then $lh(v)\geqslant lh(\beta)$. There are two cases.
\smallskip

i) $\neg(v\preccurlyeq_n \bar{\beta}(n+1))$.

Then $\neg(w\preccurlyeq_n \bar{\beta}(n+1))$ and $Q_{\beta,x}(w,k)=\widehat{K}^n=Q_{\beta,x}(v,k)$.
\smallskip

ii) $v\preccurlyeq_n \bar{\beta}(n+1)$.

Then $w\preccurlyeq_n \bar{\beta}(n+1)$. By (\ref{eq:choice2.3}) and (\ref{eq:unique}) both $Q_{\beta,x}(w,k)$ and $Q_{\beta,x}(v,k)$ equal the unique $g\in a_n$, for which $\beta\Vdash \psi(x,g)$.

c) Completeness of $Q_{\beta,x}$.

Consider a path $S_3$ in $d_n$ and $k\in \omega$. Take $v\in S_3$ with length $lh(\beta)$. There are two cases.

i) $\bar{\beta}(n+1)\notin S_3$.

Then $\neg(v\preccurlyeq_n \bar{\beta}(n+1))$ and $Q_{\beta,x}(v,k)=\widehat{K}^n$.
\smallskip

ii) $\bar{\beta}(n+1)\in S_3$.

Then $v=\bar{\beta}(n+1) $ and by (\ref{eq:unique}), (\ref{eq:choice2.3}), $Q_{\beta,x}(v,k)$ equals the unique $g\in a_n$ with $\beta\Vdash \psi(x,g)$.

This completes the proof that $Q_{\beta,x}\in a_{n+1}$ and  $f:(d_{n+1}\times\omega)\dashrightarrow a_{n+1}$.

2) Monotonicity of $f$.

Suppose $f(u,x)\downarrow$ and $v\preccurlyeq_{n+1} u$. Then $lh(v)\geqslant lh(u)\geqslant lh(\alpha)$. There are two cases.

i) $\neg(u\preccurlyeq_{n+1}\bar{\alpha}({n+2}))$.

In this case $\neg(v\preccurlyeq_{n+1}\bar{\alpha}({n+2}))$, so $f(v,x)=\widehat{K}^{n+1}=f(u,x)$.
\smallskip

ii) $u\preccurlyeq_{n+1}\bar{\alpha}({n+2})$.
\smallskip

Then $v\preccurlyeq_{n+1} \bar{\alpha}({n+2})$; there is a unique $\beta$ with $shortest(extension(u),\beta,x)$ and $f(u,x)=Q_{\beta,x}$.
\smallskip

For some $g\in a_n$ we have:  $extension(u)\Vdash \psi(x,g)$. Since $extension(u)\preccurlyeq extension(v)$, we have: $shortest(extension(v),\beta,x)$ and by monotonicity of forcing, $extension(v)\Vdash \psi(x,g)$. So $f(v,x)=Q_{\beta x}$ for the same $\beta$.

3) Completeness of $f$.

Consider $x\in \omega$ and a path $S_4$ in $d_{n+1}$. There are two cases.

i) $\bar{\alpha}({n+2})\notin S_4$.

Take $u\in S_4$ with length $lh(\alpha)$. Then $\neg (u\preccurlyeq_{n+1}\bar{\alpha}({n+2}))$ and $f(u,x)=\widehat{K}^{n+1}$.

ii) $\bar{\alpha}({n+2})\in S_4$.

Denote $S'_4=\left\lbrace extension(v)\mid v\in S_4\right\rbrace $. Then $S'_4$ is a path in $M$ through $\alpha'$ and by (\ref{eq:choice2.1}), for some $\delta\in S'_4$, $(\exists g\in a_n)\left[\delta\Vdash \psi(x,g)\right] $. 
\smallskip

Take $\gamma=min\{\alpha',\delta\}$ and $u=\bar{\gamma}({n+2})$. Then $u\preccurlyeq_{n+1}\bar{\alpha}({n+2})$, $\gamma=extension(u)$, $u\in S_4$ and $(\exists g\in a_n)\left[\gamma\Vdash \psi(x,g)\right] $. By (\ref{eq:shortest2}), there is a unique $\beta$ with $shortest(\gamma,\beta,x)$. Therefore $f(u,x)=Q_{\beta, x}$. This completes the proof of $f\in a_{n+2}$.

\begin{center}
Proof of (\ref{eq:choice2.2}). 
\end{center}

It is sufficient to prove:
\[ \forall x\left[\alpha'\Vdash \psi(x,f(x)(0))\right].\]

Fix $x\in \omega$. By Lemma \ref{lemma:Beth}.2, it is sufficient to prove:
\[(\forall \text{ path }S_5\text{ through }\alpha')(\exists \gamma\in S_5)\left[\gamma\Vdash \psi(x,f(x)(0) \right] . \]

Consider a path $S_5$ through $\alpha'$. Denote $S'_5=\left\lbrace extension(\bar{\sigma}(n+2))\mid \sigma\in S_5\right\rbrace$. Then $S'_5$ is also a path in $M$ through $\alpha'$. 
\medskip

By (\ref{eq:choice2.1}), for some $\delta'\in S'_5$, $(\exists g\in a_n)\left[\delta' \Vdash \psi(x,g) \right] $. Denote $\delta=min\{\alpha', \delta'\}$ 
\smallskip
\\and $u=\bar{\delta}(n+2)$. 
Then $\delta\in S'_5$, $u\preccurlyeq_{n+1} \bar{\alpha}({n+2})$, $\delta=extension(u)$ and $(\exists g\in a_n)\left[\delta \Vdash \psi(x,g)\right]$. 
\medskip

By (\ref{eq:shortest2}), for some $\beta$ we have: $shortest(\delta,\beta,x),\delta\preccurlyeq\beta$ and $f(x)^{[\delta]}=f(u,x)=Q_{\beta, x}$. Since $\delta\preccurlyeq\alpha'$ and $\delta\preccurlyeq\beta$, we have: 
\[\beta\preccurlyeq\alpha' \vee \alpha'\preccurlyeq \beta.\]

Therefore $f(x)(0)^{[\delta]}=f(x)^{[\delta]}(\bar{\delta}(n+1),0)=Q_{\beta, x}(\bar{\delta}(n+1),0)=$ the unique $g\in a_n$, for which $\beta \Vdash \psi(x,g)$, according to (\ref{eq:unique}). Then $f(x)(0)^{[\delta]}=g$ and $\beta \Vdash \psi(x,g)$. 

So $\delta \Vdash \psi(x,f(x)(0)^{[\delta]})$ and by Lemma \ref{lemma:sub_in_forcing}, 
$\delta \Vdash \psi(x,f(x)(0))$. Since $\delta\in S'_5$, 
\smallskip
\\there is $\gamma\in S_5$ with $\bar{\gamma}(n+2)=\bar{\delta}(n+2)$ and by (\ref{eq:contraction}), $\gamma \Vdash \psi(x,f(x)(0))$. This completes the proof for case $m=n+2$.
\smallskip

CASE $m=n+1$ is similar with $f\in a_{n+1}$ defined by:
\begin{displaymath}
f(u,x)=
\begin{cases}
\widehat{K}^n \: \text{ if } lh(u)\geqslant lh(\alpha) \& \neg(u\preccurlyeq_{n}\bar{\alpha}({n+1})),\\
\text{the unique }g\in a_n\text{, for which }\beta\Vdash \psi(x,g)\\ 
\qquad \qquad\text{for the unique } \beta \text{ with }  shortest(extension(u),\beta,x)\\ 
\qquad \text{if } u\preccurlyeq_n\bar{\alpha}(n+1) \& (\exists h\in a_n)\left[extension(u)\Vdash \psi(x,h)\right],\\ 
\text{undefined } \: \text{ otherwise.}
\end{cases}
\end{displaymath} 
\medskip

CASE $m\geqslant n+3$. 

The following functions are defined for fixed $x\in \omega$ and fixed $\beta\in M$ with $(\beta \preccurlyeq \alpha'\vee \alpha' \preccurlyeq \beta)$. 

$Q_{n,\beta, x}=Q_{\beta, x}$ as defined in case $m=n+2$.

\begin{displaymath}
Q_{n+1,\beta,x}(u,k)=
\begin{cases}
\widehat{K}^{n+1} & \text{if } lh(u)\geqslant lh(\beta) \& \neg(u\preccurlyeq_{n+1}\bar{\beta}(n+2)),\\
Q_{n,\beta, x}  & \text{if } u\preccurlyeq_{n+1}\bar{\beta}(n+2),\\
\text{undefined } & \text{otherwise;}
\end{cases}
\end{displaymath} 

\begin{displaymath}
Q_{n+2,\beta,x}(u,k)=
\begin{cases}
\widehat{K}^{n+2} & \text{if } lh(u)\geqslant lh(\beta) \& \neg(u\preccurlyeq_{n+2}\bar{\beta}(n+3)),\\
Q_{n+1,\beta, x}  & \text{if } u\preccurlyeq_{n+2}\bar{\beta}(n+3),\\
\text{undefined } & \text{otherwise;}
\end{cases}
\end{displaymath} 

\begin{center}
\ldots
\end{center}

\begin{displaymath}
Q_{m-2,\beta,x}(u,k)=
\begin{cases}
\widehat{K}^{m-2} & \text{if } lh(u)\geqslant lh(\beta) \& \neg(u\preccurlyeq_{m-2}\bar{\beta}(m-1)),\\
Q_{m-3,\beta, x}  & \text{if } u\preccurlyeq_{m-2}\bar{\beta}(m-1),\\
\text{undefined } & \text{otherwise.}
\end{cases}
\end{displaymath} 

Finally, we define function $f$:

\begin{displaymath}
f(u,x)=
\begin{cases}
\widehat{K}^{m-1} \qquad \text{ if } lh(u)\geqslant lh(\alpha) \& \neg(u\preccurlyeq_{m-1}\bar{\alpha}(m)),\\
Q_{m-2,\beta, x} \text{ for the unique } \beta \text{ with } shortest(extension(u),\beta,x) \\ 
\qquad \qquad \text{ if } u\preccurlyeq_{m-1}\bar{\alpha}(m) \& (\exists g\in a_n)\left[extension(u)\Vdash \psi(x,g)\right] ,\\
\text{undefined } \quad \text{ otherwise.}
\end{cases}
\end{displaymath} 

Then $f\in a_m$ and $\forall x[\alpha' \Vdash \psi(x,f(x)\underbrace{(0)\ldots(0)}_{m-n-1})] $.
\end{proof}

\section{Kripke's schema in the model $\mathcal{B}_s$}
In $ LP $ we consider the following Kripke's schema:

$ (KS) $ \qquad   $\exists G^m \left[ \varphi \equiv \exists x\left( G(x)(0)^{m-1}\neq 0\right) \right] $,
\\where $ \varphi $ is a formula of $L$, $ G^m $ is not a parameter of $\varphi $ and $ m\geqslant max(sort(\varphi),1)$.
\begin{theorem}
\begin{enumerate}
\item  $ LP_s+(C_s)\vdash (KS_s). $ 
\item  $ LP+(C)\vdash (KS). $ 
\item $\mathcal{B}_s \Vdash (KS_s)$.
\end{enumerate}
\label{theorem:Kripke1}
\end{theorem}
\begin{proof}
1. Axiom $(CS1_s)$ implies:
\[ \forall x\exists y \left[y\neq 0 \equiv (\vdash _x\varphi)\right]. \]
Since $sort(\vdash _x\varphi)=sort(\varphi)$, then by axiom $(C1_s)$ there exists $G^m$ such that:
\[\forall x \left[\left(G(x)(0)^{m-1} \neq 0\right) \equiv (\vdash _x\varphi) \right].  \]
Then 
\[\exists x \left[G(x)(0)^{m-1} \neq 0\right] \equiv \exists x(\vdash _x\varphi)\]
and by axiom $(CS3_s)$
\[\exists x \left[G(x)(0)^{m-1} \neq 0\right] \equiv \varphi.\]

2. Follows from part 1.

3. It follows from part 1, Theorem \ref{theorem:choice} and the soundness of the model $\mathcal{B}_s$ for $ LP_s $.
\end{proof}

\section{Other intuitionistic principles in the model $\mathcal{B}_s$}

\subsection{Weak continuity}
\[\text{(WC)    } \forall \mathcal{F}^n \exists x \varphi(\mathcal{F},x) \supset  \forall \mathcal{F}^n \exists x \exists y\forall\mathcal{G}^n [ \mathcal{\bar{G}}(y)=\mathcal{\bar{F}}(y) \supset\varphi(\mathcal{G},x)],\]
where $ \varphi $ is a formula of $ LP $ with $sort(\varphi) \leqslant n$ and $\varphi$ does not have non-lawlike parameters of type $n$ other than $\mathcal{F}^n$.
\smallskip

\begin{theorem}
\begin{enumerate}
\item $LP_s+(LL3_s)\vdash (WC_s)$.
\item $LP+(LL3)\vdash (WC)$.
\item $\mathcal{B}_s \Vdash (WC_s)$.
\end{enumerate}

\label{theorem:WC}
\end{theorem}
\begin{proof}
1. Denote $\psi(x,y,\mathcal{F}^n,\mathcal{G}^n)$ the formula $ \mathcal{\bar{G}}(y)=\mathcal{\bar{F}}(y) \supset\varphi(\mathcal{G},x)$.

By $(LL3_s)$, 
\[\forall \mathcal{F}^n \forall x \left[\varphi(\mathcal{F},x) \supset \exists y \forall \mathcal{G}^n \psi(x,y,\mathcal{F},\mathcal{G}) \right]. \]
Then 
\[\forall \mathcal{F}^n \left[\exists x \varphi(\mathcal{F},x) \supset \exists x\exists y \forall \mathcal{G}^n \psi(x,y,\mathcal{F},\mathcal{G}) \right] \]
and
\[\forall \mathcal{F}^n \exists x \varphi(\mathcal{F},x) \supset \forall \mathcal{F}^n\exists x\exists y \forall \mathcal{G}^n \psi(x,y,\mathcal{F},\mathcal{G}).\]

2. Follows from part 1.

3. It follows from part 1, Theorem \ref{theorem:lawless} and the soundness of the model $\mathcal{B}_s$ for $ LP_s $.
\end{proof}

\subsection{Bar induction}
\begin{multline*}
\text{(BI)   } \forall F^1 \exists x \varphi(\bar{F}(x))\wedge 
\forall x\forall y [\varphi(y)\supset \varphi(y*\left\langle x\right\rangle ) ]
\\
\wedge 
\forall y[\forall x\psi(y*\left\langle x\right\rangle)\supset\psi(x)]\wedge
\forall y[\varphi(y)\supset\psi(y)]\supset\psi(0).
\end{multline*}
\begin{theorem}
$\mathcal{B}_s \Vdash (BI_{s}) $.
\label{theorem:BI}
\end{theorem}

\begin{proof}
The proof is the same as for Van Dalen's model \cite{vand78}.
\end{proof}

\subsection{Markov principle}
This principle is often considered in constructive mathematics.

Markov Principle:
\[\text{(MP)   } \forall x\left[\varphi(x)\vee \neg\varphi(x)\right] \wedge \neg\neg\exists x\varphi(x) \supset \exists x \varphi(x).\]

Markov Rule:

\begin{center}
(MR) If $\forall x\left[\varphi(x)\vee \neg\varphi(x)\right]$ and $\neg\neg\exists x\varphi(x)$, then $\exists x \varphi(x)$.
\end{center}

\begin{theorem}
\begin{enumerate}
\item The Markov Rule (MR) does not hold in $\mathcal{B}_s$.
\item The Markov Principle (MP) does not hold in $\mathcal{B}_s$.
\end{enumerate}
\label{theorem:MP}
\end{theorem}

\begin{proof}
1. Denote $\psi(x,F^1)\leftrightharpoons (\exists k<x)(F(k)>0)$. For $u\in d_0, n\in \omega $, define:
\begin{displaymath}
f(u,n) =
\begin{cases}
\langle\left\langle u\right\rangle_0\rangle_n  & \text{ if } n<lh(u),\\
\text{undefined}& \text{otherwise}.
\end{cases}
\end{displaymath} 

Then $f\in a_1$. Denote $\varphi(x)\leftrightharpoons \psi(x,f)$. We need to check that the premises of $(MR)$ are forced in the model:
\begin{equation}
\mathcal{B}_s\Vdash \forall x\left[\varphi(x)\vee \neg \varphi(x) \right], 
\label{eq:MR1}
\end{equation}
\begin{equation}
\mathcal{B}_s\Vdash \neg\neg\exists x\varphi(x)
\label{eq:MR2}
\end{equation}
but the conclusion is not forced:
\begin{equation}
\neg\left( \mathcal{B}_s\Vdash \exists x\varphi(x)\right) . 
\label{eq:MR3}
\end{equation}

Proof of (\ref{eq:MR1}). We need to prove:
\[ (\forall x\in \omega)(\forall \textit{ path } S\textit{ in }M) (\exists\alpha \in S)\left[\alpha\Vdash\varphi(x) \textit{  or  }\alpha\Vdash\neg\varphi(x)\right].\]

Fix $x\in \omega$ and a path $S$ in $M$. Choose $\alpha\in S$ with $lh(\alpha)=x$. Then for any $n<x, f^{[\alpha]}(n)\downarrow$.

If $(\exists n<x)\left( f^{[\alpha]}(n)>0\right) $, then $\alpha\Vdash\varphi(x)$. Otherwise $(\forall n<x)\left(f^{[\alpha]}(n)=0\right)$, which implies $\alpha\Vdash \neg\varphi(x)$.

Proof of (\ref{eq:MR2}).
It is sufficient to prove:
\[\forall \alpha\exists x(\exists y<x)(\exists \beta\preccurlyeq \alpha)\left( \beta\Vdash f(y)>0\right) .\]

Fix $\alpha\in M$. Denote $n=lh(\alpha)$ and choose 
\[ x=n+1, y=n, \beta=\left\langle
\left\langle \alpha\right\rangle_0*\left\langle 1\right\rangle,
\left\langle \alpha\right\rangle_1*\left\langle \widehat{K}^1\right\rangle,\ldots,
\left\langle \alpha\right\rangle_{s-1}*\left\langle \widehat{K}^{s-1}\right\rangle\right\rangle.\]
Then $\beta\preceq\alpha, f^{[\beta]}(y)=\left\langle\left\langle\beta\right\rangle_0\right\rangle_n=1$ and $\beta\Vdash f(y)>0$.

Proof of (\ref{eq:MR3}).
For any $n\geqslant 0$ we define a function $g_n:\omega \rightarrow a_n$ as follows:
$g_0=\lambda m.0$ and for $n\geqslant 1, g_n=\lambda m.\widehat{K}^{n}$. We take 
\[S=\left\lbrace <\bar{g}_0(m),\ldots,\bar{g}_{s-1}(m)>|m\in \omega\right\rbrace.\]
Then $S$ is a path in $M$. For any $\alpha\in S$ and $x=lh(\alpha)$ we have: 
\[\textit{for any } y<x, f^{[\alpha]}(y)=f(\bar{\alpha}(1),y)=\langle \langle\alpha\rangle_0\rangle_y=g_0(y)=0.\]
This implies (\ref{eq:MR3}).
 
2. Follows from part 1. 
\end{proof}

\subsection{Formal Church thesis}

\[ \text{(CT)    } \forall x\exists y\varphi(x,y) \supset \exists e \forall x\exists y[\left\lbrace e\right\rbrace (x)=y \wedge \varphi(x,y)].\]

\begin{theorem}
$LP+(C1)+(CT)$ is inconsistent.
\label{theorem:CT}
\end{theorem}
\begin{proof}
Van Dalen \cite{vand78} showed that $HA+(C1)+(CT)$ is inconsistent, which implies the theorem.
\end{proof}

\section{Intuitionistic theory $SLP$ and classical theory $TI$}
By using reverse mathematics approach we combine in one axiomatic theory all intuitionistic principles that were shown to hold in the model $\mathcal{B}_s$:
\[SLP=LP+(LL)+(C)+(BI).\]

In the name of the theory $SLP$, $S$ stands for strong, $L$ for lawless and $P$ for proof (relating to the "creating subject"). The proof-theoretical strength of $SLP$ will be investigated in section 14; in particular we will show that $SLP$ is stronger than the second order arithmetic. In this section we compare $SLP$ with a classical theory $TI$.

\subsection{Theories $TI$ and $TI_s$}
Theory $TI$ is a subsystem of typed set theory with arithmetic at the bottom level. The language of $TI$ has variables of type $n$: $x^n,y^n,z^n,\ldots$ for $n=0,1,2,\ldots$, a constant 0 of type 0, functional symbols $S, +, \cdot$, and predicate symbols $=_n$ and $\in_n (n\geqslant 0)$.

Terms are defined recursively as follows.
\begin{enumerate}
\item Any variable of type $n$ is a term of type $n$.
\item The constant 0 is a term of type 0.
\item If $t$ and $\tau$ are terms of type 0, then $St, t+\tau$, and $t\cdot\tau$ are terms of type 0. 
\end{enumerate}

Atomic formulas:
\begin{list}{}{}
\item $ t=_nr$, 
\item $ t\in_n\tau $, where $t$ and $r$ are terms of type $n$, and $\tau$ is a term of type $n+1$.
\end{list}

Formulas are constructed from atomic formulas using logical connectives and quantifiers.

For a formula $ \varphi $, its sort $srt(\varphi)$ is the maximal type of parameters in $ \varphi $. 

The theory $TI$ has classical predicate logic $CPC$ with equality axioms and the following non-logical axioms.

\begin{enumerate}
\item $ Sx^0 \neq 0$, \qquad $Sx^0=Sy^0\supset x=y. $
\smallskip
\item $x^0+0=x$, \qquad $x^0+Sy^0=S(x+y)$.
\smallskip
\item $x^0\cdot 0=0$, \qquad $x^0\cdot Sy^0=x\cdot y+x$.
\smallskip
\item Induction for natural numbers: $ \varphi(0) \wedge \forall x^0 \left(\varphi(x)\supset \varphi(Sx)\right) \supset \forall x^0 \varphi(x).$
\medskip
\item Comprehension axiom: 
$ \exists x^{n+1} \forall z^n \left(z\in x\equiv \varphi(z) \right), $
where $srt(\varphi)\leqslant n+1$ and $x^{n+1}$ is not a parameter of $\varphi$.
\medskip
\item Extensionality axiom:
$\forall z^n \left(z\in x^{n+1}\equiv z\in y^{n+1} \right)\supset x=y.$ 
\end{enumerate}

McNaughton {\cite{mcnt53}} considered two typed set theories $T$ and $I$, which differ from $TI$ only in the comprehension axiom: in $T$ this axiom is stated without restrictions on the formula $\varphi$ and in $I$ the formula $\varphi$ in the comprehension  axiom has no types greater than $n+1$.

Thus, all three theories $T, TI$ and $I$ have impredicative comprehension axioms; this axiom is the least predicative in $T$ and most predicative in $I$.

In \cite{mcnt53} the variables in $T$ and $I$ have types $s=1,2,3,\ldots$ with type 1 at the bottom level. Here we change them to $s=0,1,2,\ldots$ with type 0 at the bottom level to make the notations consistent with ours.

We denote $T_s$, $TI_s$ and $I_s$ the fragments of $T$, $TI$ and $I$, respectively, in which the types of variables are not greater than $s$ $(s=0,1,2,\ldots)$. 

\subsection{Interpretation of $TI$ in $TI$ without extensionality}

First we define a formula $x^n\approx_n y^n$ of the language $TI$ by induction on $n$. 
\medskip

$x^0\approx_0 y^0$ is $x=y$;
\medskip

$x^{n+1}\approx_{n+1} y^{n+1}$ is $(\forall z^n\in x)(\exists u^n \in y)(z\approx_n u)\wedge (\forall z^n\in y)(\exists u^n \in x)(z\approx_n u)$.
\smallskip

Clearly, $\approx_{n}$ is an equivalence relation.

For any formula $\varphi$ of $TI$ we define $\varphi^*$ by induction on the complexity of $\varphi$.
\medskip

$(t =_n\tau)^*$ is $t \approx_n\tau$;
\medskip

$(t\in_n\tau)^*$ is $(\exists z^n\in \tau)(z\approx_n t)$;
\medskip

$(\bot)^*$ is $\bot$;
\medskip

$(\psi\theta \chi)^*$ is $\psi^*\theta \chi^*$, where $\theta$ is a logical connective $\wedge,\vee$ or $\supset$;
\medskip

$(Q x^n\psi)^*$ is $Q x^n\psi^*$, where $Q$ is a quantifier $\forall$ or $\exists$.
\medskip

Denote $TI_s^*$ and $TI^*$ the theories $TI_s$ and $TI$, respectively, without the extensionality axiom.

\begin{lemma}
\[TI_s\vdash\varphi \Rightarrow TI_s^*\vdash\bar{\varphi}^*,\]
where $\bar{\varphi}$ is a closure of formula $\varphi$.
\label{lemma:extensionality}
\end{lemma}

\begin{proof}
The proof is by induction on the length of the derivation of $\varphi$.

For logical axioms and equality axioms the proof is obvious from the definition of $\varphi^*$. Therefore
\begin{equation}
CPC_s^+\vdash\varphi \Rightarrow CPC_s^+\vdash\varphi^*,
\label{eq:extensionality}
\end{equation}
where $CPC_s^+$ is the classical predicate calculus with equality in the language of $TI_s$.

Of non-logical axioms we will consider the only non-trivial case of comprehension axiom:
\[\exists x^{n+1}\forall z^n(z\in x \equiv \psi(z)), \text{ where } srt(\psi)\leqslant n+1.\]

Since $CPC^+_s\vdash z^n=_nu^n \wedge \psi(z) \supset \psi(u)$, by (\ref{eq:extensionality}) we have:
\[CPC^+_s\vdash z^n\approx _n u^n \wedge \psi^*(u) \supset \psi^*(z).\]

Clearly, $srt(\psi^*)=srt(\psi)$, so by the comprehension axiom in $TI_s^*$, there exists $x^{n+1}$ such that 
\[\forall z^n\left[ z\in x \equiv \psi^*(z)\right] .\]
It remains to prove:
\[\forall z^n \left[ (z\in x)^* \equiv \psi^*(z)\right] .\]
If $(z\in x)^*$, then for some $u^n \in x, z\approx_n u$. So $\psi^*(u)$ and by (\ref{eq:extensionality}), $\psi^*(z)$. The part $\psi^*(z)\supset (z\in x)^*$ is obvious.
\end{proof}

\subsection{Interpretation of $TI^*$ in $SLP$}
First for each term $t$ of the language $TI$ we define a term $t'$ of the language $SLP$. 
If $t$ is a numerical term, then $t'$ is obtained from $t$ by replacing each variable $x_i^0$ by $x_i$. For $n\geqslant 1, (x_i^n)'=F_i^n$.

Next for each formula $\varphi$ of $TI$ we define a formula $\varphi'$ of $SLP$ by induction on the complexity of $\varphi$.
\smallskip

$(t=_n\tau)'$ is $t'=_n\tau'$;
\medskip

$(t\in_0\tau)'$ is $\exists y\left[\tau'(y)=S(t') \right]$;
\medskip

for $n\geqslant 1$, $(t\in_n\tau)'$ is $\exists y\left[\tau'(y)=N^n(t') \right]$;
\smallskip

$\bot'$ is $\bot$;
\smallskip

$(\psi\theta\chi)'$ is $\psi'\theta\chi'$, where $\theta$ is a logical connective $\wedge,\vee$ or $\supset$;
\medskip

$(Qx^n\psi)'$ is $Q(x^n)'\psi'$, where $Q$ is a quantifier $\forall$ or $\exists$.

Finally, for each formula $\varphi$ of $SLP$ we define a formula $\varphi^-$ of $SLP$ by induction on the complexity of $\varphi$.

If $\varphi$ is an atomic formula, then $\varphi^-$ is $\neg\neg\varphi$;

$\bot^-$ is $\bot$;
\smallskip

$(\psi\theta\chi)^-$ is $\psi^-\theta\chi^-$, where $\theta$ is a logical connective $\wedge$ or $\supset$;
\medskip

$(\psi\vee\chi)^-$ is $ \neg(\neg\psi^-\wedge\neg\chi^-)$;
\medskip

$(\forall X\psi)^-$ is $\forall X\psi^-$ and 
\medskip

$(\exists X\psi)^-$ is $\neg\forall X\neg\psi^-$, 
where $X$ is any variable of $SLP$.

\begin{lemma}
\[TI_s^*\vdash \varphi \Rightarrow SLP_s\vdash (\bar{\varphi}')^{-}.\] \label{lemma:int0}
\end{lemma}
\begin{proof}
Proof is similar to the proof in \cite{kash89}.
\end{proof}

Denote $int(\varphi)=((\varphi^*)')^-$. 

\begin{theorem}
\begin{enumerate}
\item $TI_s\vdash \varphi \Rightarrow SLP_s\vdash int(\bar{\varphi}).$
\item $TI\vdash \varphi \Rightarrow SLP\vdash int(\bar{\varphi}).$
\end{enumerate}

\label{theorem:int}
\end{theorem}
\begin{proof}
1. It follows from Lemmas \ref{lemma:extensionality} and \ref{lemma:int0}.

2. Follows from part 1.
\end{proof}

Formulas of the first order arithmetic can be considered the same in $TI$ and $SLP$ by identifying each variable $x_i^0$ of $TI$ with the variable $x_i$ of $SLP$.

\begin{lemma}
Suppose $\varphi$ is a closed formula of first order arithmetic of the form $\forall n\psi(n)$, where formula $\psi$ defines a primitive recursive predicate. Then 
$SLP_0\vdash int(\varphi)\equiv \varphi.$
\label{lemma:int}
\end{lemma}

\begin{proof}
Since $\psi$ is an arithmetic formula, then $\psi^*=\psi$ and $\psi'=\psi$. Since $\psi$ defines a primitive recursive predicate, then $\psi^-\equiv\psi$. Therefore $int(\varphi)\equiv \forall n \;int(\psi)\equiv \forall n \psi \equiv \varphi$.
\end{proof}

\section{The relative strengths of theories $TI_s$ and $TI$}
In \cite{bern76} and \cite{bern78} Bernini intended to create a strong intuitionistic theory of functionals but the theory in \cite{bern76} happened to be inconsistent and the consistency of its modification in \cite{bern78} was not proved. We proved the consistency of our intuitionistic theory $SLP$ of functionals by constructing the model $\mathcal{B}_s$. Now we are interested in the strength of $SLP$. We will show that it is the same as the strength of the classical theory $TI$. We will also investigate the strengths of the fragments $SLP_s$ and $TI_s$.

\subsection{Constructing a truth predicate for $TI_s$ in $TI_{s+1}$}
Fix $s\geqslant 1$. We will construct in $TI_{s+1}$ a truth predicate $Tr$ for formulas of the theory $TI_s$. In this subsection details of the construction will be described and its applications to the relative strengths of our theories will be given in next sub-sections.

First we introduce some notations in the language $TI$. We will often use letters $i,j,k,l,m,n,p,q,r$ for natural numbers, in particular instead of variables of type 0 in the language $TI$.

For a natural number $k, \check{k}=\underbrace{S\ldots S}_k0$ is a term of the language $TI$. 

$\left\lbrace x^k\right\rbrace ^1=\left\lbrace x\right\rbrace$;
$\left\lbrace x^k\right\rbrace ^{n+1}=\left\lbrace \left\lbrace x\right\rbrace ^n\right\rbrace$.

For any $n=0,1,\ldots,s$ we define an ordered pair $\left[ x^n,y^n \right]$ by induction on $n$. An ordered pair $\left[x^0,y^0\right]$ of natural numbers is coded by a natural number; the code is given by a primitive recursive function as usual in the first order arithmetic. 
\[\left[x^{n+1},y^{n+1}\right]=\left\lbrace \left[\left\lbrace 0\right\rbrace ^n,z^n\right]\mid z\in x\right\rbrace \cup 
\left\lbrace \left[\left\lbrace 1\right\rbrace ^n,z^n\right]\mid z\in y\right\rbrace.\]

Thus, the ordered pair $\left[x^n,y^n\right]$ of objects of type $n$ also has type $n$.

Next we define a finite sequence $\langle x_1^n,\ldots,x_k^n\rangle$ by induction on $n$. 

$\langle x_1^0,\ldots,x_k^0\rangle$ is coded by a natural number as usual in the first order arithmetic.
\[\langle x_1^{n+1},\ldots,x_k^{n+1}\rangle=
\left\lbrace \left[ \left\lbrace 0\right\rbrace ^n,\left\lbrace \check{k}\right\rbrace ^n\right] \right\rbrace\cup 
\bigcup_{i=1}^k \left\lbrace \left[\left\lbrace \check{i}\right\rbrace ^n,z^n\right]\mid z\in x_i\right\rbrace . \]

For arbitrary types $n_1,\ldots,n_k$ and $m=max\left\lbrace n_1,\ldots,n_k \right\rbrace$ we define:
\[\langle x_1^{n_1},\ldots,x_k^{n_k}\rangle=\langle \left\lbrace x_1\right\rbrace ^{m-n_1},\ldots,\left\lbrace x_k\right\rbrace ^{m-n_k}\rangle.\]

We will fix G\"{o}del numbering of all expressions of the language $TI_s$. 

For any expression $g$ we denote $\llcorner g\lrcorner$ its G\"{o}del number in this numbering. $t_n$ denotes the arithmetic term of $TI_s$ with G\"{o}del number $n$ and $\varphi_n$ denotes the formula of $TI_s$ with G\"{o}del number $n$; $t_n$ or $\varphi_n$ is undefined if $n$ is not the G\"{o}del number of an arithmetic term or a formula,  respectively.

Next we introduce a predicate $Eval_k(f^k), k=0,1,\ldots,s$, with the meaning: "$f$ is an evaluation of all variables of type $k$", that is $f$ defines an infinite sequence of objects, where the $i$th element evaluates variable $x_i^k$ $(i=1, 2,\ldots)$. A functional symbol $Val_k(f,i)$ is defined next; it produces the $i$th element of $f$.

For $k=1, 2,\ldots,s$
\[Eval_k(f^k)\leftrightharpoons \left(\forall x^{k-1}\in f \right)\left[x=\emptyset^{k-1} \vee \exists i, z^{k-1}(x=\langle i,z\rangle)\right]; \]
\[Val_k(f^k,i)=\left\lbrace z^{k-1}\mid \langle i,z\rangle \in f\right\rbrace  .\]
In the special case $k=0$ the parameter of the formula $Eval_0$ has type 1:
\[Eval_0(f^1)\leftrightharpoons \left(\forall x^0\in f \right)\exists i, n (x=\langle i,n\rangle)\wedge \forall i\exists ! n(\langle i,n\rangle \in f).\]
Thus, $Eval_0(f^1)$ means $f:\omega\rightarrow\omega$. Clearly, 
\[TI_{s+1}\vdash \forall i\exists ! n\left[(Eval_0(f^1)\supset\langle i,n\rangle \in f)\wedge(\neg Eval_0(f^1)\supset n=0)\right] .\]
So we can introduce a functional symbol $Val_0(f^1,i)$ such that 
\[TI_{s+1}\vdash Eval_0(f^1)\supset \langle i,Val_0(f,i)\rangle \in f.\]
Denote $\bar{f}$ the list of variables $f_0^1,f_1^1,f_2^2,\ldots,f_s^s$ and 
\[Eval(\bar{f})\leftrightharpoons \bigwedge_{k=0}^s Eval_k(f_k).\] 
For $k=1, 2,\ldots,s$ denote
\[Subst_k(f^k,i,y^k)=\left\lbrace \langle j,z^{k-1}\rangle \in f \mid j\neq i \right\rbrace 
\cup \left\lbrace \langle i,z^{k-1}\rangle \mid z\in y \right\rbrace.\]
Then 
\begin{multline}
TI_{s+1}\vdash Eval_k(f^k)\wedge g^k=Subst_k(f,i,y^k)
\\
\supset Eval_k(g)\wedge Val_k(g,i)=y \wedge (\forall j\neq i)\left[Val_k(g,j)=Val_k(f,j) \right].
\label{eq:term0}
\end{multline}

Similarly, denote
\[Subst_0(f^1,i,y^0)=\left\lbrace \langle j,z^0\rangle \in f \mid j\neq i \right\rbrace 
\cup \left\lbrace \langle i,y\rangle\right\rbrace;\]
then it has a property similar to (\ref{eq:term0}).

Thus, $Subst_k(f,i,y)$ is the new evaluation obtained from evaluation $f$ by changing the value for variable $x_i^k$ to $y$.

For any $k=0,1,\ldots,s$, denote $Sub_k(\bar{f},i,y^k)$ the list of variables obtained from list $\bar{f}$ by replacing $f_k$ with $Subst_k(f_k,i,y^k)$. 
\begin{lemma}
$TI_{s+1}\vdash Eval(\bar{f})\supset Eval(Sub_k(\bar{f},i,k)), k=0,1,\ldots,s.$
\label{lemma:term1}
\end{lemma}
\begin{proof}
Follows from the definitions of $Eval$ and $Sub_k$.
\end{proof}

We denote $Term(n)$ the arithmetic formula stating that $n$ is the G\"{o}del number of an arithmetic term  of $TI_s$. We denote $Subterm(m,n)$ the arithmetic formula stating that $m$ is the G\"{o}del number of a sub-term of the term with G\"{o}del number $n$. 

In next few steps we define a functional symbol $Arval(n,f^1)$ that produces the value of the arithmetic term $t_n$ under evaluation $f^1$.

First we define a predicate $Arterm(n,f^1, a^1)$ which means that $a$ is the set of the values of all sub-terms of the term $t_n$ under evaluation $f$.
\begin{multline*}
Arterm(n,f^1, a^1)\leftrightharpoons 
Term(n)\wedge Eval_0(f)
\\
\wedge (\forall x^0\in a)\exists m,k\left[x=\langle m,k\rangle \wedge Subterm(m,n)\right] 
\\
\wedge
\forall m\left[Subterm(m,n)\supset B(m,a)\right], 
\end{multline*}
where $B(m,a)\leftrightharpoons B_1\vee B_2\vee B_3\vee B_4\vee B_5$ and the predicates $B_1 - B_5$ determine the values of sub-terms of $t_n$ as follows.
\[B_1\leftrightharpoons m=\llcorner 0\lrcorner \wedge
\forall k (\langle m,k\rangle \in a\equiv k=0).\]
\[B_2\leftrightharpoons m=\llcorner x_i^0\lrcorner \wedge
\forall k\left[\langle m,k\rangle \in a\equiv k=Val_0(f,i)\right] .\]
\[B_3\leftrightharpoons \exists j \left\lbrace m=\llcorner St_j\lrcorner 
\wedge\forall k\left[\langle m,k\rangle \in a\equiv 
\exists l(\langle j,l\rangle \in a \wedge k=Sl)
\right]\right\rbrace  .\]
Denote + as $\theta_4$ and $\cdot$ as $\theta_5$. For $p=4,5$,
\begin{multline*}
B_p\leftrightharpoons \exists i,j \left\lbrace m=\llcorner t_i \theta_p t_j\lrcorner 
\right.
\\ 
\left.
\wedge\forall k\left[\langle m,k\rangle \in a\equiv 
\exists r,l(\langle i,r\rangle \in a \wedge \langle j,l\rangle \in a \wedge k=r \theta_p l)
\right]\right\rbrace  .
\end{multline*}

Denote 
\[a_{\mid\mid m}=\lbrace\langle j,k\rangle\in a\mid Subterm(j,m)\rbrace.\]

\begin{lemma}
The following formulas are derived in $TI_{s+1}$.
\begin{enumerate}
\item $Subterm(m,n)\wedge Arterm(n,f^1,a^1)\supset Arterm(m,f,a_{\mid\mid m})$.
\medskip
\item $Arterm(n,f^1,a^1)\wedge Arterm(n,f,b^1)\supset a=b$.
\medskip
\item $Term(n)\wedge Eval_0(f^1)\supset\exists ! a^1 Arterm(n,f,a)$.
\end{enumerate}
\label{lemma:term2}
\end{lemma}
\begin{proof}
1. Assume the premises. For any $i$ with $Subterm(i,m)$, we have:
\begin{equation}
B(i,a) \text{ and }\forall k\left(\langle i,k\rangle\in a\equiv\langle i,k\rangle\in a_{\mid\mid m} \right).
\label{eq:term1}
\end{equation}

To prove $Arterm(m,f,a_{\mid\mid m})$, it is sufficient to show for any $j$:
\begin{equation}
Subterm(j,m)\supset B(j,a_{\mid\mid m}).
\label{eq:term2}
\end{equation}

When $t_j$ is a constant or a variable, (\ref{eq:term2}) follows from $B(j,a)$. 

Suppose $j=\llcorner t_p+t_q\lrcorner$. Since $B(j,a)$, we have for any $k$:
\[\langle j,k\rangle\in a
\equiv\exists r,l(\langle p,r\rangle\in a\wedge\langle q,l\rangle\in a\wedge k=r+l)\]
and by (\ref{eq:term1}):
\[\langle j,k\rangle\in a_{\mid\mid m}
\equiv\exists r,l(\langle p,r\rangle\in a_{\mid\mid m}\wedge\langle q,l\rangle\in a_{\mid\mid m}\wedge k=r+l),\]
which implies $B(j,a\mid\mid_m)$.

For other forms of $t_j$ the proof is similar.

2. Assume $Arterm(n,f^1,a^1)\wedge Arterm(n,f,b^1)$. Then $Term(n)$ and $Eval_0(f)$. It is sufficient to prove:
\begin{equation}
\forall m\left[Subterm(m,n)\supset \forall k(\langle m,k\rangle\in a\equiv\langle m,k\rangle\in b)\right].
\label{eq:term3}
\end{equation}

We will prove it by induction on $m$. The cases of constant 0 and variables are obvious.

Assume (\ref{eq:term3}) holds for numbers less than $m$. We will consider only the case $m=\llcorner t_i+t_j\lrcorner$, other cases are similar. By the inductive assumption, 
\[\forall k(\langle i,k\rangle\in a\equiv\langle i,k\rangle\in b)\text{ and }
\forall k(\langle j,k\rangle\in a\equiv\langle j,k\rangle\in b).\]
Therefore for any $k$,
\begin{multline*}
\langle m,k\rangle\in a\equiv
\exists r,l(\langle i,r\rangle\in a\wedge \langle j,l\rangle\in a\wedge k=r+l)
\\
\equiv
\exists r,l(\langle i,r\rangle\in b\wedge \langle j,l\rangle\in b\wedge k=r+l)\equiv
\langle m,k\rangle\in b.
\end{multline*}

3. The uniqueness was proven in part 2. We prove the existence by induction on $n$.

1) If $n=\llcorner 0\lrcorner$, we take $a^1=\lbrace \langle n,0\rangle\rbrace$. 
\smallskip

2) If $n=\llcorner x_i^0\lrcorner$, we take $a^1=\lbrace \langle n,Val_0(f,i)\rangle\rbrace$.
\smallskip

In each case $Arterm(n,f,a)$ is obvious.

Assume the existence for numbers less than $n$. 

3) $n=\llcorner t_i+t_j\lrcorner$. 

By the inductive assumption, there exist $c^1$ and $d^1$ such that $Arterm(i,f,c)$ and $Arterm(j,f,d)$. We take 
\[a^1=c\cup d\cup\lbrace \langle n,r+l\rangle \mid  \langle i,r\rangle\in c\wedge \langle j,l\rangle\in d\rbrace\] 
and prove $Arterm(n,f,a)$.

First we prove:
\begin{equation}
\forall p\left[Subterm(p,i)\supset\forall k\left(
\langle p,k\rangle\in a\equiv\langle p,k\rangle\in c \right) \right].
\label{eq:term4}
\end{equation}

Assume $Subterm(p,i)$. If $\neg Subterm(p,j)$, then the conclusion is obvious. If $Subterm(p,j)$, then  $Arterm(p,f,c_{\mid\mid p})$ and $Arterm(p,f,d_{\mid\mid p})$ by part 1, $c_{\mid\mid p}=d_{\mid\mid p}$ by part 2 and for any $k$:
\[\langle p,k\rangle\in a\equiv\langle p,k\rangle\in c_{\mid\mid p}\cup d_{\mid\mid p}\equiv\langle p,k\rangle\in c.\]

(\ref{eq:term4}) is proven. Similarly,
\begin{equation}
\forall p\left[Subterm(p,j)\supset\forall k\left(
\langle p,k\rangle\in a\equiv\langle p,k\rangle\in d \right) \right].
\label{eq:term5}
\end{equation}

To prove $Arterm(n,f,a)$ , it is sufficient to show:
\begin{equation}
\forall m \left[Subterm(m,n)\supset B(m,a)\right].
\label{eq:term6}
\end{equation}

Assume $Subterm(m,n)$.

i) $m=n$.

Applying (\ref{eq:term4}) to $p=i$ and (\ref{eq:term5}) to $p=j$, we get for any $k$:
\begin{multline*}
\langle n,k\rangle\in a
\equiv
\exists r,l
\left(\langle i,r\rangle\in c\wedge 
\langle j,l\rangle\in d
\wedge
k=r+l\right)
\\
\equiv
\exists r,l
\left(\langle i,r\rangle\in a\wedge 
\langle j,l\rangle\in a
\wedge
k=r+l\right),
\end{multline*}
which implies $B(n,a)$.

ii) $m<n$.

Then $Subterm(m,i)\vee Subterm(m,j)$. We can assume $Subterm(m,i)$ without loss of generality. Applying (\ref{eq:term4}) to each possible form of $t_m$, we get $B(m,a)$ due to $Arterm(i,f,c)$. 

For example, if $t_m=S(t_q)$, by (\ref{eq:term4}) we have for any $k$:
\begin{multline*}
\langle m,k\rangle\in a\equiv\langle m,k\rangle\in c\equiv\exists l\left(\langle q,l\rangle\in c\wedge  k=Sl\right)
\equiv\exists l\left(\langle q,l\rangle\in a\wedge k=Sl\right),
\end{multline*}
which implies $B(m,a)$.

This completes the proof for case 3). The cases of other functional symbols are similar.
\end{proof}

\begin{lemma}
It is derived in $TI_{s+1}$ that the formula 
\[Subterm(m,n)\wedge Arterm(n,f^1,a^1)\]
implies each of the following formulas:
\begin{enumerate}
\item $m=\llcorner 0\lrcorner\supset\forall k\left(\langle m,k\rangle\in a \equiv k=0\right)$.
\medskip
\item
$m=\llcorner x_i^0\lrcorner \supset\forall k\left[\langle m,k\rangle\in a \equiv k= Val_0(f,i)\right]$.
\medskip
\item
$m=\llcorner S(t_j)\lrcorner
\supset\forall k\left[\langle m,k\rangle\in a \equiv 
\exists l\left(\langle j,l\rangle\in a\wedge k=Sl  \right)\right]$.
\medskip
\item
$m=\llcorner t_i\theta t_j  \lrcorner 
\supset\forall k\left[\langle m,k\rangle\in a 
\equiv \exists r,l\left(\langle i,r\rangle\in a\wedge \langle j,l\rangle\in a\wedge k=r\theta l\right)\right]$,
\\
where $\theta$ denotes $\cdot$ or $+$.
\end{enumerate}
\label{lemma:term3}
\end{lemma}
\begin{proof}
Follows from the definition of $Arterm$.
\end{proof}

Next we define a functional symbol 
\[Art^1(f^1)=\left\lbrace \langle n,k\rangle \mid \exists a^1(Arterm(n,f,a)\wedge \langle n,k\rangle\in a)\right\rbrace.\]

\begin{lemma}
\begin{multline*}
TI_{s+1}\vdash Subterm(m,n)\wedge Arterm(n,f^1,a^1)
\\
\supset\forall k\left[ \langle m,k\rangle\in a
\equiv\langle m,k\rangle\in Art(f)\right].
\end{multline*}
\label{lemma:term4}
\end{lemma}
\begin{proof}
1. Assume the premises. By Lemma \ref{lemma:term2}.1, $Arterm(m,f,a_{\mid\mid m})$. Fix $k$.

$\Rightarrow$ If $\langle m,k\rangle\in a$, then $\langle m,k\rangle\in a_{\mid\mid m}$ and $\langle m,k\rangle\in Art(f)$ by the definition of $Art$.

$\Leftarrow$ Suppose $\langle m,k\rangle\in Art(f)$. Then there exists $b^1$ such that $Arterm(m,f,b)$ and $\langle m,k\rangle\in b$. By Lemma \ref{lemma:term2}.2, $b=a_{\mid\mid m}$. So $\langle m,k\rangle\in a_{\mid\mid m}$ and $\langle m,k\rangle\in a$.
\end{proof}

\begin{lemma}
It is derived in $TI_{s+1}$ that the formula 
\[Eval_0(f^1)\]
implies each of the following formulas.
\begin{enumerate}
\item $\langle\llcorner 0\lrcorner,k\rangle\in Art(f)\equiv k=0$.
\medskip
\item
$\langle \llcorner x_i^0\lrcorner,k\rangle\in Art(f) \equiv k=Val_0(f,i)$.
\medskip
\item
$\langle \llcorner S(t_j)\lrcorner,k\rangle\in Art(f)\equiv 
\exists l\left(\langle j,l\rangle\in Art(f)\wedge k=Sl \right)$.
\medskip
\item
$\langle \llcorner t_i\theta t_j  \lrcorner,k\rangle\in Art(f) 
\equiv \exists r,l\left(\langle i,r\rangle\in Art(f)
\wedge \langle j,l\rangle\in Art(f)
\wedge k=r\theta l \right)$,
\\
where $\theta$ denotes $\cdot$ or $+$.
\end{enumerate}
\label{lemma:term5}
\end{lemma}
\begin{proof}
Assume $Eval_0(f^1)$. By Lemma \ref{lemma:term2}.3, for any $n$ with $Term(n)$, there is a set $a^1$ such that $Arterm(n,f,a)$, and by Lemma \ref{lemma:term4}:
\begin{equation}
\forall m\left[Subterm(m,n)\supset
\forall k\left( \langle m,k\rangle\in a
\equiv\langle m,k\rangle\in Art(f)\right)\right].
\label{eq:term7}
\end{equation}

1. Denote $n=\llcorner 0\lrcorner$. By (\ref{eq:term7}), we have for any $k$:
\[\langle n,k\rangle\in Art(f)\equiv\langle n,k\rangle\in a\equiv k=0\]
by Lemma \ref{lemma:term3}.1.

2. Denote $n=\llcorner x_i^0\lrcorner$. By (\ref{eq:term7}), we have for any $k$:
\[\langle n,k\rangle\in Art(f)\equiv\langle n,k\rangle\in a\equiv k=Val_0(f,i)\]
by Lemma \ref{lemma:term3}.2.

Parts 3 is similar to part 4, which we will consider next.

4. Denote $n=\llcorner t_i+t_j\lrcorner$. By (\ref{eq:term7}):
\begin{eqnarray*}
\forall k\left( \langle n,k\rangle\in Art(f)\equiv \langle n,k\rangle\in a\right);
\\
\forall k\left( \langle i,k\rangle\in Art(f)\equiv \langle i,k\rangle\in a\right);
\\
\forall k\left( \langle j,k\rangle\in Art(f)\equiv \langle j,k\rangle\in a\right).
\end{eqnarray*}

By Lemma \ref{lemma:term3}.4, for any $k$:
\begin{multline*}
\langle n,k\rangle\in Art(f)\equiv\langle n,k\rangle\in a
\equiv \exists r,l
\left(\langle i,r\rangle\in a
\wedge \langle j,l\rangle\in a\wedge
k=r+l\right)
\\
\equiv \exists r,l
\left(\langle i,r\rangle\in Art(f)
\wedge \langle j,l\rangle\in Art(f)\wedge
k=r+l\right).
\end{multline*}

The case $n=\llcorner t_i\cdot t_j\lrcorner$ is similar.
\end{proof}

\begin{lemma}
The following formulas are derived in $TI_{s+1}$.

1. $\langle n,k\rangle\in Art(f)\wedge\langle n,q\rangle\in Art(f)\supset k=q$.
\begin{multline*}
2. \forall n,f^1\exists !k\left\lbrace \left[ Term(n)\wedge Eval_0(f)\supset \langle n,k\rangle\in Art(f)\right] 
\right.
\\
\left.
\wedge\left[\neg (Term(n)\wedge Eval_0(f))\supset k=0\right] \right\rbrace.
\end{multline*}
\label{lemma:term6}
\end{lemma}
\begin{proof}
1. Proof is by induction on $n$. The cases of constants and variables follow from Lemma \ref{lemma:term5}.1, 2.

Assume the formula holds for numbers less than $n$. We will consider only the case $n=\llcorner t_i+ t_j\lrcorner$, other cases are similar.

Suppose $\langle n,k\rangle\in Art(f)$ and $\langle n,q\rangle\in Art(f)$. By Lemma \ref{lemma:term5}.4, there exist $r_1, l_1$ and $r_2, l_2$ such that 
\[\langle i,r_1\rangle\in Art(f)\wedge
\langle j,l_1\rangle\in Art(f)\wedge
k=r_1+l_1\]
and
\[\langle i,r_2\rangle\in Art(f)\wedge
\langle j,l_2\rangle\in Art(f)\wedge
q=r_2+l_2.\]

By the inductive assumption, $r_1=r_2$ and $l_1=l_2$, so $k=q$.

2. When $\neg (Term(n)\wedge Eval_0(f))$, it is obvious. 

Assume $Term(n)\wedge Eval_0(f)$. The uniqueness was proven in part 1. We prove the existence by induction on $n$.

If $n=\llcorner 0\lrcorner$, then $k=0$.

If $n=\llcorner x_i^0\lrcorner$, then $k=Val_0(f,i)$.

Assume the existence for numbers less than $n$. We will consider only the case $n=\llcorner t_i+t_j\lrcorner$, other cases are similar. By the inductive assumption, there exist $r$ and $l$ such that $\langle i,r\rangle\in Art(f)$ and $\langle j,l\rangle\in Art(f)$. Then for $k=r+l$ we have $\langle n,k\rangle\in Art(f)$ by Lemma \ref{lemma:term5}.4.
\end{proof}

By the previous lemma we can introduce a functional symbol $Arval^0(n,f^1)$ such that 
\[TI_{s+1}\vdash Term(n)\wedge Eval_0(f^1)\supset \langle n,Arval(n,f)\rangle\in Art(f).\]

\begin{lemma}
It is derived in $TI_{s+1}$ that the formula 
\[Eval_0(f^1)\]
implies each of the following formulas.
\begin{enumerate}
\item
$Arval(\llcorner 0\lrcorner,f)=0$.
\medskip
\item $Arval(\llcorner x_i^0\lrcorner, f)=Val_0(f,i)$.
\medskip
\item
$Arval(\llcorner S(t_j)\lrcorner,f)=S(Arval(j,f))$.
\medskip
\item
$Arval(\llcorner t_i\theta t_j\lrcorner,f)=Arval(i,f)\theta Arval(j,f)$,
where $\theta$ denotes $\cdot$ or $+$.
\medskip
\item
$Arval(\llcorner t\lrcorner,f)=\tilde{t}$,
where $t$ is any arithmetic term and $\tilde{t}$ is obtained from $t$ by replacing each parameter $x_i^0$ by $Val_0(f,i)$.
\end{enumerate}
\label{lemma:term7}
\end{lemma}
\begin{proof}
1. Follows from $\langle \llcorner 0\lrcorner,0\rangle\in Art(f)$, which is in Lemma \ref{lemma:term5}.1.

2. Follows from $\langle \llcorner x_i^0\lrcorner,Val_0(f,i)\rangle\in Art(f)$, which is in Lemma \ref{lemma:term5}.2.

3. Denote $l=Arval(j,f)$. Then $\langle j,l\rangle\in Art(f)$ and by Lemma \ref{lemma:term5}.3, $\langle \llcorner S(t_j)\lrcorner,Sl\rangle\in Art(f)$, which implies part 3.

4. We prove it for +, for $\cdot$ the proof is similar. Denote $r=Arval(i,f)$ and $l=Arval(j,f)$. Then $\langle i,r\rangle\in Art(f)$, $\langle j,l\rangle\in Art(f)$ and by Lemma \ref{lemma:term5}.4, $\langle\llcorner t_i+ t_j\lrcorner,r+l\rangle\in Art(f)$, which implies part 4.

5. Proof is by induction on the complexity of $t$. The cases of constant 0 and variables follow from parts 1 and 2, respectively. 

Assume the formula holds for all proper sub-terms of $t$. We will consider only the case $t=\tau+ u$, other cases are similar. By the inductive assumption:
\[Arval(\llcorner \tau\lrcorner,f)=\tilde{\tau}\text{ and }Arval(\llcorner u\lrcorner,f)=\tilde{u}.\]
By part 4, $Arval(\llcorner t\lrcorner,f)=\tilde{\tau}+\tilde{u}=\tilde{t}$.
\end{proof}

We denote $Form(n)$ the arithmetic formula stating that $n$ is the G\"{o}del number of a formula of $TI_s$. We denote $Subform(m,n)$ the arithmetic formula stating that $m$ is the G\"{o}del number of a sub-formula of the formula with G\"{o}del number $n$. 

Next we introduce a predicate $Trset(n,b^{s+1})$, which  means that $b$ is the set of the truth values of all evaluated sub-formulas of $\varphi_n$.
\begin{multline*}
Trset(n,b^{s+1})\leftrightharpoons 
Form(n)\wedge
\\
(\forall x^s\in b)\exists m,\bar{f},\delta^0
\left[x=\langle m,\bar{f},\delta\rangle \wedge Subform(m,n)\wedge Eval(\bar{f})\wedge 
(\delta = 0 \vee \delta = 1)\right] 
\\  
\wedge\forall m\left\lbrace Subform(m,n)\supset \forall \bar{f}\left[ Eval(\bar{f})\supset 
(\langle m,\bar{f},0\rangle \in b \vee \langle m,\bar{f},1\rangle \in b)
\right.
\right.
\\
\left.
\left.
\wedge\neg (\langle m,\bar{f},0\rangle \in b \wedge \langle m,\bar{f},1\rangle \in b)
\wedge A(m,\bar{f},b)\right]\right\rbrace ,
\end{multline*}
\[\text{where }A(m,\bar{f},b)=\bigvee_{r=1}^{10}A_r\text{ and }A_1 - A_{10}\text{ are defined as follows.}\]
\[A_1\leftrightharpoons m=\llcorner\bot\lrcorner \wedge \langle m,\bar{f},0\rangle \in b.\]
\[A_2\leftrightharpoons \bigvee_{k=1}^s \exists i,j \left[m=\llcorner x_i^k=_k x_j^k\lrcorner \wedge 
\left(\langle m,\bar{f},1\rangle \in b \equiv Val_k(f_k,i)=_k Val_k(f_k,j)\right) \right]. \]
\begin{multline*}
A_3\leftrightharpoons 
\bigvee_{k=1}^{s-1} \exists i,j \left[m=\llcorner x_i^k\in_k x_j^{k+1}\lrcorner 
\right.
\\
\wedge 
\left.
\left(\langle m,\bar{f},1\rangle \in b \equiv Val_k(f_k,i)\in_k Val_{k+1}(f_{k+1},j)\right) \right]. 
\end{multline*}
\[A_4\leftrightharpoons \exists i,j \left[m=\llcorner t_i=_0t_j\lrcorner \wedge 
\left(\langle m,\bar{f},1\rangle \in b \equiv Arval(i,f_0)=_0Arval(j,f_0)\right) \right]. \]
\[A_5\leftrightharpoons 
\exists i,j \left[m=\llcorner t_i\in_0 x_j^1\lrcorner
\wedge 
\left(\langle m,\bar{f},1\rangle \in b \equiv Arval(i,f_0)\in_0 Val_1(f_1,j)\right) \right].\] 
Denote the logical connective $\wedge$ as $\theta_6$, $\vee$ as $\theta_7$ and $\supset$ as $\theta_8$. For $p=6, 7, 8$,
\[A_p\leftrightharpoons \exists i,j\left[m=\llcorner \varphi_i \theta_p \varphi_j\lrcorner \wedge \left(\langle m,\bar{f},1\rangle \in b \equiv \left(\langle i,\bar{f},1\rangle \in b \;\theta_p \langle j,\bar{f},1\rangle \in b \right)  \right) 
\right]. \]

Denote $Q_9$ the quantifier $\forall$ and $Q_{10}$ the quantifier $\exists$. For $p=9,10$,
\begin{multline*}
A_p\leftrightharpoons \bigvee_{k=0}^s\exists i,j\left[m=\llcorner Q_p x_i^k \varphi_j\lrcorner 
\right.
\\
\left.
\wedge \left(\langle m,\bar{f},1\rangle \in b \equiv 
Q_py^k(\langle j,Sub_k(\bar{f},i,y),1\rangle \in b )\right)\right].
\end{multline*}

Denote
\[b_{\mid m}=\lbrace\langle j,\bar{f},\delta\rangle\in b\mid Subform(j,m)\rbrace.\]
\begin{lemma}
The following formulas are derived in $TI_{s+1}$.
\begin{enumerate}
\item $Subform(m,n)\wedge Trset(n, b^{s+1})\supset Trset(m, b_{\mid m})$.
\medskip
\item $Trset(n, b^{s+1})\wedge Trset(n, c^{s+1})\supset b=c$.
\medskip
\item $Form(n)\supset\exists !b^{s+1} Trset(n,b)$.
\end{enumerate}
\label{lemma:trset8}
\end{lemma}
\begin{proof}
1. Assume the premises. For any $j$ with $Subform(j,m)$ we have:
\begin{equation}
\forall\bar{f}\left[Eval(\bar{f})\supset A(j,\bar{f},b) \right] 
\label{eq:trset8}
\end{equation}
and 
\begin{equation}
\forall\bar{f},\delta^0\left[\langle j,\bar{f},\delta\rangle\in b\equiv\langle j,\bar{f},\delta\rangle\in b_{\mid m} \right]. 
\label{eq:trset9}
\end{equation}

To prove $Trset(m,b_{\mid m})$, it is sufficient to show for any $j$:
\begin{equation}
Subform(j,m)\supset\forall\bar{f}\left[Eval(\bar{f})\supset A(j,\bar{f},b_{\mid m})\right]. 
\label{eq:trset10}
\end{equation}

When $\varphi_j$ is an atomic formula or $\bot$, (\ref{eq:trset10}) follows from (\ref{eq:trset8}). 

Suppose $j=\llcorner\exists x_i^k  \varphi_p\lrcorner$. Consider $\bar{f}$ with $Eval(\bar{f})$. By (\ref{eq:trset8}),
\[\langle j,\bar{f},1\rangle \in b\equiv
\exists y^k\left(\langle p,Sub_k(\bar{f},i,y),1\rangle \in b\right) \]
and by (\ref{eq:trset9}):
\[\langle j,\bar{f},1\rangle \in b_{\mid m}\equiv
\exists y^k\left(\langle p,Sub_k(\bar{f},i,y),1\rangle \in b_{\mid m}\right),\]
which implies $A(j,\bar{f},b_{\mid m})$. 

In the cases of other quantifiers and logical connectives the proof is similar.

2. Assume the premises. It is sufficient to prove:
\begin{equation}
\forall m\left[Subform(m,n)\supset(\langle m,\bar{f},1\rangle\in b\equiv\langle m,\bar{f},1\rangle\in c)\right]. 
\label{eq:trset11}
\end{equation}

If (\ref{eq:trset11}) is proven, then for any $m,\bar{f}$ we have:
\[\langle m,\bar{f},0\rangle\in b\equiv\langle m,\bar{f},1\rangle\notin b\equiv\langle m,\bar{f},1\rangle\notin c\equiv\langle m,\bar{f},0\rangle\in c,\]
and this implies $b=c$.

We prove (\ref{eq:trset11}) by induction on $m$.

Case 1: $m=\llcorner\bot\lrcorner$. Then $\langle m,\bar{f},1\rangle\notin b$ and $\langle m,\bar{f},1\rangle\notin c$.

Case 2: $m=\llcorner t_i\in_0 x_j^1\lrcorner$. Proof follows from the definition of $A_5$. The cases of other atomic formulas are similar.

Assume (\ref{eq:trset11}) holds for numbers less than $m$.

Case 3: $m=\llcorner\exists x_i^k  \varphi_j\lrcorner$.
By the inductive assumption: 
\[\forall \bar{f}(\langle j,\bar{f},1\rangle\in b\equiv\langle j,\bar{f},1\rangle\in c).\]

So for any $ \bar{f}$:
\begin{multline*}
\langle m,\bar{f},1\rangle\in b
\equiv
\exists y^k
(\langle j,Sub_k(\bar{f},i,y),1\rangle\in b)
\\
\equiv
\exists y^k
(\langle j,Sub_k(\bar{f},i,y),1\rangle\in c)
\equiv\langle m,\bar{f},1\rangle\in c.
\end{multline*}

The cases of other quantifiers and logical connectives are similar.

3. The uniqueness was proven in part 2. We will prove the existence by induction on $n$.

Case 1: $n=\llcorner\bot\lrcorner$. We take $b^{s+1}=\lbrace\langle n,\bar{f},0\rangle\mid Eval(\bar{f})\rbrace.$

Case 2: $n=\llcorner t_i\in_0 x_j^1\lrcorner$. We take 
\begin{multline*}
b^{s+1}=\left\lbrace \langle n,\bar{f},1\rangle\mid Eval(\bar{f})\wedge Arval(i,f_0)\in_0 Val_1(f_1,j)\right\rbrace 
\\
\cup
\left\lbrace \langle n,\bar{f},0\rangle\mid Eval(\bar{f})\wedge Arval(i,f_0)\notin_0 Val_1(f_1,j)\right\rbrace .
\end{multline*}

Clearly $Trset(n,b)$ holds.

Assume the existence for numbers less than $n$.

Case 3: $n=\llcorner\varphi_i\vee \varphi_j\lrcorner$. By the inductive assumption there exist sets $c^1$ and $d^1$ such that $Trset(i,c)$ and $Trset(j,d)$. Denote 
\[\varphi( i,j,\bar{f})\leftrightharpoons \langle i,\bar{f},1\rangle\in c\vee \langle j,\bar{f},1\rangle\in d.\]

We take 
\begin{multline*}
b^{s+1}=c\cup d\cup\left\lbrace \langle n,\bar{f},1\rangle\mid Eval(\bar{f})\wedge\varphi( i,j,\bar{f})\right\rbrace 
\\
\cup
\left\lbrace \langle n,\bar{f},0\rangle\mid Eval(\bar{f})\wedge\neg\varphi( i,j,\bar{f})\right\rbrace .
\end{multline*}

First we prove:
\begin{equation}
\forall p\left[Subform(p,i)\supset \forall \bar{f},\delta^0\left(\langle p,\bar{f},\delta\rangle\in b\equiv\langle  p,\bar{f},\delta\rangle\in c \right) \right].
\label{eq:trset12}
\end{equation}

Assume the premises of (\ref{eq:trset12}). If $\neg Subform(p,j)$, then the conclusion is obvious. If $Subform(p,j)$, then $Trset(p,c_{\mid p})$ and $Trset(p,d_{\mid p})$ by part 1, $c_{\mid p}=d_{\mid p}$ by part 2 and for any $\bar{f},\delta^0$:
\[\langle p,\bar{f},\delta\rangle\in b\equiv
\langle p,\bar{f},\delta\rangle\in c_{\mid p}\cup d_{\mid p} \equiv\langle p,\bar{f},\delta\rangle\in c_{\mid p}\equiv\langle p,\bar{f},\delta\rangle\in c.\]

(\ref{eq:trset12}) is proven. Similarly,
\begin{equation}
\forall p\left[Subform(p,j)\supset \forall \bar{f},\delta^0\left(\langle p,\bar{f},\delta\rangle\in b\equiv\langle  p,\bar{f},\delta\rangle\in d \right) \right].
\label{eq:trset13}
\end{equation}

To prove $Trset(n,b)$, it is sufficient to show:
\begin{multline}
\forall m\left\lbrace Subform(m,n)\supset\forall\bar{f}\left[Eval(\bar{f})
\right.
\right.
\\
\left.
\left.
\supset
\neg(\langle m,\bar{f},0\rangle\in b\wedge\langle m,\bar{f},1\rangle\in b)\wedge A(m,\bar{f},b)\right]\right\rbrace .
\label{eq:trset14}
\end{multline}

Assume $Subform(m,n)$.

i) $m=n$.

Suppose $Eval(\bar{f})$. Then $\neg(\langle n,\bar{f},0\rangle\in b\wedge\langle n,\bar{f},1\rangle\in b)$ follows from the definition of $b$.

Applying (\ref{eq:trset12}) to $p=i$ and (\ref{eq:trset13}) to $p=j$, we get:
\[\langle n,\bar{f},1\rangle\in b\equiv\varphi(i,j,\bar{f})
\equiv\langle i,\bar{f},1\rangle\in c\vee \langle j,\bar{f},1\rangle\in d 
\equiv \langle i,\bar{f},1\rangle\in b\vee \langle j,\bar{f},1\rangle\in b,\]
which implies $A(n,\bar{f},b)$.

ii) $m<n$.

Then $Subform(m,i)\vee Subform(m,j)$. We can assume $Subform(m,i)$ without loss of generality. Due to $Trset(i,c)$ we have $\neg(\langle m,\bar{f},0\rangle\in c\wedge\langle m,\bar{f},1\rangle\in c)$ and by (\ref{eq:trset12}), 
\[\neg(\langle m,\bar{f},0\rangle\in b\wedge\langle m,\bar{f},1\rangle\in b).\] 
 
$A( m,\bar{f},b)$ is also also derived from $Trset (i,c)$ by applying (\ref{eq:trset12}) to each possible form of the formula $\varphi_m$.

This completes the proof for case 3. The cases of other logical connectives and quantifiers are similar.
\end{proof}

\begin{lemma}
It is derived in $TI_{s+1}$ that the formula 
\[ Subform(m,n)\wedge Eval(\bar{f})\wedge Trset(n,b^{s+1})\]
implies each of the following formulas.
\begin{enumerate}
\item $\langle m,\bar{f},1 \rangle\in b \vee \langle  m,\bar{f},0 \rangle\in b$.
\medskip
\item $\neg (\langle m,\bar{f},1 \rangle\in b \wedge \langle m,\bar{f},0 \rangle \in b)$.
\medskip
\item $\langle \llcorner \bot \lrcorner,\bar{f},0 \rangle \in b$.
\medskip
\item $ m=\llcorner x_i^k=_k x_j^k\lrcorner 
\supset 
\left[\langle m,\bar{f},1\rangle\in b \equiv 
 Val_k(f_k,i)=_k Val_k(f_k,j)\right], 
\\ 
 k=1,2,\ldots,s.$
 \medskip
\item
$m=\llcorner x_i^k\in_k x_j^{k+1}\lrcorner 
\supset
\left[\langle m,\bar{f},1\rangle \in b \equiv Val_k(f_k,i)\in_k Val_{k+1}(f_{k+1},j)\right], 
\\
k=1,2,\ldots,s-1.$ 
\medskip
\item
$m=\llcorner t_i=_0 t_j\lrcorner \supset
\left[\langle m,\bar{f},1\rangle \in b \equiv Arval(i,f_0)=_0 Arval(j,f_0)\right]. $ 
\medskip
\item
$m=\llcorner t_i\in_0 x_j^1\lrcorner
\supset
\left[\langle m,\bar{f},1\rangle \in b \equiv Arval(i,f_0)\in_0 Val_1(f_1,j)\right].$ 
\medskip
\item
$m=\llcorner \varphi_i \theta \varphi_j\lrcorner \supset \left[\langle m,\bar{f},1\rangle \in b \equiv \left(\langle i,\bar{f},1\rangle \in b \;\theta \langle j,\bar{f},1\rangle \in b \right)  \right]$,  where $\theta$ is a logical connective $\wedge,\vee$ or $\supset$.
\medskip
\item
$m=\llcorner Q x_i^k \varphi_j\lrcorner \supset \left[\langle m,\bar{f},1\rangle \in b \equiv 
Qy^k(\langle j,Sub_k(\bar{f},i,y),1\rangle \in b )\right],
\\k=0,1,\ldots,s$, where $Q$ is a quantifier $\forall$ or $\exists$.
\end{enumerate}
\label{lemma:trset9}
\end{lemma}
\begin{proof}
Follows from the definition of $Trset$.
\end{proof}

We define
\[Truth^{s+1}=\left\lbrace x^s\mid \exists n,b^{s+1}(Trset(n,b)\wedge x\in b)\right\rbrace. \]
Thus, $Truth$ is the set of the truth values of all evaluated formulas of $TI_s$.

\begin{lemma}
$TI_{s+1}\vdash Subform(m,n)\wedge Trset(n,b^{s+1})
\medskip
\\
\supset\left(\langle m,\bar{f},\delta^0\rangle\in b\equiv \langle m,\bar{f},\delta\rangle\in Truth\right).$
\label{lemma:trset10}
\end{lemma}
\begin{proof}
Assume the premises. Then $Trset(m,b_{\mid m})$.
\medskip

$\Rightarrow$ If $\langle m,\bar{f},\delta\rangle\in b$, then $\langle  m,\bar{f},\delta\rangle\in b_{\mid m}$ and $\langle m,\bar{f},\delta\rangle\in Truth$ by the definition of the set $Truth$.

$\Leftarrow$ Suppose $\langle m,\bar{f},\delta\rangle\in Truth$. Then there is $c^{s+1}$ with 
$Trset(m,c)$ and $\langle m,\bar{f},\delta\rangle\in c$. By Lemma \ref{lemma:trset8}.2, $c=b\mid_m$, so $\langle m,\bar{f},\delta\rangle\in b_{\mid m}$ and $\langle m,\bar{f},\delta\rangle\in b$.
\end{proof}

Finally we introduce the truth predicate $Tr$:
\[Tr(n,\bar{f})\leftrightharpoons \langle n,\bar{f}, 1\rangle \in Truth.\]

Clearly, $Tr$ depends on $s$ but we will omit $s$ for brevity.

\begin{lemma}
It is derived in $TI_{s+1}$ that the formula 
\[Eval(\bar{f})\] 
implies each of the following formulas.
\begin{enumerate}
\item $Form(n)\supset 
\left( \neg Tr(n,\bar{f})\equiv \langle n,\bar{f},0 \rangle \in Truth\right) $.
\item $Tr(\llcorner \perp \lrcorner,\bar{f})\equiv \perp$.
\medskip
\item 
$Tr(\llcorner x_i^k=_k x_j^k \lrcorner,\bar{f}) \equiv Val_k(f_k,i)=_k Val_k(f_k,j),k=1,2,\ldots,s$.
\medskip
\item $Tr(\llcorner x_i^k\in_k x_j^{k+1} \lrcorner, \bar{f})\equiv Val_k(f_k,i)\in_k Val_{k+1}(f_{k+1},j),k=1,2,\ldots,s-1$.
\medskip
\item $Tr(\llcorner t_i=_0 t_j \lrcorner,\bar{f}) \equiv Arval(i,f_0)=_0 Arval(j,f_0)$.
\medskip
\item $Tr(\llcorner t_i\in_0 x_j^1 \lrcorner, \bar{f})\equiv Arval(i,f_0)\in_0 Val_1(f_1,j)$.
\medskip
\item $Tr(\llcorner \varphi_i \theta \varphi_j \lrcorner,\bar{f}) \equiv Tr(i,\bar{f}) \theta Tr(j,\bar{f})$, where $\theta$ is a logical connective $\wedge,\vee$ or $\supset$.
\medskip
\item $Tr(\llcorner Qx^k_i \varphi_j \lrcorner,\bar{f}) \equiv Qy^kTr(j,Sub_k(\bar{f},i,y)),k=0,1,\ldots,s$, where $Q$ is a quantifier $\forall$ or $\exists$.
\medskip
\item $Tr(\llcorner \varphi\lrcorner, \bar{f})\equiv \tilde{\varphi}$, where $\varphi$ is any formula of $TI_s$ and $\tilde{\varphi}$ is obtained from $\varphi$ by replacing each parameter $x_i^k$ by $Val_k(f_k,i)$.
\medskip
\item $Tr(\llcorner \varphi\lrcorner, \bar{f})\equiv \varphi$ for any closed formula $\varphi$ of $TI_s$.
\end{enumerate}
\label{lemma:trset11}
\end{lemma}

\begin{proof}
Assume $Eval(\bar{f})$. By Lemma \ref{lemma:trset8}.3, for any $n$ with $Form(n)$ there is a set $b^{s+1}$ such that $Trset(n,b)$ and by Lemma \ref{lemma:trset10}:
\begin{equation}
\forall m \left[Subform(m,n)\supset\forall \bar{f}(Tr( m,\bar{f})\equiv\langle m,\bar{f},1\rangle\in b) \right].
\label{eq:trset15}
\end{equation}

1. By (\ref{eq:trset15}) for any $n$ with $Form(n)$:
\[\neg Tr( n,\bar{f})\equiv\langle n,\bar{f},1\rangle\notin b\equiv\langle n,\bar{f},0\rangle\in b
\equiv\langle n,\bar{f},0\rangle\in Truth\]
by Lemma \ref{lemma:trset9}.2, 1 and Lemma \ref{lemma:trset10}.

2. Denote $n=\llcorner\bot\lrcorner$. 

By  (\ref{eq:trset15}) and Lemma \ref{lemma:trset9}.3, 2, $Tr( n,\bar{f})\equiv\langle n,\bar{f},1\rangle\in b\equiv\bot$.

Parts 3-5 are similar to part 6, which we consider next.

6. Denote $n=\llcorner t_i\in_0 x_j^1 \lrcorner$. By (\ref{eq:trset15}) and Lemma \ref{lemma:trset9}.7, 
\[Tr( n,\bar{f})\equiv\langle n,\bar{f},1\rangle\in b\equiv Arval(i,f_0)\in_0 Val_1(f_1,j). \] 

7. Denote $n=\llcorner \varphi_i\vee \varphi_j \lrcorner$. By (\ref{eq:trset15}):
\begin{eqnarray*}
\forall \bar{f} \left[Tr( n,\bar{f})\equiv\langle  n,\bar{f},1\rangle\in b\right] ,
\\
\forall \bar{f} \left[Tr( i,\bar{f})\equiv\langle  i,\bar{f},1\rangle\in b\right],
\\
\forall \bar{f} \left[Tr( j,\bar{f})\equiv\langle  j,\bar{f},1\rangle\in b\right].
\end{eqnarray*}

By Lemma \ref{lemma:trset9}.8,
\[Tr(n,\bar{f})
\equiv \langle n,\bar{f},1\rangle \in b\equiv
\langle i,\bar{f},1\rangle\in b \vee\langle j,\bar{f},1\rangle\in b 
\equiv Tr(i,\bar{f})\vee Tr(j,\bar{f}). \]

The cases of other logical connectives are similar.

8. Denote $n=\llcorner \forall x_i^k\varphi_j \lrcorner$. By (\ref{eq:trset15}):
\begin{eqnarray*}
\forall \bar{f} \left[Tr( n,\bar{f})\equiv\langle  n,\bar{f},1\rangle\in b\right] ,
\\
\forall \bar{f} \left[Tr( j,\bar{f})\equiv\langle  j,\bar{f},1\rangle\in b\right].
\end{eqnarray*}

By Lemma \ref{lemma:trset9}.9,
\begin{multline*}
Tr( n,\bar{f})
\equiv \langle n,\bar{f},1\rangle \in b\equiv
\forall y^k\left(\langle j,Sub_k(\bar{f},i,y),1\rangle \in b \right) 
\\
\equiv \forall y^k Tr(j,Sub_k(\bar{f},i,y)).
\end{multline*} 

The case of the quantifier $\exists$ is similar.

9. Proof is by induction on the complexity of $\varphi$.
 
1) If $\varphi$ is $\bot$, then the formula follows from part 2.

2) Suppose $\varphi$ is $t\in_0 x_j^1$. By part 6,
\[Tr(\llcorner\varphi\lrcorner,\bar{f})\equiv Arval(\llcorner t\lrcorner,f_0)\in_0 Val_1(f_1,j)\equiv \tilde{t}\in_0 Val_1(f_1,j)\equiv\tilde{\varphi}\]
by Lemma \ref{lemma:term7}.5.

Assume part 9 holds for all proper sub-formulas of $\varphi$. 

3) Suppose $\varphi$ is $\psi\supset\chi$. By the inductive assumption:
\[Tr(\llcorner\psi\lrcorner,\bar{f})\equiv \tilde{\psi} \text{ and }Tr(\llcorner\chi\lrcorner,\bar{f})\equiv \tilde{\chi}.\]

By part 7:
\[Tr(\llcorner\varphi\lrcorner,\bar{f})\equiv \left[Tr(\llcorner\psi\lrcorner,\bar{f})\supset Tr(\llcorner\chi\lrcorner,\bar{f}) \right] \equiv (\tilde{\psi}\supset \tilde{\chi})\equiv\tilde{\varphi}.\]

The cases of other logical connectives are similar.

4) Suppose $\varphi$ is $\exists x_j^r\psi(x_j)$. By part 8:
\[Tr(\llcorner\varphi\lrcorner,\bar{f})\equiv\exists y^r Tr(\llcorner\psi\lrcorner,Sub_r(\bar{f},j,y)).\] 

By the inductive assumption,
\[Tr(\llcorner\psi\lrcorner,Sub_r(\bar{f},j,y))\equiv \tilde{\psi}(y),\]
where $\tilde{\psi}(y)$ is obtained from $\psi(x_j)$ by replacing $x_j$ by $y$ and each other parameter $x_i^k$ by $Val_k(f_k,i)$. So
\[Tr(\llcorner\varphi\lrcorner,\bar{f})\equiv\exists y^r \tilde{\psi}(y)\equiv\tilde{\varphi}.\] 

The case of quantifier $\forall$ is similar.

10. Follows from part 9, since for a closed formula $\varphi$, $\tilde{\varphi}=\varphi$.
\end{proof}

\begin{theorem}
$TI_{s+1}\vdash Prf_{TI_s}(m)\wedge Eval(\bar{f})\supset Tr(m,\bar{f})$.
\label{theorem:proof}
\end{theorem}

\begin{proof}
We will describe derivation in $TI_{s+1}$ informally by induction on the derivation of $\varphi_m$ in $TI_s$.

Suppose $m$ is the G\"{o}del number of the comprehension axiom:
\[\exists x_j^{n+1} \forall x_i^n \left(x_i\in x_j\equiv \varphi \right),\]
where $srt(\varphi)\leqslant n+1$ and $x_j^{n+1}$ is not a parameter of $\varphi$.

By Lemma \ref{lemma:trset11}.10, 
\[Tr(\llcorner \varphi\lrcorner,Sub_n(\bar{f},i,z))\equiv \tilde{\varphi},\] 
where $\tilde{\varphi}$ is obtained from $\varphi$ by replacing parameter $x_i^n$ by $z$ and each other parameter $x_r^k$ by $Val_k(f_k,r)$. 

If $k\geqslant 1$, then $Val_k(f_k,r)$ has only a parameter of type $k$ and $Val_0(f_1,r)$ has only a parameter of type 1. So $srt(\tilde{\varphi})= max(srt(\varphi),1) \leqslant n+1$. 
By the comprehension axiom in $TI_{s+1}$, there exists a set $y^{n+1}=\left\lbrace z^n \mid \tilde{\varphi}\right\rbrace$. Then 
\[\forall z^n \left[z\in y \equiv Tr(\llcorner  \varphi\lrcorner, Sub_n(\bar{f},i,z))\right],\]
which implies $Tr(m,\bar{f})$.

For all other axioms and derivation rules the theorem statement follows from Lemma \ref{lemma:trset11}.
\end{proof}

\subsection{The strengths of theories $TI_s$ and $TI$}
For an axiomatic theory $A$ and a fixed G\"{o}del numbering of its expressions denote $Proof_A(n,m)$ the arithmetic formula stating that a formula with G\"{o}del number $m$ has a proof with G\"{o}del number $n$ in the theory $A$. 
Denote $Prf_A(m)\leftrightharpoons \exists n Proof_A(n,m)$ and $Con(A)\leftrightharpoons\neg  Prf_A(\llcorner \perp \lrcorner)$. Thus, $Con(A)$ is the arithmetic formula that states the consistency of $A$. Detailed constructions of these formulas can be found, for example, in \cite{mend09}. 

For next theorems we fix a trivial evaluation $\bar{f}$: $f^k_k = \left\lbrace \varnothing^{k-1}\right\rbrace , k=1,\ldots,s$, and $f^1_0=\left\lbrace \langle i,0 \rangle \mid i=i\right\rbrace $. Then $Eval(\bar{f})$.

\begin{theorem}
$TI_{s+1}\vdash Con(TI_s), s=0,1,2,\ldots$.
\label{theorem:con1}
\end{theorem}

\begin{proof}
Suppose $s\geqslant 1$. If $TI_{s+1}$ is inconsistent, then any formula can be derived in it. We assume that it is consistent. We will describe derivation in $TI_{s+1}$ informally. 

Assume $Prf_{TI_s}(\llcorner \perp \lrcorner)$. By Theorem \ref{theorem:proof}, $Tr(\llcorner \perp \lrcorner,\bar{f}) $ that contradicts Lemma \ref{lemma:trset11}.2. Therefore $\neg Prf_{TI_s}(\llcorner \perp \lrcorner)$.

If $s=0$, then the proof has similar steps, just the definition of the truth predicate $Tr_0 $  does not contain clauses for $\in_k$.
\end{proof}

The proofs of the remaining lemmas and theorems in this subsection are similar to the corresponding proofs in \cite{mcnt53}.

\begin{lemma}
$TI_2\vdash Con(I_n)\supset Con(I_{n+1})$.
\label{lemma:con}
\end{lemma}

\begin{proof}

In Theorem 4.3 in \cite{mcnt53} it was proven that $T_2\vdash Con(I_n)\supset Con(I_{n+1})$. The same proof applies to $TI_2$.
\end{proof}

\begin{theorem}
$TI_2\vdash Con(I_n)$ for $n=0,1,2,\ldots$.
\label{theorem:con2}
\end{theorem}

\begin{proof}
The proof is by induction on $n$. 

$n=0$: $TI_0=I_0=PA$, where $PA$ is the first order arithmetic (Peano arithmetic).
By Theorem \ref{theorem:con1}, $TI_1\vdash Con(I_0)$ and therefore $TI_2\vdash Con(I_0)$. 

Assume $TI_2\vdash Con(I_n)$. Then by Lemma \ref{lemma:con}, $TI_2\vdash Con(I_{n+1})$.
\end{proof}

\begin{theorem}
\begin{enumerate}
\item $TI_3\vdash Con(I)$.
\item $TI\vdash Con(I)$.
\end{enumerate}
\label{theorem:con3}
\end{theorem}

\begin{proof}
1. We will describe derivation in $TI_3$ informally. Assume  $Prf_I(\llcorner \perp \lrcorner)$. Since every proof in $I$ is finite, for some $n$ we have: 
\begin{equation}
Prf_{I_n}(\llcorner \perp \lrcorner).
\label{eq:con}
\end{equation}
By formalizing Theorem \ref{theorem:con2}, we get $Prf_{TI_2}(\llcorner Con(I_n) \lrcorner)$. By Theorem \ref{theorem:proof} for $s=2$, $Tr(\llcorner Con(I_n) \lrcorner,\bar{f})$ and by Lemma \ref{lemma:trset11}.10, $Con(I_n)$, since it is a closed formula. This contradicts (\ref{eq:con}).

2. Follows from part 1.
\end{proof}

\section{Comparison of theories $SLP$ and $TI$}
\subsection{Interpretation of $SLP_s$ in $TI_s$}
Here we formalize our proofs about the model $\mathcal{B}_s$ in the theory $TI_s$ for a fixed $s\geqslant 1$. First we introduce some notations in $TI_s$.

By induction on $n$ we define $\left(y^n\right)_i$. 

$\left(y^0\right)_i$ is the $i$th element of $y$ when $y$ is regarded as the numerical code of a finite sequence of natural numbers, it is given by a primitive recursive function.
\[\left(y^{n+1}\right)_i=\left\lbrace z^n \mid \left[\left\lbrace i+1 \right\rbrace ^n, z \right]\in y\right\rbrace.\]

Clearly, $TI_s \vdash \left(\langle x_1^n,\ldots,x_k^n\rangle \right)_{i-1}=x_i$, where $i=1,2,\ldots,k$.

All the aforementioned sets can be constructed in $TI_s$ using its limited comprehension axiom.

For types $n_1,\ldots,n_k$ and $m=max\left\lbrace n_1,\ldots,n_k \right\rbrace$ we define:
\[y_1^{n_1+1}\times \ldots \times y_k^{n_k+1}=
\left\lbrace \langle x_1^{n_1},\ldots,x_k^{n_k}\rangle \mid \bigwedge_{i=1}^k(x_i\in y_i) \right\rbrace. \]

Next we introduce a formula $Seq(x^{n+1},k)$ of $TI_s$, which means that $x$ is a sequence of length $k$.
$Seq(x^0,k)$ is an arithmetic formula stating that the natural number $x$ is a code for a sequence of natural numbers of length $k$.
\begin{multline*}
Seq(x^{n+1},k)\leftrightharpoons
 \\
\forall u^n\left(u\in x\equiv  
u=\left[\left\lbrace 0\right \rbrace^{n},\left\lbrace k\right \rbrace^{n} \right] \vee \exists i, z^n 
\left( 1\leqslant i\leqslant k \wedge u=\left[\left\lbrace i\right \rbrace^{n},z \right] \right)\right).
\end{multline*}

Clearly,
\[TI_s\vdash \langle x_1^{n_1},\ldots,x_k^{n_k}\rangle= \langle y_1^{n_1},\ldots,y_k^{n_k}\rangle \supset \bigwedge_{i=1}^k(x_i=y_i);\]
\[TI_s\vdash y^m=\langle x_1^{n_1},\ldots,x_k^{n_k}\rangle \supset Seq(y,\tilde{k}),\]
where $m=max\left\lbrace n_1,\ldots,n_k\right\rbrace$. 

The relation $\leqslant$ on finite sequences is defined by:
\[x^n\leqslant y^n\leftrightharpoons \exists k,l \left[Seq(x,k) \wedge Seq(y,l) \wedge l\leqslant k \wedge(\forall i<l)\left((x)_i=(y)_i \right)\right].\]

For fixed $i$ and $n\leqslant m$ the following formula states: "if $y^m$ is a sequence, then $x^n$ is its $i$th element, and $\emptyset$ otherwise".
\begin{multline*}
elt_{m,n}(x^n,y^m,i) \leftrightharpoons \exists k \left[Seq(y,k)\wedge i< k\wedge (y)_i=\left\lbrace x \right\rbrace ^{m-n}\right]
\\
\vee \left[\neg \exists k,z^n \left(Seq(y,k)\wedge i< k\wedge (y)_i=\left\lbrace z \right\rbrace ^{m-n}\right)\wedge x=\emptyset\right].
\end{multline*}

Clearly, $TI_s \vdash \forall i, y^m\exists ! x^n elt_{m,n}(x,y,i)$.

So we can introduce a functional symbol $\langle y^m\rangle_{n,i}$ in $TI_s$ for any $n\leqslant m\leqslant s$. It is easy to prove that:
\[TI_s \vdash y^m=\langle x_1^{n_1},\ldots,x_k^{n_k}\rangle \supset \langle y\rangle_{n_i,i-1}=x_i.\]

Similarly we introduce a functional symbol $lth(x)$ such that 
\[TI_s \vdash Seq(x^n,k) \supset lth(x)=k.\]

Functions and partial functions are defined as sets of ordered pairs as usual. The predicates "$S$ \textit{is a path in} $d$", "$f$ \textit{is monotonic on} $d$", "$f$ \textit{is complete on} $d$" are easily expressed in the language $TI_s$. Thus, for a fixed partial order $\leqslant$ on $d$,
\begin{multline*}
\text{"}S\textit{ is a path in }d\text{"} 
\\
\leftrightharpoons (\forall x,y\in S)(x\leqslant y \vee y\leqslant x) 
\wedge (\forall x\in d)(\forall y\in S)(x \leqslant y \vee y \leqslant x \supset x\in S);
\end{multline*}
\begin{multline*}
\text{"}f\textit{ is monotonic on }d\text{"} 
\\
\leftrightharpoons (\forall x,y\in d)(x\leqslant y 
\supset\forall n,z\left(\langle\langle y,n\rangle,z\rangle \in f \supset \langle\langle x,n\rangle,z\rangle \in f \right);
\end{multline*}
\begin{multline*}
\text{"}f\textit{ is complete on }d\text{"} 
\\
\leftrightharpoons\forall S \left[\text{"}S\textit{ is a path in }d\text{"} \supset 
\forall n (\exists x\in S)\exists y \left(\langle\langle x,n\rangle,y\rangle \in f \right)\right]. 
\end{multline*}

Here the variables have appropriate types depending on the type of $d$. 

Next we formalize the Definition \ref{def:my_beth_model} of the Beth model $\mathcal{B}_s$ in $TI_s$. Most of it is straightforward. We will elaborate only on non-trivial parts. The definition of $\nu_k(\xi)$ is formalized as follows:
\[\nu_k(\xi)=\left\lbrace \langle\langle x,n\rangle, y\rangle \mid n<lh(x) \wedge \langle\langle n,\left( \langle x \rangle_{k-1,k-1}\right)_n\rangle, y \rangle \in \xi\right\rbrace. \]

For each $k=0,1,\ldots,s-1$, the sets $a_k,d_k,b_k,c_k$ and $l_k$ have type $k+1$. The set $M=d_{s-1}$ has type $s$. In theory $TI_s$ we use predicates (not sets) for 
$a_s,b_s,c_s$ and $l_s$. For example, for $a_s$ the corresponding predicate is 
\begin{multline*}
A_s(f^s)\leftrightharpoons f:\left(d_{s-1}\times \omega \right)\dashrightarrow a_{s-1}
\\
\wedge "f \textit{ is monotonic on }d_{s-1}"\wedge "f \textit{ is complete on }d_{s-1}"
\end{multline*}
and it can be expressed by a formula of $TI_s$. For universality we will use the notation $f^s\in a_s$ for the predicate $A_s(f^s)$, as well as the notations $f^s\in b_s$, $f^s\in c_s$ and $f^s\in l_s$  standing for corresponding predicates. 

For $x,y\in d_k$ we define: 
\[x\preceq_k y \leftrightharpoons \bigwedge_{i=0}^k \left(\langle x\rangle_{i,i}\leqslant \langle y\rangle_{i,i} \right). \]

It is easy to see that for any $y^k\in d_k$ with $lh(y)=m$ $(k=0,1,\ldots,s-1)$, for any $n\leqslant k$ and $j<m$:
\[\langle y\rangle_{n,n}\in a_n^{(m)} \text{ and }\left( \langle y\rangle_{n,n}\right)_j \in a_n.\]

For any $y^k \in d_k$ and $n=1,2,\ldots,k+1$ we denote 
\[segment_n(y)=\langle\langle y\rangle_{0,0},\langle y\rangle_{1,1},\ldots,\langle y\rangle_{n-1,n-1},\rangle.\]
Clearly the type of $segment_n(y)$ is $n-1$. As in the informal definition, we will use the notation $\bar{y}(n)$ for $segment_n(y)$; it denotes the initial segment of $y$ of length $n$. Here $\bar{y}(n)$ is introduced only for $y\in d_k$ and not for arbitrary $y$. 

Suppose $\varphi$ is an evaluated formula of $SLP_s$. Denote $int_s(\varphi)$ the formula of $TI_s$ that expresses the relation $\mathcal{B}_s\Vdash \varphi$. For any $\varphi$, $int_s(\varphi)$ can be formally defined following the Definitions \ref{def:forcing} and \ref{def:my_beth_model}. 

\begin{theorem}
$SLP_s\vdash \varphi \Rightarrow TI_s\vdash int_s(\bar{\varphi}).$
\label{theorem:ints}
\end{theorem}

\begin{proof}
The proofs in Sections 4-9 are easily formalized using the aforementioned notations. That implies the theorem.
\end{proof}

\begin{theorem}
\begin{enumerate}
\item The theories $SLP_s$ and $TI_s$ are equiconsistent.
\item The theories $SLP$ and $TI$ are equiconsistent.
\end{enumerate}
\label{theorem:con}
\end{theorem}

\begin{proof}
Part 1 follows from Theorem \ref{theorem:int}.1 and Theorem \ref{theorem:ints}. Part 2 follows from Theorem \ref{theorem:int}.2 and  part 1, since any derivation in $SLP$ is finite, therefore it is a derivation in $SLP_s$ for some $s$.
\end{proof}

\subsection{Constructing a forcing predicate for $SLP_s$ in $TI_{s+1}$}

Fix $s\geqslant 1$. In $TI_{s+1}$ we will construct a formula $Fr(\alpha,n,\bar{f})$ expressing the forcing predicate $\alpha\Vdash\varphi_n(\bar{f})$ in the Beth model $\mathcal{B}_s$. Thus, $Fr(\alpha,n,\bar{f})$ means that the formula of $SLP_s$ with G\"{o}del number $n$ under evaluation $\bar{f}$ is forced at node $\alpha$. The construction follows similar steps as the construction of the truth predicate $Tr$ in the previous section, but with a few extra details. In this subsection details of the construction will be described and its applications to the relative strengths of our theories will be given in next sub-section.

We will use the following precise notations for variables of $SLP_s$. For $n=1,2,\ldots,s$; $i=1,2, \ldots$:
\begin{list}{•}{•}
\item $x_i$ are variables over natural numbers,
\smallskip
\item $F_i^n$ are variables over $n$-functionals,
\smallskip
\item $A_i^n$ are variables over lawlike $n$-functionals,
\smallskip
\item $\mathcal{F}_i^n$ are variables over lawless $n$-functionals.
\end{list}

We will map variables of $SLP_s$ to variables of $TI_s$ as follows:
\begin{list}{•}{•}
\item $x_i \rightarrow x_i^0$;
\medskip
\item $F_i^n \rightarrow x_{3i}^n$;
\medskip
\item $A_i^n \rightarrow x_{3i+1}^n$;
\medskip
\item $\mathcal{F}_i^n \rightarrow x_{3i+2}^n$.
\end{list}

The definitions of the domains for $k$-functionals, lawlike $k$-functionals, lawless $k$-functionals, and the sets $d_k$ and $c_k$ are easily formalised in $TI_{s+1}$ as in the previous sub-section. All the sets $a_k$, $b_k$, $l_k$, $d_k$, and $c_k$ have type $k+1$, including the case $k=s$ (unlike the previous sub-section, where $a_s$, $b_s$ and $l_s$ were predicates). The domain $M=d_{s-1}$ has type $s$.

We will fix G\"{o}del numbering of all expressions of the language $SLP_s$. For any expression $g$ we denote $\llcorner g\lrcorner$ its G\"{o}del number in this numbering.

The functional symbol $Val_k(f,i)$ that produces the $i$th element of $f$ was introduced in section 13.  Next we introduce a predicate $Evl_k(f^k), k=1,\ldots,s$, with the meaning: "$f$ is an evaluation of all variables for $k$-functionals".
\begin{multline*}
Evl_k(f^k)\leftrightharpoons \left(\forall x^{k-1}\in f \right)\exists i, z^{k-1}(x=\langle i,z\rangle)
\\
\wedge \forall i\left[ Val_k(f,3i)\in a_k \wedge Val_k(f,3i+1)\in b_k \wedge Val_k(f,3i+2)\in l_k \right]. 
\end{multline*}

For type 0 the predicate $Evl_0(f^1)$ is the same as $Eval_0(f^1)$ defined in section 13.

As before we denote $\bar{f}$ the list of variables $f_0^1,f_1^1,f_2^2,\ldots,f_s^s$ and 
\[Evl(\bar{f})\leftrightharpoons \bigwedge_{k=0}^s Evl_k(f_k).\] 

The functional symbol $Subst_k(f^k,i,y^k)$ and notation $Sub_k(\bar{f},i,y^k)$ were defined in section 13 $(k=0,1,\ldots,s)$.

\begin{lemma}
Suppose $k=1,2,\ldots,s$. It is derived in $TI_{s+1}$ that the formula $Evl(\bar{f})$ implies each of the following formulas:
\begin{enumerate}
\item $Evl(Sub_0(\bar{f},i,u^0))$;
\smallskip
\item $u^k\in a_k \supset Evl(Sub_k(\bar{f},3i,u))$;
\smallskip
\item $u^k\in b_k \supset Evl(Sub_k(\bar{f},3i+1,u))$;
\smallskip
\item $u^k\in l_k \supset Evl(Sub_k(\bar{f},3i+2,u))$.
\end{enumerate}
\label{lemma:force0}
\end{lemma}

\begin{proof}
Follows from the definitions of $Evl$ and $Sub_k$.
\end{proof}

We denote $Fnl_k(n)$ the arithmetic formula stating that $n$ is the G\"{o}del number of a $k$-functional ($k=1,2,\ldots,s$). We denote $Fnl_0(n)$ the arithmetic formula stating that $n$ is the G\"{o}del number of a term. We denote $Z_n$ the term or $k$-functional of $SLP_s$ with G\"{o}del number $n$ and 
\[Subfnl_k(m,n)\leftrightharpoons Fnl_k(m)\wedge "Z_m\textit{ is a part of }Z_n".\]
Thus, for $k=1,2,\ldots,s$, $Subfnl_k(m,n)$ means that the $k$-functional $Z_m$ is a sub-functional of $Z_n$, and $Subfnl_0(m,n)$ means that the term $Z_m$ is a sub-term of $Z_n$.

Next we define a predicate $Fnterm(\alpha,\bar{f},n,a^{s+1})$, which means that $a$ is the set of the values at node $\alpha$ of all sub-terms and sub-functionals of $Z_n$ under evaluation $\bar{f}$. The predicate uses the interpretations $\hat{K}^k$ and $\hat{N}^k$ from Definition \ref{def:my_beth_model}, which are easily formalised.
\begin{multline*}
Fnterm(\alpha^{s-1},n,\bar{f},a^{s+1})\leftrightharpoons \alpha\in M\wedge \left( \bigvee_{k=0}^s Fnl_k(n)\right) \wedge Evl(\bar{f})
\\
\wedge (\forall x^s\in a)\exists m \left[\bigvee_{k=0}^s \left(Subfnl_k(m,n)\wedge \exists u^k(x=\langle m,u\rangle)\right)\right] 
\\
\wedge\forall m \bigwedge_{k=0}^s\left[Subfnl_k(m,n)\supset C_k(m,a)\right],\text{ where } 
\end{multline*}
\[C_s(m,a)\leftrightharpoons \bigvee_{r=1}^{10}C_{s,r};
\text{ for }  k=0,1,\ldots,s-1, 
C_k(m,a)\leftrightharpoons \bigvee_{r=1}^{11}C_{k,r}
\]
and $C_{k,1} - C_{k,11}$ determine the value of $Z_m$ as follows.
\[C_{k,1}\leftrightharpoons m=\llcorner 0\lrcorner\wedge\forall y^s\left(\langle m,y\rangle\in a \equiv y=\left\lbrace 0 \right\rbrace^s \right).\]
\[C_{k,2}\leftrightharpoons m=\llcorner K^k\lrcorner\wedge\forall y^s\left(\langle m,y\rangle\in a \equiv y=\left\lbrace \hat{K}^k \right\rbrace^{s-k} \right).\] 
\[C_{k,3}\leftrightharpoons \exists i\left[m=\llcorner x_i\lrcorner \wedge\forall y^s\left(\langle m,y\rangle\in a \equiv y=\left\lbrace Val_0(f_0,i) \right\rbrace^s \right) \right].\]
\[C_{k,4}\leftrightharpoons \exists i\left[m=\llcorner F^k_i\lrcorner \wedge\forall y^s\left(\langle m,y\rangle\in a \equiv y=\left\lbrace Val_k(f_k,3i) \right\rbrace^{s-k} \right) \right].\]
\[C_{k,5}\leftrightharpoons \exists i\left[m=\llcorner A^k_i\lrcorner \wedge\forall y^s\left(\langle m,y\rangle\in a \equiv y=\left\lbrace Val_k(f_k,3i+1) \right\rbrace^{s-k} \right) \right].\]
\[C_{k,6}\leftrightharpoons \exists i\left[m=\llcorner \mathcal{F}^k_i\lrcorner \wedge\forall y^s\left(\langle m,y\rangle\in a \equiv y=\left\lbrace Val_k(f_k,3i+2) \right\rbrace^{s-k} \right) \right].\]
\[C_{k,7}\leftrightharpoons 
\exists j\left[m=\llcorner S(Z_j)\lrcorner
\wedge\forall y^s\left(\langle m,y\rangle\in a \equiv 
\exists u^0\left(\langle j,u\rangle\in a\wedge y=\left\lbrace S(u) \right\rbrace^s \right)\right)\right].\]
Denote $+$ as $\theta_8$ and $\cdot$ as $\theta_9$. For $p=8,9$,
\begin{multline*}
C_{k,p}\leftrightharpoons 
\exists j,i\left[m=\llcorner Z_j\theta_p Z_i  \lrcorner \right.
\\
\left.
\wedge\forall y^s\left(\langle m,y\rangle\in a \equiv \exists u^0,v^0\left(\langle j,u\rangle\in a\wedge \langle i,v\rangle\in a\wedge y=\left\lbrace u\theta_p v \right\rbrace^s \right)\right)\right].
\end{multline*}
\begin{multline*}
C_{k,10}\leftrightharpoons \exists j\left\lbrace m=\llcorner N^k(Z_j)\lrcorner 
\right.
\\
\left.
\wedge\forall y^s\left[\langle m,y\rangle\in a \equiv \exists u^k
\left(\langle j,u\rangle\in a\wedge y=\left\lbrace \hat{N}^k(u) \right\rbrace^{s-k} \right)
 \right] \right\rbrace.
\end{multline*}
\begin{multline*}
C_{k,11}\leftrightharpoons \exists j,i\left\lbrace m=\llcorner Ap^{k+1}(Z_j,Z_i)\lrcorner\wedge\forall y^s\left[\langle m,y\rangle\in a \right.
\right.
\\
\left.
\left.
\equiv \exists w^{k+1}, v^0,u^{k}
\left(\langle j,w\rangle\in a\wedge 
\langle i,v\rangle\in a
\right.
\right.
\right.
\\
\left.
\left.
\left.
\wedge
\langle\langle\bar{\alpha}(k+1),v\rangle,u\rangle\in w
\wedge y=\lbrace u\rbrace^{s-k}\right)\right]\right\rbrace .
\end{multline*}

Denote 
\[a\mid\mid_m=\left\lbrace \langle j,y^s\rangle\in a\mid \bigvee_{p=0}^s Subfnl_p(j,m)\right\rbrace .\]

\begin{lemma}
The following formulas are derived in $TI_{s+1}$.
\begin{enumerate}
\item $Subfnl_k(m,n)\wedge Fnterm(\alpha,n,\bar{f},a^{s+1})\supset Fnterm(\alpha,m,\bar{f},a\mid\mid_m), 
\\
k=0,1,\ldots,s$.
\medskip
\item $Fnterm(\alpha,n,\bar{f},a^{s+1})\wedge Fnterm(\alpha,n,\bar{f},b^{s+1})\supset a=b$.
\medskip
\item $Fnl_k(n)\wedge\alpha\in M\wedge Evl(\bar{f})\supset\exists ! a^{s+1} Fnterm(\alpha,n,\bar{f},a)$, $k=0,1,\ldots,s$.
\end{enumerate}
\label{lemma:frc2}
\end{lemma}
\begin{proof}
1. Assume the premises. For any $i$ with $Subfnl_p(i,m),p=0,1,\ldots,s$, we have:
\begin{equation}
C_p(i,a) \text{ and }\forall y^s\left(\langle i,y\rangle\in a\equiv\langle i,y\rangle\in a\mid\mid_m \right).
\label{eq:force00}
\end{equation}

To prove $Fnterm(\alpha,m,\bar{f},a\mid\mid_m)$, it is sufficient to show for any $j$:
\begin{equation}
Subfnl_p(j,m)\supset C_p(j,a\mid\mid_m).
\label{eq:force01}
\end{equation}

When $Z_j$ is a constant or a variable, (\ref{eq:force01}) follows from $C_p(j,a)$. 

Suppose $j=\llcorner Ap^{p+1}(Z_q,Z_r)\lrcorner$. Since $C_p(j,a)$, we have for any $y^s$:
\begin{multline*}
\langle j,y\rangle\in a
\\
\equiv\exists w^{p+1}, v^0,u^p\left[ \langle q,w\rangle\in a\wedge\langle r,v\rangle\in a\wedge\langle\langle \bar{\alpha}(p+1),v\rangle,u\rangle\in w\wedge y=\lbrace u\rbrace^{s-p}\right] 
\end{multline*}
and by (\ref{eq:force00}):
\begin{multline*}
\langle j,y\rangle\in a\mid\mid_m
\equiv
\exists w^{p+1}, v^0,u^p\left[ \langle q,w\rangle\in a\mid\mid_m
\right.
\\
\left.
\wedge\langle r,v\rangle\in a\mid\mid_m\wedge\langle\langle \bar{\alpha}(p+1),v\rangle,u\rangle\in w\wedge y=\lbrace u\rbrace^{s-p}\right], 
\end{multline*}
which implies $C_p(j,a\mid\mid_m)$.

For other forms of $Z_j$ the proof is similar.

2. Assume $Fnterm(\alpha,n,\bar{f},a^{s+1})$ and $ Fnterm(\alpha,n,\bar{f},b^{s+1})$. Then $\alpha\in M$, $Evl(\bar{f})$ and $Fnl_q(n)$ for some $q=0,1,\dots,s$. It is sufficient to prove:
\begin{equation}
\forall m\left[\bigvee_{k=0}^s Subfnl_k(m,n)\supset \forall y^s(\langle m,y\rangle\in a\equiv\langle m,y\rangle\in b)\right].
\label{eq:force2}
\end{equation}

We will prove it by induction on $m$. The cases of constants and variables are obvious.

Assume (\ref{eq:force2}) holds for numbers less than $m$. We will consider the only non-trivial case $m=\llcorner Ap^{k+1}(Z_j, Z_i)\lrcorner$. By the inductive assumption, 
\[\forall y^s(\langle j,y\rangle\in a\equiv\langle j,y\rangle\in b)\text{ and }
\forall y^s(\langle i,y\rangle\in a\equiv\langle i,y\rangle\in b).\]
Therefore for any $y^s$,
\begin{multline*}
\langle m,y\rangle\in a
\\
\equiv
\exists w^{k+1}, v^0,u^{k}\left(\langle j,w\rangle\in a\wedge \langle i,v\rangle\in a
\wedge\langle\langle\bar{\alpha}(k+1),v\rangle,u\rangle\in w\wedge y=\lbrace u\rbrace^{s-k}\right)
\\
\equiv
\exists w^{k+1}, v^0,u^{k}\left(\langle j,w\rangle\in b\wedge \langle i,v\rangle\in b
\wedge\langle\langle\bar{\alpha}(k+1),v\rangle,u\rangle\in w\wedge y=\lbrace u\rbrace^{s-k}\right)
\\
\equiv\langle m,y\rangle\in b.
\end{multline*}

3. The uniqueness was proven in part 2. We prove the existence by induction on $n$.

1) If $n=\llcorner 0\lrcorner$, we take $a^{s+1}=\lbrace \langle n,0\rangle\rbrace$. 
\smallskip

2) If $n=\llcorner K^k\lrcorner$, we take $a^{s+1}=\lbrace \langle n,\hat{K}^k\rangle\rbrace$. 
\smallskip

3) If $n=\llcorner x_i\lrcorner$, we take $a^{s+1}=\lbrace \langle n,Val_0(f_0,i)\rangle\rbrace$.
\smallskip

4) If $n=\llcorner F_i^k\lrcorner$, we take $a^{s+1}=\lbrace \langle n,Val_k(f_k,3i)\rangle\rbrace$.  
\smallskip

5) If $n=\llcorner A_i^k\lrcorner$, we take $a^{s+1}=\lbrace \langle n,Val_k(f_k,3i+1)\rangle\rbrace$.  
\smallskip

6) If $n=\llcorner \mathcal{F}_i^k\lrcorner$, we take $a^{s+1}=\lbrace \langle n,Val_k(f_k,3i+2)\rangle\rbrace$.  
\smallskip

In each case $Fnterm(\alpha,n,\bar{f},a)$ is obvious.

Assume the existence for numbers less than $n$. 

7) $n=\llcorner Ap^{k+1}(Z_j, Z_i)\lrcorner$. 

By the inductive assumption, there exist sets $c^{s+1}$ and $d^{s+1}$ such that 
\[Fnterm(\alpha,j,\bar{f},c)\wedge Fnterm(\alpha,i,\bar{f},d).\] 

We take $a^{s+1}=$
\[=c\cup d\cup\lbrace \langle n,u^{k}\rangle \mid\exists w^{k+1},v^0
\left(\langle j,w\rangle\in c\wedge 
\langle i,v\rangle\in d\wedge
\langle\langle\bar{\alpha}(k+1),v\rangle,u\rangle\in w
\right)\rbrace\] 
and prove $Fnterm(\alpha,n,\bar{f},a)$.

First we prove:
\begin{equation}
\forall p\left[Subfnl_q(p,j)\supset\forall u^q\left(
\langle p,u\rangle\in a\equiv\langle p,u\rangle\in c \right) \right],q=0,1,\ldots,s.
\label{eq:force3a}
\end{equation}

Assume $Subfnl_q(p,j)$. If $\neg Subfnl_q(p,i)$, then the conclusion is obvious. If $Subfnl_q(p,i)$, then  $Fnterm(\alpha,p,\bar{f},c\mid\mid_p)$ and $Fnterm(\alpha,p,\bar{f},d\mid\mid_p)$ by part 1, $c\mid\mid_p=d\mid\mid_p$ by part 2 and for any $u^q$:
\[\langle p,u\rangle\in a\equiv\langle p,u\rangle\in c\mid\mid_p\cup d\mid\mid_p\equiv\langle p,u\rangle\in c.\]

(\ref{eq:force3a}) is proven. Similarly,
\begin{equation}
\forall p\left[Subfnl_q(p,i)\supset\forall u^q \left(
\langle p,u\rangle\in a\equiv\langle p,u\rangle\in d \right) \right],q=0,1,\ldots,s.
\label{eq:force3b}
\end{equation}

To prove $Fnterm(\alpha,n,\bar{f},a)$ , it is sufficient to show:
\begin{equation}
\forall m \bigwedge_{q=0}^s\left[Subfnl_q(m,n)\supset C_q(m,a)\right].
\label{eq:force0}
\end{equation}

Assume $Subfnl_q(m,n)$.

i) $m=n$.

Then $q=k$. Applying (\ref{eq:force3a}) to $p=j$ and (\ref{eq:force3b}) to $p=i$, we get for any $y^s$:
\begin{multline*}
\langle n,y\rangle\in a
\\
\equiv
\exists w^{k+1}, v^0,u^{k}
\left(\langle j,w\rangle\in c\wedge 
\langle i,v\rangle\in d
\wedge
\langle\langle\bar{\alpha}(k+1),v\rangle,u\rangle\in w
\wedge y=\lbrace u\rbrace^{s-k}\right)
\\
\equiv
\exists w^{k+1}, v^0,u^{k}
\left(\langle j,w\rangle\in a\wedge 
\langle i,v\rangle\in a
\wedge
\langle\langle\bar{\alpha}(k+1),v\rangle,u\rangle\in w
\wedge y=\lbrace u\rbrace^{s-k}\right)
\end{multline*}
which implies $C_k(n,a)$.

ii) $m<n$.

Then $Subfnl_q(m,j)\vee Subfnl_q(m,i)$. We can assume $Subfnl_q(m,j)$ without loss of generality. Applying (\ref{eq:force3a}) to each possible form of $Z_m$, we get $C_q(m,a)$ due to $Fnterm(\alpha,j,\bar{f},c)$. 

For example, if $Z_m=N^q(Z_r)$, by (\ref{eq:force3a}) we have for any $y^s$:
\begin{multline*}
\langle m,y\rangle\in a\equiv\langle m,y\rangle\in c\equiv\exists u^q\left(\langle r,u\rangle\in c\wedge y=\lbrace \hat{N}^q(u)\rbrace^{s-q}\right)
\\
\equiv\exists u^q\left(\langle r,u\rangle\in a\wedge y=\lbrace \hat{N}^q(u)\rbrace^{s-q}\right),
\end{multline*}
which implies $C_q(m,a)$.

This completes the proof for case 7). The cases of other functional symbols are similar.
\end{proof}

\begin{lemma}
Suppose $k=0,1,\ldots,s$. It is derived in $TI_{s+1}$ that the formula 
\[Subfnl_k(m,n)\wedge Fnterm(\alpha,n,\bar{f},a^{s+1})\]
implies each of the following formulas:
\begin{enumerate}
\item $m=\llcorner 0\lrcorner\supset\forall y^s\left(\langle m,y\rangle\in a \equiv y=\left\lbrace 0 \right\rbrace^s \right)$.
\medskip
\item $m=\llcorner K^k\lrcorner\supset\forall y^s\left(\langle m,y\rangle\in a \equiv y=\left\lbrace \hat{K}^k \right\rbrace^{s-k} \right)$.
\medskip
\item
$m=\llcorner x_i\lrcorner \supset\forall y^s\left(\langle m,y\rangle\in a \equiv y=\left\lbrace Val_0(f_0,i) \right\rbrace^s \right)$.
\medskip
\item
$m=\llcorner F^k_i\lrcorner\supset\forall y^s\left(\langle m,y\rangle\in a \equiv y=\left\lbrace Val_k(f_k,3i) \right\rbrace^{s-k} \right)$.
\medskip
\item
$m=\llcorner A^k_i\lrcorner \wedge\forall y^s\left(\langle m,y\rangle\in a \equiv y=\left\lbrace Val_k(f_k,3i+1) \right\rbrace^{s-k} \right)$.
\medskip
\item
$m=\llcorner \mathcal{F}^k_i\lrcorner\supset\forall y^s\left(\langle m,y\rangle\in a \equiv y=\left\lbrace Val_k(f_k,3i+2) \right\rbrace^{s-k} \right)$.
\medskip
\item
$m=\llcorner S(Z_j)\lrcorner
\supset\forall y^s\left(\langle m,y\rangle\in a \equiv 
\exists u^0\left(\langle j,u\rangle\in a\wedge y=\left\lbrace S(u) \right\rbrace^s \right)\right)$.
\medskip
\item
$m=\llcorner Z_j\theta Z_i  \lrcorner 
\supset\forall y^s\left(\langle m,y\rangle\in a 
\right.
\medskip
\\
\left.
\equiv \exists u^0,v^0\left(\langle j,u\rangle\in a\wedge \langle i,v\rangle\in a\wedge y=\left\lbrace u\theta v \right\rbrace^s \right)\right)$,
where $\theta$ denotes $\cdot$ or $+$.
\medskip
\item
$m=\llcorner N^k(Z_j)\lrcorner
\\ 
\supset\forall y^s\left[\langle m,y\rangle\in a \equiv \exists u^k
\left(\langle j,u\rangle\in a\wedge y=\left\lbrace \hat{N}^k(u) \right\rbrace^{s-k} \right)
 \right]$.
\medskip
\item
$m=\llcorner Ap^{k+1}(Z_j,Z_i)\lrcorner\supset\forall y^s\left[\langle m,y\rangle\in a 
\equiv \exists w^{k+1}, v^0,u^{k}
\left(\langle j,w\rangle\in a
\right.
\right.
\smallskip
\\
\left.
\left.
\wedge 
\langle i,v\rangle\in a
\wedge
\langle\langle\bar{\alpha}(k+1),v\rangle,u\rangle\in w
\wedge y=\lbrace u\rbrace^{s-k}\right)\right]$.
\end{enumerate}
\label{lemma:frc3}
\end{lemma}
\begin{proof}
Follows from the definition of $Fnterm$.
\end{proof}

Next we define a functional symbol $Fnt(\alpha,\bar{f})$:
\[Fnt^{s+1}(\alpha,\bar{f})=\left\lbrace \langle n,y^s\rangle \mid \exists a^{s+1}(Fnterm(\alpha,n,\bar{f},a)\wedge\langle n,y\rangle\in a)\right\rbrace.\]

\begin{lemma}
The following formulas are derived in $TI_{s+1}$.
\begin{multline*}
1. Subfnl_k(m,n)\wedge Fnterm(\alpha,n,\bar{f},a^{s+1})
\\
\supset\forall y^s\left[ \langle m,y\rangle\in a
\equiv\langle m,y\rangle\in Fnt(\alpha,\bar{f})\right],
k=0,1,\dots,s.
\end{multline*}
\[2. \langle n,y^s\rangle\in Fnt(\alpha,\bar{f})\supset \bigvee_{k=0}^s\left[Fnl_k(n)\wedge \exists u^k\left(y=\lbrace u\rbrace^{s-k}\right) \right].\]
\label{lemma:frc4}
\end{lemma}
\begin{proof}
1. Assume the premises. By Lemma \ref{lemma:frc2}.1, $Fnterm(\alpha,m,\bar{f},a\mid\mid_m)$. Fix $y^s$.

$\Rightarrow$ If $\langle m,y\rangle\in a$, then $\langle m,y\rangle\in a\mid\mid_m$ and $\langle m,y\rangle\in Fnt(\alpha,\bar{f})$ by the definition of $Fnt$.

$\Leftarrow$ Suppose $\langle m,y\rangle\in Fnt(\alpha,\bar{f})$. Then there is $b^{s+1}$ with $Fnterm(\alpha,m,\bar{f},b)$ and $\langle m,y\rangle\in b$. By Lemma \ref{lemma:frc2}.2, $b=a\mid\mid_m$. So $\langle m,y\rangle\in a\mid\mid_m$ and $\langle m,y\rangle\in a$.

2. Suppose $\langle n,y^s\rangle\in Fnt(\alpha,\bar{f})$. Then for some $a^{s+1}$, $Fnterm(\alpha,n,\bar{f},a)$ and $\langle n,y\rangle\in a$. The conclusion follows from the definition of $Fnterm$.
\end{proof}

\begin{lemma}
It is derived in $TI_{s+1}$ that the formula 
\[\alpha\in M\wedge Evl(\bar{f})\]
implies each of the following formulas.
\begin{enumerate}
\item $\langle\llcorner 0\lrcorner,y^s\rangle\in Fnt(\alpha,\bar{f})\equiv y=\left\lbrace 0 \right\rbrace^s$.
\medskip
\item $\langle\llcorner K^k\lrcorner,y^s\rangle\in Fnt(\alpha,\bar{f}) \equiv y=\left\lbrace \hat{K}^k \right\rbrace^{s-k}$, $k=1,2,\ldots,s$.
\medskip
\item
$\langle \llcorner x_i\lrcorner ,y^s\rangle\in Fnt(\alpha,\bar{f}) \equiv y=\left\lbrace Val_0(f_0,i) \right\rbrace^s$.
\medskip
\item
$\langle \llcorner F^k_i\lrcorner,y^s\rangle\in Fnt(\alpha,\bar{f}) \equiv y=\left\lbrace Val_k(f_k,3i) \right\rbrace^{s-k}$, $k=1,2,\ldots,s$.
\medskip
\item
$\langle \llcorner A^k_i\lrcorner,y^s\rangle\in Fnt(\alpha,\bar{f}) \equiv y=\left\lbrace Val_k(f_k,3i+1) \right\rbrace^{s-k}$, $k=1,2,\ldots,s$.
\medskip
\item
$\langle \llcorner \mathcal{F}^k_i\lrcorner,y^s\rangle\in Fnt(\alpha,\bar{f}) \equiv y=\left\lbrace Val_k(f_k,3i+2) \right\rbrace^{s-k}$, $k=1,2,\ldots,s$.
\medskip
\item
$\langle \llcorner S(Z_j)\lrcorner,y^s\rangle\in Fnt(\alpha,\bar{f}) \equiv 
\exists u^0\left(\langle j,u\rangle\in Fnt(\alpha,\bar{f})\wedge y=\left\lbrace S(u) \right\rbrace^s \right)$.
\medskip
\item
$\langle \llcorner Z_j\theta Z_i  \lrcorner,y^s\rangle\in Fnt(\alpha,\bar{f}) 
\equiv \exists u^0,v^0\left(\langle j,u\rangle\in Fnt(\alpha,\bar{f})
\right.
\smallskip
\\
\left.
\wedge \langle i,v\rangle\in Fnt(\alpha,\bar{f})
\wedge y=\left\lbrace u\theta v \right\rbrace^s \right)$,
where $\theta$ denotes $\cdot$ or $+$.
\medskip
\item
$\langle \llcorner N^k(Z_j)\lrcorner,y^s\rangle\in Fnt(\alpha,\bar{f}) 
\\
\equiv \exists u^k
\left(\langle j,u\rangle\in Fnt(\alpha,\bar{f})\wedge y=\left\lbrace \hat{N}^k(u) \right\rbrace^{s-k} \right)$, $k=1,2,\ldots,s$.
\medskip
\item
$\langle \llcorner Ap^{k+1}(Z_j,Z_i)\lrcorner,y^s\rangle\in Fnt(\alpha,\bar{f}) 
\equiv \exists w^{k+1}, v^0,u^{k}
\left(\langle j,w\rangle\in Fnt(\alpha,\bar{f})
\right.
\smallskip
\\
\left.
\wedge 
\langle i,v\rangle\in Fnt(\alpha,\bar{f})
\wedge
\langle\langle\bar{\alpha}(k+1),v\rangle,u\rangle\in w
\wedge y=\lbrace u\rbrace^{s-k}\right), 
\\
k=0,1,\ldots,s-1$.
\end{enumerate}
\label{lemma:frc5}
\end{lemma}
\begin{proof}
Assume $\alpha\in M\wedge Evl(\bar{f})$. By Lemma \ref{lemma:frc2}.3, for any $n$ with $Fnl_k(n)$, $k=0,1,\ldots,s$, there is a set $a^{s+1}$ such that $Fnterm(\alpha,n,\bar{f},a)$, and by Lemma \ref{lemma:frc4}.1:
\begin{equation}
\forall m\left[Subfnl_q(m,n)\supset
\forall y^s\left( \langle m,y\rangle\in a
\equiv\langle m,y\rangle\in Fnt(\alpha,\bar{f})\right)\right], k=0,1,\ldots,s.
\label{eq:frc1}
\end{equation}

1. Denote $n=\llcorner 0\lrcorner$. By (\ref{eq:frc1}), we have for any $y^s$:
\[\langle n,y\rangle\in Fnt(\alpha,\bar{f})\equiv\langle n,y\rangle\in a\equiv y=\lbrace 0\rbrace^s
\]
by Lemma \ref{lemma:frc3}.1.

Part 2 is proven similarly. Parts 3, 5 and 6 are similar to part 4, which we will consider next.

4. Denote $n=\llcorner F_i^k\lrcorner$. By (\ref{eq:frc1}), we have for any $y^s$:
\[\langle n,y\rangle\in Fnt(\alpha,\bar{f})\equiv\langle n,y\rangle\in a\equiv y=\left\lbrace Val_k(f_k,3i) \right\rbrace^{s-k}\]
by Lemma \ref{lemma:frc3}.4.

Parts 7-9 are similar to part 10, which we will consider next.

10. Denote $n=\llcorner Ap^{k+1}(Z_j,Z_i)\lrcorner$. By (\ref{eq:frc1}):
\begin{eqnarray*}
\forall y^s\left( \langle n,y\rangle\in Fnt(\alpha,\bar{f})\equiv \langle n,y\rangle\in a\right);
\\
\forall y^s\left( \langle j,y\rangle\in Fnt(\alpha,\bar{f})\equiv \langle j,y\rangle\in a\right);
\\
\forall y^s\left( \langle i,y\rangle\in Fnt(\alpha,\bar{f})\equiv \langle i,y\rangle\in a\right).
\end{eqnarray*}

By Lemma \ref{lemma:frc3}.10, for any $y^s$:
\begin{multline*}
\langle n,y\rangle\in Fnt(\alpha,\bar{f})\equiv\langle n,y\rangle\in a
\\
\equiv \exists w^{k+1}, v^0,u^{k}
\left(\langle j,w\rangle\in a
\wedge \langle i,v\rangle\in a\wedge
\langle\langle\bar{\alpha}(k+1),v\rangle,u\rangle\in w
\wedge y=\lbrace u\rbrace^{s-k}\right)
\\
\equiv \exists w^{k+1}, v^0,u^{k}
\left(\langle j,w\rangle\in Fnt(\alpha,\bar{f})
\wedge 
\langle i,v\rangle\in Fnt(\alpha,\bar{f})
\right.
\smallskip
\\
\left.
\wedge
\langle\langle\bar{\alpha}(k+1),v\rangle,u\rangle\in w
\wedge y=\lbrace u\rbrace^{s-k}\right).
\end{multline*}
\end{proof}

\begin{lemma}
The following formulas are derived in $TI_{s+1}$.
\begin{enumerate}
\item $Fnl_k(n)\wedge\langle n,u^k\rangle\in Fnt(\alpha,\bar{f})\supset u\in a_k.$
\medskip
\item $\langle n,u^k\rangle\in Fnt(\alpha,\bar{f})\wedge\langle n,v^k\rangle\in Fnt(\alpha,\bar{f})\supset u=v.$
\medskip
\item $\alpha,\beta\in M\wedge\beta\preceq\alpha\supset Fnt(\alpha,\bar{f})\subseteq Fnt(\beta,\bar{f})$.
\medskip
\item $\alpha\in M\wedge Fnl_k(n)\wedge Evl(\bar{f})\wedge$ "$S^{s+1}$ is a path in $M$ through $\alpha$" $
\smallskip
\\
\supset(\exists\beta\in S)\exists u^k[\langle n,u\rangle\in Fnt(\beta,\bar{f})].$
\end{enumerate}
\label{lemma:frc6}
\end{lemma}
\begin{proof}
1. Proof is by induction on $n$. The cases of constants follow from Lemma \ref{lemma:frc5}.1, 2. The cases of variables follow from $Evl(\bar{f})$ and Lemma \ref{lemma:frc5}.3-6.

Assume the formula holds for numbers less than $n$. We will consider the only non-trivial case $n=\llcorner Ap^{k+1}(Z_j,Z_i)\lrcorner$. Suppose $\langle n,u^k\rangle\in Fnt(\alpha,\bar{f})$. Then by Lemma \ref{lemma:frc5}.10. there exist $w^{k+1}, v^0$ such that
\[\langle j,w\rangle\in Fnt(\alpha,\bar{f})
\wedge 
\langle i,v\rangle\in Fnt(\alpha,\bar{f})
\wedge
\langle\langle\bar{\alpha}(k+1),v\rangle,u\rangle\in w.\]

By the inductive assumption, $w\in a_{k+1}$ and by the definition of $a_{k+1}$, $u\in a_k$.

2. Proof is by induction on $n$. The cases of constants and variables follow from Lemma \ref{lemma:frc5}.1-6. 

Assume the formula holds for numbers less than $n$. We will consider the only non-trivial case $n=\llcorner Ap^{k+1}(Z_j,Z_i)\lrcorner$. 

\smallskip
Suppose $\langle n,u^k\rangle\in Fnt(\alpha,\bar{f})$ and $\langle n,v^k\rangle\in Fnt(\alpha,\bar{f})$. By Lemma \ref{lemma:frc5}.10, there exist $w_1^{k+1},r_1^0$ and $w_2^{k+1},r_2^0$ such that 
\begin{multline*}
\langle j,w_1\rangle\in Fnt(\alpha,\bar{f})\wedge\langle i,r_1\rangle\in Fnt(\alpha,\bar{f})\wedge\langle\langle \bar{\alpha}(k+1),r_1\rangle,u\rangle\in w_1
\\
\wedge
\langle j,w_2\rangle\in Fnt(\alpha,\bar{f})\wedge\langle i,r_2\rangle\in Fnt(\alpha,\bar{f})\wedge\langle\langle \bar{\alpha}(k+1),r_2\rangle,v\rangle\in w_2.
\end{multline*}
By the inductive assumption, $w_1=w_2$ and $r_1=r_2$, so $\langle \bar{\alpha}(k+1),r_1\rangle,v\rangle\in w_1$. Since $w_1$ is a function, we get $u=v$. 

3. Assume the premises. It is sufficient to prove that for any $n$ with $Fnl_k(n)$:
\begin{equation}
\langle n,u^k\rangle\in Fnt(\alpha,\bar{f})\supset\langle n,u\rangle\in  Fnt(\beta,\bar{f}), k=0,1,\ldots,s.
\label{eq:frc2}
\end{equation}

We prove (\ref{eq:frc2}) by induction on $n$. The cases of constants and variables follow from Lemma \ref{lemma:frc5}.1-6.  

Assume (\ref{eq:frc2}) holds for numbers less than $n$. We consider the only non-trivial case $n=\llcorner Ap^{k+1}(Z_j,Z_i)\lrcorner$. Suppose $\langle n,u^k\rangle\in Fnt(\alpha,\bar{f})$. By Lemma \ref{lemma:frc5}.10, there exist $w^{k+1}$ and $v^0$ such that 
\[\langle j,w\rangle\in Fnt(\alpha,\bar{f})\wedge\langle i,v\rangle\in Fnt(\alpha,\bar{f})\wedge\langle\langle\bar{\alpha}(k+1),v\rangle,u\rangle\in w.\]

By the inductive assumption, $\langle j,w\rangle\in Fnt(\beta,\bar{f})$ and $\langle i,v\rangle\in Fnt(\beta,\bar{f})$. By part 1, $w\in a_{k+1}$ and is monotonic. Since $\beta\preceq\alpha$, we have:
\[\bar{\beta}(k+1)\preceq_k\bar{\alpha}(k+1)\text{ and }\langle\langle\bar{\beta}(k+1),v\rangle,u\rangle\in w,\] 
which implies $\langle n,u\rangle\in Fnt(\beta,\bar{f})$.

4. Proof is by induction on $n$. The cases of constants and variables follow from Lemma \ref{lemma:frc5}.1-6 by taking $\beta=\alpha$. 

Assume the formula holds for numbers less than $n$. We will consider the only non-trivial case $n=\llcorner Ap^{k+1}(Z_j,Z_i)\lrcorner$. 
\smallskip

Suppose "$S^{s+1}$ is a path in $M$ through $\alpha$". By the inductive assumption there exist $\beta_1,\beta_2\in S$ and $w^{k+1},v^0$ such that $\langle j,w\rangle\in Fnt(\beta_1,\bar{f})$ and $\langle i,v\rangle\in Fnt(\beta_2,\bar{f})$. By part 1, $w\in a_{k+1}$. 

Denote $S_1=\lbrace\bar{\gamma}(k+1)\mid\gamma\in S\rbrace$. Then $S_1$ is a path in $d_k$ and by completeness of $w$ there exist $c\in S_1$ and $u^k$ such that $\langle\langle c,v\rangle,u\rangle\in w$. There is $\gamma\in S$ with $\bar{\gamma}(k+1)=c$.

Denote $\beta=min\lbrace\gamma,\beta_1,\beta_2\rbrace$. Then
\[\beta\preceq\gamma\text{ and }\bar{\beta}(k+1)\preceq_k\bar{\gamma}(k+1),\] 
and by monotonicity of $w$,
$\langle\langle\bar{\beta}(k+1),v\rangle,u\rangle\in w$. 

By part 3, $\langle j,w\rangle\in Fnt(\beta,\bar{f})$ and $\langle i,v\rangle\in Fnt(\beta,\bar{f})$. By Lemma \ref{lemma:frc5}.10, this implies $\langle n,u\rangle\in Fnt(\beta,\bar{f})$.
\end{proof}

We denote $Frm(n)$ the arithmetic formula stating that $n$ is the G\"{o}del number of a formula of $SLP_s$. We denote $Subfrm(m,n)$ the arithmetic formula stating that $m$ is the G\"{o}del number of a sub-formula of the formula with G\"{o}del number $n$. Next we introduce a predicate $Forset(n,b^{s+1})$, which means that $b$ is the set of forcing values of all evaluated sub-formulas of $\varphi_n$ at all nodes in $M$.
\begin{multline*}
Forset(n,b^{s+1})\leftrightharpoons Frm(n)\wedge(\forall x^s\in b)(\exists\alpha\in M)\exists m,\bar{f},\delta\left[x=\langle\alpha,m,\bar{f},\delta\rangle
\right.
\\
\left.
\wedge Subfrm(m,n)\wedge Evl(\bar{f})\wedge(\delta=0\vee\delta=1)\right]
\wedge\forall m\left\lbrace Subfrm(m,n)
\right.
\\
\left.
\supset(\forall\alpha\in M)\forall\bar{f}\left[ Evl(\bar{f}) 
\supset(\langle\alpha,m,\bar{f},0\rangle\in b\vee \langle\alpha,m,\bar{f},1\rangle\in b)
\right.
\right.
\\
\left.
\left.
\wedge \neg(\langle\alpha,m,\bar{f},0\rangle\in b\wedge \langle\alpha,m,\bar{f},1\rangle\in b)\wedge 
D(\alpha,m,\bar{f},b)\right] \right\rbrace ,
\end{multline*}
\[\text{ where } D(\alpha,m,\bar{f},b)\leftrightharpoons \bigvee_{r=1}^{13}D_r, 
\text{ and }D_1 - D_{13}\text{ are defined as follows.}\]
\[D_1\leftrightharpoons m=\llcorner\bot\lrcorner \wedge \langle \alpha,m,\bar{f},0\rangle \in b.\]
\begin{multline*}
D_2\leftrightharpoons \bigvee_{k=0}^s \exists i,j \left\lbrace m=\llcorner Z_i=_k Z_j\lrcorner \wedge 
\left[\langle\alpha, m,\bar{f},1\rangle \in b 
\right.
\right.
\\
\left.
\left.
\equiv 
\forall S^{s+1}\left["S\textit{ is a path in }M\textit{ through }\alpha" 
\right.
\right.
\right.
\\
\left.
\left.
\left.
\supset(\exists\beta\in S)\exists u^k(\langle i,u\rangle\in Fnt(\beta,\bar{f})\wedge\langle j,u\rangle\in Fnt(\beta,\bar{f}))\right] 
\right]\right\rbrace. 
\end{multline*}
\begin{multline*}
D_3\leftrightharpoons \exists i,j\left\lbrace m=\llcorner \varphi_i \wedge \varphi_j\lrcorner 
\right.
\\
\left.
\wedge \left[\langle\alpha, m,\bar{f},1\rangle \in b \equiv \left(\langle \alpha,i,\bar{f},1\rangle \in b \wedge \langle\alpha, j,\bar{f},1\rangle \in b \right)  \right] 
\right\rbrace. 
\end{multline*}
\begin{multline*}
D_4\leftrightharpoons \exists i,j\left\lbrace m=\llcorner \varphi_i \supset \varphi_j\lrcorner 
\right.
\\
\left.
\wedge \left[\langle\alpha, m,\bar{f},1\rangle \in b \equiv 
(\forall\beta\in M)(\beta\preceq\alpha\wedge\langle\beta, i,\bar{f},1\rangle \in b\supset\langle\beta, j,\bar{f},1\rangle \in b)\right] 
\right\rbrace. 
\end{multline*}
\begin{multline*}
D_5\leftrightharpoons \exists i,j\left\lbrace m=\llcorner \varphi_i \vee \varphi_j\lrcorner 
\wedge \left[\langle\alpha, m,\bar{f},1\rangle \in b  
\right.
\right.
\\
\left.
\left.
\equiv 
\forall S^{s+1}\left("S\textit{ is a path in }M\textit{ through }\alpha" 
\right.
\right.
\right.
\\
\left.
\left.
\left.
\supset(\exists\beta\in S)\left(\langle \beta,i,\bar{f},1\rangle \in b\vee\langle \beta,j,\bar{f},1\rangle \in b\right) \right)
\right]\right\rbrace. 
\end{multline*}
\begin{multline*}
D_6\leftrightharpoons \exists i,j\left\lbrace m=\llcorner\forall x_i \varphi_j\lrcorner 
\right.
\\
\left.
\wedge \left[\langle\alpha, m,\bar{f},1\rangle \in b \equiv 
\forall u^0
\left(\langle \alpha,j,
Sub_0(\bar{f},i,u),1\rangle \in b \right)  \right] 
\right\rbrace. 
\end{multline*}
\begin{multline*}
D_7\leftrightharpoons \bigvee_{k=1}^s
\exists i,j\left\lbrace m=\llcorner\forall F_i^k \varphi_j\lrcorner 
\right.
\\
\left.
\wedge \left[\langle\alpha, m,\bar{f},1\rangle \in b \equiv 
(\forall u^k\in a_k)
\left(\langle \alpha,j,
Sub_k(\bar{f},3i,u),1\rangle \in b \right)  \right] 
\right\rbrace. 
\end{multline*}
\begin{multline*}
D_8\leftrightharpoons \bigvee_{k=1}^s
\exists i,j\left\lbrace m=\llcorner\forall A_i^k \varphi_j\lrcorner 
\right.
\\
\left.
\wedge \left[\langle\alpha, m,\bar{f},1\rangle \in b \equiv 
(\forall u^k\in b_k)
\left(\langle \alpha,j,
Sub_k(\bar{f},3i+1,u),1\rangle \in b \right)  \right] 
\right\rbrace. 
\end{multline*}
\begin{multline*}
D_9\leftrightharpoons \bigvee_{k=1}^s
\exists i,j\left\lbrace m=\llcorner\forall \mathcal{F}_i^k \varphi_j\lrcorner 
\right.
\\
\left.
\wedge \left[\langle\alpha, m,\bar{f},1\rangle \in b \equiv 
(\forall u^k\in l_k)
\left(\langle \alpha,j,
Sub_k(\bar{f},3i+2,u),1\rangle \in b \right)  \right] 
\right\rbrace. 
\end{multline*}
\begin{multline*}
D_{10}\leftrightharpoons \exists i,j\left\lbrace m=\llcorner\exists x_i \varphi_j\lrcorner 
\right.
\\
\left.
\wedge \left[\langle\alpha, m,\bar{f},1\rangle \in b \equiv 
\forall S^{s+1}\left("S\textit{ is a path in }M\textit{ through }\alpha" 
\right.
\right.
\right.
\\
\left.
\left.
\left.
\supset(\exists\beta\in S)\exists u^0\left(\langle \beta,j,Sub_0(\bar{f},i,u),1\rangle \in b \right)
\right) 
\right]\right\rbrace. 
\end{multline*}
\begin{multline*}
D_{11}\leftrightharpoons \bigvee_{k=1}^s\exists i,j\left\lbrace m=\llcorner\exists F_i^k \varphi_j\lrcorner 
\right.
\\
\left.
\wedge \left[\langle\alpha, m,\bar{f},1\rangle \in b \equiv 
\forall S^{s+1}\left("S\textit{ is a path in }M\textit{ through }\alpha" 
\right.
\right.
\right.
\\
\left.
\left.
\left.
\supset(\exists\beta\in S)(\exists u^k\in a_k)\left(\langle \beta,j,Sub_k(\bar{f},3i,u),1\rangle \in b \right)
\right) 
\right]\right\rbrace. 
\end{multline*}
\begin{multline*}
D_{12}\leftrightharpoons \bigvee_{k=1}^s\exists i,j\left\lbrace m=\llcorner\exists A_i^k \varphi_j\lrcorner 
\right.
\\
\left.
\wedge \left[\langle\alpha, m,\bar{f},1\rangle \in b \equiv 
\forall S^{s+1}\left("S\textit{ is a path in }M\textit{ through }\alpha" 
\right.
\right.
\right.
\\
\left.
\left.
\left.
\supset(\exists\beta\in S)(\exists u^k\in b_k)\left(\langle \beta,j,Sub_k(\bar{f},3i+1,u),1\rangle \in b \right)
\right) 
\right]\right\rbrace. 
\end{multline*}\begin{multline*}
D_{13}\leftrightharpoons \bigvee_{k=1}^s\exists i,j\left\lbrace m=\llcorner\exists \mathcal{F}_i^k \varphi_j\lrcorner 
\right.
\\
\left.
\wedge \left[\langle\alpha, m,\bar{f},1\rangle \in b \equiv 
\forall S^{s+1}\left("S\textit{ is a path in }M\textit{ through }\alpha" 
\right.
\right.
\right.
\\
\left.
\left.
\left.
\supset(\exists\beta\in S)(\exists u^k\in l_k)\left(\langle\beta,j,Sub_k(\bar{f},3i+2,u),1\rangle \in b \right)
\right) 
\right]\right\rbrace. 
\end{multline*}

Denote
\[b\mid_m=\lbrace\langle\alpha,j,\bar{f},\delta\rangle\in b\mid Subfrm(j,m)\rbrace.\]
\begin{lemma}
The following formulas are derived in $TI_{s+1}$.
\begin{enumerate}
\item $Subfrm(m,n)\wedge Forset(n, b^{s+1})\supset Forset(m, b\mid_m)$.
\medskip
\item $Forset(n, b^{s+1})\wedge Forset(n, c^{s+1})\supset b=c$.
\medskip
\item $Frm(n)\supset\exists !b^{s+1} Forset(n,b)$.
\end{enumerate}
\label{lemma:force5}
\end{lemma}
\begin{proof}
1. Assume the premises. For any $j$ with $Subfrm(j,m)$ we have:
\begin{equation}
(\forall\alpha\in M)\forall\bar{f}\left[Evl(\bar{f}\supset D(\alpha,j,\bar{f},b)) \right] 
\label{eq:force02}
\end{equation}
and 
\begin{equation}
\forall\beta,\bar{f},\delta\left[\langle\beta,j,\bar{f},\delta\rangle\in b\equiv\langle\beta,j,\bar{f},\delta\rangle\in b\mid_m \right]. 
\label{eq:force03}
\end{equation}

To prove $Forset(m,b\mid_m )$, it is sufficient to show for any $j$:
\begin{equation}
Subfrm(j,m)\supset(\forall\alpha\in M)\forall\bar{f}\left[Evl(\bar{f})\supset D(\alpha,j,\bar{f},b\mid_m)\right]. 
\label{eq:force04}
\end{equation}

When $\varphi_j$ is an atomic formula or $\bot$, (\ref{eq:force04}) follows from (\ref{eq:force02}). 

Suppose $j=\llcorner\exists\mathcal{F}_r^k  \varphi_p\lrcorner$. Consider $\alpha\in M$ and $\bar{f}$ with $Evl(\bar{f})$. By (\ref{eq:force02}),
\begin{multline*}
\langle \alpha,j,\bar{f},1\rangle \in b\equiv
\forall S^{s+1}\left["S\textit{ is a path in }M\textit{ through }\alpha" 
\right.
\\
\left.
\supset(\exists\beta\in S)(\exists u^k\in l_k)\left(\langle \beta,p,Sub_k(\bar{f},3r+2,u),1\rangle \in b\right) \right]
\end{multline*}
and by (\ref{eq:force03}):
\begin{multline*}
\langle \alpha,j,\bar{f},1\rangle \in b\mid_m\equiv
\forall S^{s+1}\left["S\textit{ is a path in }M\textit{ through }\alpha" 
\right.
\\
\left.
\supset(\exists\beta\in S)(\exists u^k\in l_k)\left(\langle \beta,p,Sub_k(\bar{f},3r+2,u),1\rangle \in b\mid_m\right)\right],
\end{multline*}
which implies $D(\alpha,j,\bar{f},b\mid_m)$. 

In the cases of other quantifiers and logical connectives the proof is similar.

2. Assume the premises. It is sufficient to prove:
\begin{equation}
\forall m\left[Subfrm(m,n)\supset(\langle\alpha,m,\bar{f},1\rangle\in b\equiv\langle\alpha,m,\bar{f},1\rangle\in c)\right]. 
\label{eq:force6}
\end{equation}

If (\ref{eq:force6}) is proven, then for any $\alpha,m,\bar{f}$ we have:
\[\langle\alpha,m,\bar{f},0\rangle\in b\equiv\langle\alpha,m,\bar{f},1\rangle\notin b\equiv\langle\alpha,m,\bar{f},1\rangle\notin c\equiv\langle\alpha,m,\bar{f},0\rangle\in c,\]
and this implies $b=c$.

We prove (\ref{eq:force6}) by induction on $m$.

Case 1: $m=\llcorner\bot\lrcorner$. Then $\langle\alpha,m,\bar{f},1\rangle\notin b$ and $\langle\alpha,m,\bar{f},1\rangle\notin c$.

Case 2: $m=\llcorner Z_i=_k Z_j\lrcorner$. Proof follows from the definition of $D_2$.

Assume (\ref{eq:force6}) holds for numbers less than $m$.

Case 3: $m=\llcorner\exists\mathcal{F}_i^k  \varphi_j\lrcorner$.
By the inductive assumption: 
\[\forall\beta,\bar{f}(\langle\beta,j,\bar{f},1\rangle\in b\equiv\langle\beta,j,\bar{f},1\rangle\in c).\]

So for any $\alpha,\bar{f}$:
\begin{multline*}
\langle\alpha,m,\bar{f},1\rangle\in b
\equiv
\forall S^{s+1}\left["S\textit{ is a path in }M
\textit{through }\alpha" 
\right.
\\
\left.
\supset(\exists\beta\in S)(\exists u^k\in l_k)
(\langle\beta, j,Sub_k(\bar{f},3i+2,u),1\rangle\in b)
\right]
\\
\equiv
\forall S^{s+1}\left["S\textit{ is a path in }M
\textit{through }\alpha" 
\right.
\\
\left.
\supset(\exists\beta\in S)(\exists u^k\in l_k)
(\langle\beta, j,Sub_k(\bar{f},3i+2,u),1\rangle\in c)
\right]
\equiv\langle\alpha,m,\bar{f},1\rangle\in c.
\end{multline*}

The cases of other quantifiers and logical connectives are similar.

3. The uniqueness was proven in part 2. We will prove the existence by induction on $n$.

Case 1: $n=\llcorner\bot\lrcorner$. We take $b^{s+1}=\lbrace\langle\alpha, n,\bar{f},0\rangle\mid\alpha\in M\wedge Evl(\bar{f})\rbrace.$

Case 2: $n=\llcorner Z_i=_kZ_j\lrcorner$. Denote 
\begin{multline*}
\varphi(\alpha, i,j,\bar{f})\leftrightharpoons\forall S^{s+1}\left["S\textit{ is a path in }M
\textit{through }\alpha" 
\right.
\\
\left.
\supset(\exists\beta\in S)\exists u^k(\langle i,u\rangle\in Fnt(\beta,\bar{f})\wedge\langle j,u\rangle\in Fnt(\beta,\bar{f}))\right]. 
\end{multline*}

We take 
\begin{multline*}
b^{s+1}=\lbrace\langle\alpha, n,\bar{f},1\rangle\mid\alpha\in M\wedge Evl(\bar{f})\wedge\varphi(\alpha, i,j,\bar{f})\rbrace
\\
\cup
\lbrace\langle\alpha, n,\bar{f},0\rangle\mid\alpha\in M\wedge Evl(\bar{f})\wedge\neg\varphi(\alpha, i,j,\bar{f})\rbrace.
\end{multline*}

Clearly $Forset(n,b)$ holds.

Assume the existence for numbers less than $n$.

Case 3: $n=\llcorner\varphi_i\vee \varphi_j\lrcorner$. By the inductive assumption there exist sets $c^{s+1}$ and $d^{s+1}$ such that $Forset(i,c)$ and $Forset(j,d)$. Denote 
\begin{multline*}
\psi(\alpha, i,j,\bar{f})\leftrightharpoons\forall S^{s+1}\left["S\textit{ is a path in }M
\textit{through }\alpha" 
\right.
\\
\left.
\supset(\exists\beta\in S)(\langle \beta,i,\bar{f},1\rangle\in c\vee\langle \beta,j,\bar{f},1\rangle\in d)\right]. 
\end{multline*}

We take 
\begin{multline*}
b^{s+1}=c\cup d\cup\lbrace\langle\alpha, n,\bar{f},1\rangle\mid\alpha\in M\wedge Evl(\bar{f})\wedge\psi(\alpha, i,j,\bar{f})\rbrace
\\
\cup
\lbrace\langle\alpha, n,\bar{f},0\rangle\mid\alpha\in M\wedge Evl(\bar{f})\wedge\neg\psi(\alpha, i,j,\bar{f})\rbrace.
\end{multline*}

First we prove:
\begin{equation}
\forall p\left[Subfrm(p,i)\supset \forall\beta,\bar{f},\delta\left(\langle\beta, p,\bar{f},\delta\rangle\in b\equiv\langle\beta, p,\bar{f},\delta\rangle\in c \right) \right].
\label{eq:forcef}
\end{equation}

Assume the premises of (\ref{eq:forcef}). If $\neg Subfrm(p,j)$, then the conclusion is obvious. If $Subfrm(p,j)$, then $Forset(p,c\mid_p)$ and $Forset(p,d\mid_p)$ by part 1, $c\mid_p=d\mid_p$ by part 2 and for any $\beta,\bar{f},\delta$:
\[\langle\beta, p,\bar{f},\delta\rangle\in b\equiv
\langle\beta, p,\bar{f},\delta\rangle\in c\mid_p\cup d\mid_p\equiv\langle\beta, p,\bar{f},\delta\rangle\in c.\]

(\ref{eq:forcef}) is proven. Similarly,
\begin{equation}
\forall p\left[Subfrm(p,j)\supset \forall\beta,\bar{f},\delta\left(\langle\beta, p,\bar{f},\delta\rangle\in b\equiv\langle\beta, p,\bar{f},\delta\rangle\in d \right) \right].
\label{eq:forceg}
\end{equation}

To prove $Forset(n,b)$, it is sufficient to show:
\begin{multline}
\forall m\left\lbrace Subfrm(m,n)\supset(\forall\alpha\in M)\forall\bar{f}\left[Evl(\bar{f})
\right.
\right.
\\
\left.
\left.
\supset
\neg(\langle\alpha, m,\bar{f},0\rangle\in b\wedge\langle\alpha, m,\bar{f},1\rangle\in b)\wedge D(\alpha,m,\bar{f},b)\right]\right\rbrace .
\label{eq:forcea}
\end{multline}

Assume $Subfrm(m,n)$.

i) $m=n$.

Suppose $\alpha\in M$ and $Evl(\bar{f})$. Then $\neg(\langle\alpha, n,\bar{f},0\rangle\in b\wedge\langle\alpha, n,\bar{f},1\rangle\in b)$ follows from the definition of $b$.

Applying (\ref{eq:forcef}) to $p=i$ and (\ref{eq:forceg}) to $p=j$, we get:
\begin{multline*}
\langle\alpha,n,\bar{f},1\rangle\in b\equiv\psi(\alpha, i,j,\bar{f})
\equiv\forall S^{s+1}\left["S\textit{ is a path in }M
\textit{through }\alpha" 
\right.
\\
\left.
\supset(\exists\beta\in S)(\langle \beta,i,\bar{f},1\rangle\in c\vee\langle \beta,j,\bar{f},1\rangle\in d)\right]
\\
\equiv\forall S^{s+1}\left["S\textit{ is a path in }M
\textit{through }\alpha" 
\right.
\\
\left.
\supset(\exists\beta\in S)(\langle \beta,i,\bar{f},1\rangle\in b\vee\langle \beta,j,\bar{f},1\rangle\in b)\right],
\end{multline*}
which implies $D(\alpha,n,\bar{f},b)$.

ii) $m<n$.

Then $Subfrm(m,i)\vee Subfrm(m,j)$. We can assume $Subfrm(m,i)$ without loss of generality. Due to $Forset(i,c)$ we have $\neg(\langle\alpha, m,\bar{f},0\rangle\in c\wedge\langle\alpha, m,\bar{f},1\rangle\in c)$ and by (\ref{eq:forcef}), 
\[\neg(\langle\alpha, m,\bar{f},0\rangle\in b\wedge\langle\alpha, m,\bar{f},1\rangle\in b).\] 
 
$D(\alpha,m,\bar{f},b)$ is also also derived from $Forset (i,c)$ by applying (\ref{eq:forcef}) to each possible form of the formula $\varphi_m$.

This completes the proof for case 3. The cases of other logical connectives and quantifiers are similar.
\end{proof}
\begin{lemma}
It is derived in $TI_{s+1}$ that the formula 
\[\alpha\in M\wedge Subfrm(m,n)\wedge Evl(\bar{f})\wedge Forset(n,b^{s+1})\]
implies each of the following formulas.
\begin{enumerate}

\item $\langle\alpha, m,\bar{f},1 \rangle\in b \vee \langle \alpha,m,\bar{f},0 \rangle\in b$.
\medskip

\item $\neg (\langle\alpha, m,\bar{f},1 \rangle\in b \wedge \langle \alpha,m,\bar{f},0 \rangle \in b)$.
\medskip

\item $\langle\alpha, \llcorner \bot \lrcorner,\bar{f},0 \rangle \in b$.
\medskip

\item $ m=\llcorner Z_i=_k Z_j\lrcorner 
\supset 
\left\lbrace\langle\alpha,m,\bar{f},1\rangle\in b \equiv 
\forall S^{s+1}\left["S\textit{ is a path in }M
\right.
\right.
\medskip
\\
\left.
\left.
\textit{through }\alpha" 
\supset(\exists\beta\in S)\exists u^k
(\langle i,u\rangle\in Fnt(\beta,\bar{f})\wedge\langle j,u\rangle\in Fnt(\beta,\bar{f}))
\right]\right\rbrace,
\\k=0,1,\ldots,s  $.
\medskip

\item $m=\llcorner \varphi_i \wedge \varphi_j\lrcorner 
\\
\supset \left[\langle\alpha, m,\bar{f},1\rangle \in b \equiv \left(\langle \alpha,i,\bar{f},1\rangle \in b \wedge \langle\alpha, j,\bar{f},1\rangle \in b \right)\right]$.
\medskip

\item $m=\llcorner \varphi_i \supset \varphi_j\lrcorner 
\\
\supset \left[\langle\alpha, m,\bar{f},1\rangle \in b \equiv 
(\forall\beta\in M)(\beta\preceq\alpha\wedge\langle\beta, i,\bar{f},1\rangle \in b\supset\langle\beta, j,\bar{f},1\rangle \in b)\right] $.
\medskip

\item $m=\llcorner \varphi_i \vee \varphi_j\lrcorner 
\supset \left[\langle\alpha, m,\bar{f},1\rangle \in b  
\equiv 
\forall S^{s+1}\left("S\textit{ is a path in }M\textit{ through }\alpha" 
\right.
\right.
\medskip
\\
\left.
\left.
\supset(\exists\beta\in S)\left(\langle \beta,i,\bar{f},1\rangle \in b\vee\langle \beta,j,\bar{f},1\rangle \in b\right) \right)
\right]$.
\medskip

\item $m=\llcorner\forall x_i \varphi_j\lrcorner 
\supset \left[\langle\alpha, m,\bar{f},1\rangle \in b \equiv \forall u^0\left(\langle \alpha,j,
Sub_0(\bar{f},i,u),1\rangle \in b \right)\right]$.
\medskip

\item $m=\llcorner\forall F_i^k \varphi_j\lrcorner 
\supset \left[\langle\alpha, m,\bar{f},1\rangle \in b 
\right.
\smallskip
\\
\left.
\equiv(\forall u^k\in a_k)\left(\langle \alpha,j,
Sub_k(\bar{f},3i,u),1\rangle \in b \right)\right],
k=1,2,\ldots,s$.
\medskip

\item $m=\llcorner\forall A_i^k \varphi_j\lrcorner 
\supset \left[\langle\alpha, m,\bar{f},1\rangle \in b 
\right.
\smallskip
\\
\left.
\equiv(\forall u^k\in b_k)\left(\langle \alpha,j,
Sub_k(\bar{f},3i+1,u),1\rangle \in b \right)\right],
k=1,2,\ldots,s$.
\medskip

\item $m=\llcorner\forall \mathcal{F}_i^k \varphi_j\lrcorner 
\supset \left[\langle\alpha, m,\bar{f},1\rangle \in b 
\right.
\smallskip
\\
\left.
\equiv(\forall u^k\in l_k)\left(\langle \alpha,j,
Sub_k(\bar{f},3i+2,u),1\rangle \in b \right)\right],
k=1,2,\ldots,s$.
\medskip

\item $m=\llcorner\exists x_i \varphi_j\lrcorner 
\supset \left\lbrace \langle\alpha, m,\bar{f},1\rangle \in b \equiv 
\forall S^{s+1}\left["S\textit{ is a path in }M\textit{ through }\alpha" 
\right.
\right.
\smallskip
\\
\left.
\left.
\supset(\exists\beta\in S)\exists u^0\left(\langle \beta,j,Sub_0(\bar{f},i,u),1\rangle \in b \right)
\right] \right\rbrace$.
\medskip

\item $m=\llcorner\exists F_i^k \varphi_j\lrcorner 
\supset \left\lbrace \langle\alpha, m,\bar{f},1\rangle \in b \equiv 
\forall S^{s+1}\left["S\textit{ is a path in }M\textit{ through }\alpha" 
\right.
\right.
\smallskip
\\
\left.
\left.
\supset(\exists\beta\in S)(\exists u^k\in a_k)\left(\langle\beta,j,Sub_k(\bar{f},3i,u),1\rangle \in b \right)\right] \right\rbrace,k=1,2,\ldots,s$.
\medskip

\item $m=\llcorner\exists A_i^k \varphi_j\lrcorner 
\supset \left\lbrace \langle\alpha, m,\bar{f},1\rangle \in b \equiv 
\forall S^{s+1}\left["S\textit{ is a path in }M\textit{ through }\alpha" 
\right.
\right.
\smallskip
\\
\left.
\left.
\supset(\exists\beta\in S)(\exists u^k\in b_k)\left(\langle\beta,j,Sub_k(\bar{f},3i+1,u),1\rangle \in b \right)\right] \right\rbrace,k=1,2,\ldots,s$.
\medskip

\item $m=\llcorner\exists \mathcal{F}_i^k \varphi_j\lrcorner 
\supset \left\lbrace \langle\alpha, m,\bar{f},1\rangle \in b \equiv 
\forall S^{s+1}\left["S\textit{ is a path in }M\textit{ through }\alpha" 
\right.
\right.
\smallskip
\\
\left.
\left.
\supset(\exists\beta\in S)(\exists u^k\in l_k)\left(\langle\beta,j,Sub_k(\bar{f},3i+2,u),1\rangle \in b \right)\right] \right\rbrace,k=1,2,\ldots,s$.
\end{enumerate}
\label{lemma:force6}
\end{lemma}
\begin{proof}
Follows from the definition of $Forset$.
\end{proof}

We define
\[Force^{s+1}=\left\lbrace x^s\mid \exists n,b^{s+1}(Forset(n,b)\wedge x\in b)\right\rbrace. \]
Thus, $Force$ is the set of the forcing values of all evaluated formulas of $SLP_s$ at all nodes of the Beth model $\mathcal{B}_s$.

\begin{lemma}
\begin{multline*}
TI_{s+1}\vdash Forset(n,b^{s+1})\wedge Subfrm(m,n)
\\
\supset\left(\langle\alpha, m,\bar{f},\delta\rangle\in b\equiv \langle \alpha,m,\bar{f},\delta\rangle\in Force\right).
\end{multline*}
\label{lemma:force7}
\end{lemma}
\begin{proof}
Assume the premises. Then $Forset(m,b\mid_m)$ by Lemma \ref{lemma:force5}.1.

$\Rightarrow$ If $\langle\alpha, m,\bar{f},\delta\rangle\in b$, then $\langle\alpha, m,\bar{f},\delta\rangle\in b\mid_m$ and $\langle\alpha, m,\bar{f},\delta\rangle\in Force$ by the definition of the set $Force$.

$\Leftarrow$ Suppose $\langle\alpha, m,\bar{f},\delta\rangle\in Force$. Then there is $c^{s+1}$ with 
$Forset(m,c)$ and $\langle\alpha, m,\bar{f},\delta\rangle\in c$. By Lemma \ref{lemma:force5}.2, $c=b\mid_m$, so $\langle\alpha, m,\bar{f},\delta\rangle\in b\mid_m\subseteq b$.
\end{proof}

Finally we introduce the forcing predicate $Fr$: \[Fr(\alpha,n,\bar{f})\leftrightharpoons \langle\alpha,n,\bar{f},1\rangle\in Force.\]
Clearly, $Fr$ depends on $s$ but we will omit $s$ for brevity.

\begin{lemma}
It is derived in $TI_{s+1}$ that the formula
\[\alpha\in M\wedge Evl(\bar{f})\]
implies each of the following formulas.
\begin{enumerate}
\item $Frm(n)\supset 
\left( \neg Fr(\alpha,n,\bar{f})\equiv \langle\alpha, n,\bar{f},0 \rangle \in Force\right) $.
\smallskip
\item $Fr(\alpha,\llcorner \perp \lrcorner,\bar{f})\equiv \perp$.
\medskip

\item 
$Fr(\alpha,\llcorner Z_i=_k Z_j \lrcorner,\bar{f}) \equiv 
\forall S^{s+1}\left["S\textit{ is a path in }M\textit{ through }\alpha" 
\right.
\medskip
\\
\left.
\supset(\exists\beta\in S)\exists u^k(\langle i,u\rangle\in Fnt(\beta,\bar{f})\wedge\langle j,u\rangle\in Fnt(\beta,\bar{f}))\right],
k=0,1,\ldots,s$.
\medskip

\item $Fr(\alpha,\llcorner \varphi_i \wedge \varphi_j \lrcorner,\bar{f}) \equiv Fr(\alpha,i,\bar{f})\wedge Fr(\alpha,j,\bar{f})$.
\medskip

\item $Fr(\alpha,\llcorner \varphi_i \supset \varphi_j \lrcorner,\bar{f}) \equiv 
(\forall\beta\in M)\left[\beta\preceq\alpha\wedge Fr(\alpha,i,\bar{f})\supset Fr(\alpha,j,\bar{f})\right]$.
\medskip

\item $Fr(\alpha,\llcorner \varphi_i \vee \varphi_j \lrcorner,\bar{f}) \equiv 
\forall S^{s+1}\left["S\textit{ is a path in }M\textit{ through }\alpha" 
\right.
\medskip
\\
\left.
\supset(\exists\beta\in S)\left(Fr(\beta,i,\bar{f})\vee Fr(\beta,j,\bar{f})\right) \right]$.
\medskip

\item $Fr(\alpha,\llcorner \forall x_i \varphi_j \lrcorner,\bar{f}) \equiv \forall u^0 Fr(\alpha,j,Sub_0(\bar{f},i,u))$.
\medskip

\item $Fr(\alpha,\llcorner \forall F_i^k \varphi_j \lrcorner,\bar{f}) \equiv (\forall u^k\in a_k) Fr(\alpha,j,Sub_k(\bar{f},3i,u)),k=1,2,\ldots,s$.
\medskip

\item $Fr(\alpha,\llcorner \forall A_i^k \varphi_j \lrcorner,\bar{f}) \equiv (\forall u^k\in b_k) Fr(\alpha,j,Sub_k(\bar{f},3i+1,u)),k=1,2,\ldots,s$.
\medskip

\item $Fr(\alpha,\llcorner \forall \mathcal{F}_i^k \varphi_j \lrcorner,\bar{f}) \equiv (\forall u^k\in l_k) Fr(\alpha,j,Sub_k(\bar{f},3i+2,u)),k=1,2,\ldots,s$.
\medskip

\item $Fr(\alpha,\llcorner \exists x_i \varphi_j \lrcorner,\bar{f}) \equiv 
\forall S^{s+1}\left["S\textit{ is a path in }M\textit{ through }\alpha" 
\right.
\medskip
\\
\left.
\supset(\exists\beta\in S)\exists u^0 Fr(\beta,j,Sub_0(\bar{f},i,u))\right] $.
\medskip

\item $Fr(\alpha,\llcorner \exists F_i^k \varphi_j \lrcorner,\bar{f}) \equiv 
\forall S^{s+1}\left["S\textit{ is a path in }M\textit{ through }\alpha" 
\right.
\medskip
\\
\left.
\supset(\exists\beta\in S)(\exists u^k\in a_k) Fr(\beta,j,Sub_k(\bar{f},3i,u))\right],
k=1,2,\ldots,s $.
\medskip

\item $Fr(\alpha,\llcorner \exists A_i^k \varphi_j \lrcorner,\bar{f}) \equiv 
\forall S^{s+1}\left["S\textit{ is a path in }M\textit{ through }\alpha" 
\right.
\medskip
\\
\left.
\supset(\exists\beta\in S)(\exists u^k\in b_k) Fr(\beta,j,Sub_k(\bar{f},3i+1,u))\right],
k=1,2,\ldots,s $.
\medskip

\item $Fr(\alpha,\llcorner \exists \mathcal{F}_i^k \varphi_j \lrcorner,\bar{f}) \equiv 
\forall S^{s+1}\left["S\textit{ is a path in }M\textit{ through }\alpha" 
\right.
\medskip
\\
\left.
\supset(\exists\beta\in S)(\exists u^k\in l_k) Fr(\beta,j,Sub_k(\bar{f},3i+2,u))\right],
k=1,2,\ldots,s $.
\end{enumerate}
\label{lemma:force8}
\end{lemma}
\begin{proof}
Assume $\alpha\in M\wedge Evl(\bar{f})$. By Lemma \ref{lemma:force5}.3, for any $n$ with $Frm(n)$ there is a set $b^{s+1}$ such that $Forset(n,b)$ and by Lemma \ref{lemma:force7}:
\begin{equation}
\forall m \left[Subfrm(m,n)\supset\forall\beta, \bar{f}(Fr(\beta,m,\bar{f})\equiv\langle\beta,m,\bar{f},1\rangle\in b) \right].
\label{eq:forced}
\end{equation}

1. By (\ref{eq:forced}) for any $n$ with $Frm(n)$:
\begin{multline*}
\neg Fr(\alpha,n,\bar{f})\equiv\langle\alpha,n,\bar{f},1\rangle\notin b\equiv\langle\alpha,n,\bar{f},0\rangle\in b
\equiv\langle\alpha,n,\bar{f},0\rangle\in Force
\end{multline*}
by Lemma \ref{lemma:force6}.2, 1 and Lemma \ref{lemma:force7}.

2. Denote $n=\llcorner\bot\lrcorner$. 

By  (\ref{eq:forced}) and Lemma \ref{lemma:force6}.3, 2, $Fr(\alpha,n,\bar{f})\equiv\langle\alpha,n,\bar{f},1\rangle\in b\equiv\bot$.

3. Denote $n=\llcorner Z_i=_k Z_j \lrcorner$. By (\ref{eq:forced}) and Lemma \ref{lemma:force6}.4, 
\begin{multline*}
Fr(\alpha,n,\bar{f})\equiv\langle\alpha,n,\bar{f},1\rangle\in b\equiv
\forall S^{s+1}\left["S\textit{ is a path in }M
\textit{through }\alpha" 
\right.
\\
\left.
\supset(\exists\beta\in S)\exists u^k
(\langle i,u\rangle\in Fnt(\beta,\bar{f})\wedge\langle j,u\rangle\in Fnt(\beta,\bar{f}))\right]. 
\end{multline*} 

Parts 4 and 5 are similar to part 6, which we consider next.

6. Denote $n=\llcorner \varphi_i\vee \varphi_j \lrcorner$. By (\ref{eq:forced}):
\begin{eqnarray*}
\forall\beta,\bar{f} \left[Fr(\beta,n,\bar{f})\equiv\langle \beta,n,\bar{f},1\rangle\in b\right] ,
\\
\forall\beta,\bar{f} \left[Fr(\beta,i,\bar{f})\equiv\langle \beta,i,\bar{f},1\rangle\in b\right],
\\
\forall\beta,\bar{f} \left[Fr(\beta,j,\bar{f})\equiv\langle \beta,j,\bar{f},1\rangle\in b\right].
\end{eqnarray*}

By Lemma \ref{lemma:force6}.7,
\begin{multline*}
Fr(\alpha,n,\bar{f})
\equiv \langle \alpha,n,\bar{f},1\rangle \in b\equiv
\forall S^{s+1}\left["S\textit{ is a path in }M\textit{ through }\alpha" 
\right.
\\
\left.
\supset(\exists\beta\in S)\left(\langle \beta,i,\bar{f},1\rangle \in b\vee\langle \beta,j,\bar{f},1\rangle \in b \right]
\right]
\\
\equiv
\forall S^{s+1}\left["S\textit{ is a path in }M\textit{ through }\alpha" 
\right.
\\
\left.
\supset(\exists\beta\in S)\left(Fr(\beta,i,\bar{f})\vee Fr(\beta,j,\bar{f})\right)  \right].
\end{multline*} 

Parts 7-13 are similar to part 14, which we consider next.

14. Denote $n=\llcorner \exists \mathcal{F}_i^k\varphi_j \lrcorner$. By (\ref{eq:forced}):
\begin{eqnarray*}
\forall\beta,\bar{f} \left[Fr(\beta,n,\bar{f})\equiv\langle \beta,n,\bar{f},1\rangle\in b\right] ,
\\
\forall\beta,\bar{f} \left[Fr(\beta,j,\bar{f})\equiv\langle \beta,j,\bar{f},1\rangle\in b\right].
\end{eqnarray*}

By Lemma \ref{lemma:force6}.15,
\begin{multline*}
Fr(\alpha,n,\bar{f})
\equiv \langle \alpha,n,\bar{f},1\rangle \in b\equiv
\forall S^{s+1}\left["S\textit{ is a path in }M\textit{ through }\alpha" 
\right.
\\
\left.
\supset(\exists\beta\in S)(\exists u^k\in l_k)
\left(\langle \beta,j,Sub_k(\bar{f},3i+2,u),1\rangle \in b \right]\right]
\\
\equiv
\forall S^{s+1}\left["S\textit{ is a path in }M\textit{ through }\alpha" 
\right.
\\
\left.
\supset(\exists\beta\in S)(\exists u^k\in l_k)
Fr(\beta,j,Sub_k(\bar{f},3i+2,u))\right].
\end{multline*} 
\end{proof}

\begin{theorem}
$TI_{s+1}\vdash Prf_{SLP_s}(m)\wedge Evl(\bar{f})\supset Fr(\varepsilon,m,\bar{f})$, where $\varepsilon$ was defined in Definition \ref{def:my_beth_model}.
\label{theorem:force1}
\end{theorem}

\begin{proof}
Informal proofs in sections 4, 6, 7, 8 and 9 are easily formalised in $TI_{s+1}$, due to Lemmas \ref{lemma:force8}, \ref{lemma:frc5} and \ref{lemma:frc6}.
\end{proof}

\subsection{The strengths of theories $SLP_s$ and $SLP$}
\begin{theorem}
$TI_{s+1}\vdash Con(SLP_s), s=0,1,2,\ldots$.
\label{theorem:force2}
\end{theorem}
\begin{proof}
Suppose $s\geqslant 1$. If $TI_{s+1}$ is inconsistent, then any formula can be derived in it. We assume that it is consistent. We will describe derivation in $TI_{s+1}$ informally. 

For any $k=1,2,\ldots,s$ we fix trivial elements in the domains $a_k,b_k$ and $l_k$:
\begin{list}{•}{•}
\item $g_{0k}=g_{1k}=\hat{K}^k$;
\smallskip
\item $g_{2k}=\nu_k(\xi)$, where $\xi=\left\lbrace\langle\langle n,y^k
\rangle,y^k\rangle \mid y\in a_{k-1}\right\rbrace$. 
\end{list}

Thus, $g_{0k}\in a_k,g_{1k}\in b_k$ and $g_{2k}\in l_k$. Next we fix a trivial evaluation $\bar{f}$:
\smallskip

$f_0^1=\left\lbrace \langle i,0\rangle\mid i=i
\right\rbrace$; for $k=1,2,\ldots,s$:
\begin{multline*}
f_k^k=\left\lbrace \langle 3i,z^{k-1}\rangle\mid i\in\omega\wedge z\in g_{0k}\right\rbrace\cup
\left\lbrace \langle 3i+1,z^{k-1}\rangle\mid i\in\omega \wedge z\in g_{1k}\right\rbrace
\\
\cup
\left\lbrace \langle 3i+2,z^{k-1}\rangle\mid i\in\omega\wedge z\in g_{2k}\right\rbrace.
\end{multline*}

Then $Evl(\bar{f})$.

Assume $Prf_{SLP_s}(\llcorner \perp \lrcorner)$. By Theorem \ref{theorem:force1}, $Fr(\varepsilon,\llcorner \perp \lrcorner,\bar{f}) $ that contradicts Lemma \ref{lemma:force8}.2. Therefore $\neg Prf_{SLP_s}(\llcorner \perp \lrcorner)$.

Now suppose $s=0$. Clearly $SLP_0$ is the intuitionistic first-order arithmetic $HA$ and $TI_1$ is the classical second-order arithmetic, the same as $T_1$. The theorem follows from the well-known fact that $T_1\vdash Con(HA)$.
\end{proof}

\begin{theorem}
For $s=0,1,2,\ldots$,
\begin{enumerate}
\item $SLP_{s+1}\vdash Con(TI_s)$; 
\item $SLP_{s+1}\vdash Con(SLP_s)$ .
\end{enumerate}
\label{theorem:con4}
\end{theorem}

\begin{proof}
1. By Theorem \ref{theorem:con1}, we have $TI_{s+1}\vdash Con(TI_s)$. By Theorem \ref{theorem:int}.1, we have $SLP_{s+1}\vdash int(Con(TI_s))$. By Lemma \ref{lemma:int}, $SLP_{s+1}\vdash Con(TI_s)$, since $Con(TI_s)\equiv \forall n\neg Proof_{TI_s}(n,\llcorner \bot\lrcorner)$ and $Proof_{TI_s}(n,\llcorner \bot\lrcorner)$ defines a primitive recursive predicate.

2. By Theorem \ref{theorem:force2}, we have $TI_{s+1}\vdash Con(SLP_s)$. By Theorem \ref{theorem:int}.1, $SLP_{s+1}\vdash int(Con(SLP_s))$ and by Lemma \ref{lemma:int}, $SLP_{s+1}\vdash Con(SLP_s)$.
\end{proof}

\begin{theorem}
\begin{enumerate}
\item $SLP_2\vdash Con(I_n)$ for $n=0,1,2,\ldots$.
\item $SLP_3\vdash Con(I)$.
\item $SLP\vdash Con(I)$.
\end{enumerate}
\label{theorem:con5}
\end{theorem}
\begin{proof}
1. By Theorem \ref{theorem:con2}, we have $TI_2\vdash Con(I_n)$. By Theorem \ref{theorem:int}.1, $SLP_2\vdash int(Con(I_n))$ and by Lemma \ref{lemma:int}, $SLP_2\vdash Con(I_n)$.

2. By Theorem \ref{theorem:con3}.1,$TI_3\vdash Con(I)$. By Theorem \ref{theorem:int}.1, $SLP_3\vdash int(Con(I))$ and  $SLP_3\vdash Con(I)$ by Lemma \ref{lemma:int}.

Part 3 follows from part 2.
\end{proof}

The theory $SLP$ is clearly weaker than the fully impredicative theory $T$ but  it has the same strength as the quite impredicative theory $TI$ and is stronger than the more predicative theory $I$. In particular, $SLP$ is stronger than the second order arithmetic.

\section{Discussion}

In this paper we constructed a Beth model $\mathcal{B}_s$ for the language of intuitionistic functionals of types $1, 2, \ldots, s$. We combined the intuitionistic principles that hold in the models $\mathcal{B}_s (s\geqslant 1)$ into a relatively strong intuitionistic theory $SLP$. The theory $SLP$ contains axioms for lawless functionals of high types, some versions of choice axioms and the theory of the "creating subject". 

We did not manage to construct one Beth model for the entire theory $SLP$; combining all the models $\mathcal{B}_s$ in one does not work because their cardinalities increase indefinitely with the increase of $s$. However, the models $\mathcal{B}_s (s\geqslant 1)$ are sufficient, since any proof in $SLP$ is finite and contains only a finite number of types.

We compared $SLP$ with a classical typed set theory $TI$ that has a restricted but impredicative comprehension axiom. By formalising the Beth model $\mathcal{B}_s$  in the language $TI$, we showed the equiconsistency of each fragment $TI_s$ with the corresponding fragment $SLP_s$ and also the equiconsistency of $TI$ and $SLP$. Thus, $SLP$ is stronger than any fragment $TI_s$ and in particular, it is stronger than the second order arithmetic. 

We interpreted the classical theory $TI$ in our intuitionistic theory $SLP$. Many parts of classical mathematics can be developed in $TI$ and therefore justified from the intuitionistic point of view, due to this interpretation.

\bibliographystyle{asl}
\bibliography{Farida}

\end{document}